\documentclass{article}

\usepackage{mathptmx}       
\usepackage{helvet}         
\usepackage{courier}        
\usepackage{type1cm}        
\usepackage{graphicx}        
\usepackage{multicol}        
\usepackage[bottom]{footmisc}

\usepackage{verbatim}        
\usepackage{amsmath,amssymb}
\usepackage{amsthm}
\usepackage{amsfonts}
\usepackage{marginnote}

\usepackage[T1]{fontenc}     
\usepackage{enumitem}        

\usepackage{geometry}
\usepackage{xcolor}

\theoremstyle{plain} 
\newtheorem{theorem}{Theorem}[section]

\newtheorem{lemma}[theorem]{Lemma}
\newtheorem{proposition}[theorem]{Proposition}

\theoremstyle{definition}
\newtheorem{definition}[theorem]{Definition}
\newtheorem{remark}[theorem]{Remark}

\newcommand{\C}{\mathbb{C}}
\newcommand{\R}{\mathbb{R}}
\newcommand{\N}{\mathbb{N}}
\newcommand{\Z}{\mathbb{Z}}

\renewcommand{\d}[1][x]{\,\operatorname{d}\!#1}
\newcommand{\ddt}{\frac{\d[]}{\d[t]}}

\newcommand{\intRd}{\int_{\R^d}}
\newcommand{\Rd}{\R^d}
\newcommand{\eps}{\varepsilon}
\newcommand{\spn}{\operatorname{span}}

\newcommand{\corank}{\operatorname{corank}}

\newcommand{\tr}{\operatorname{Tr}}

\newcommand{\A}{\mathbf{A}}
\newcommand{\B}{\mathbf{B}}
\newcommand{\CC}{\mathbf{C}}
\newcommand{\DD}{\mathbf{D}}
\newcommand{\II}{\mathbf{I}}
\newcommand{\K}{\mathbf{K}}
\newcommand{\LL}{\mathbf{L}}

\newcommand{\Null}{\mathbf{0}}
\renewcommand{\P}{\mathbf{P}}
\newcommand{\Q}{\mathbf{Q}}
\newcommand{\Qlin}{\mathbf{Q_{lin}}}
\newcommand{\RR}{\mathbf{R}}
\newcommand{\Smatrix}{\mathbf{S}}
\newcommand{\U}{\mathbf{U}}

\newcommand{\ip}[2]{\langle {#1}\ ,\, {#2} \rangle}

\newcommand{\norm}[1]{\| {#1} \|}

\DeclareMathOperator*{\argmax}{arg\,max}
\DeclareMathOperator{\diver}{div}
\DeclareMathOperator{\diag}{diag}
\DeclareMathOperator{\rank}{rank}

\newcommand{\T}{\mathbf{T}}

\newcommand{\cK}{\mathcal{K}}
\newcommand{\cE}{\mathcal{E}}

\begin{document}

\title{On multi-dimensional hypocoercive BGK models}

\author{Franz Achleitner\thanks{University of Vienna, Faculty of Mathematics, Oskar-Morgenstern-Platz 1, A-1090 Wien, Austria, franz.achleitner@univie.ac.at},
 Anton Arnold\thanks{Vienna University of Technology, Institute of Analysis and Scientific Computing, Wiedner Hauptstr. 8-10, A-1040 Wien, Austria, anton.arnold@tuwien.ac.at},
 and Eric A.\ Carlen\thanks{Department of Mathematics, Rutgers University, 110 Frelinghuysen Rd., Piscataway NJ 08854, USA, carlen@math.rutgers.edu}
} 


\maketitle

\numberwithin{equation}{section} 

We study hypocoercivity for a class of linearized BGK models for continuous phase spaces. 
We develop methods for constructing entropy functionals that enable us to prove exponential relaxation to equilibrium with explicit and physically meaningful rates. 
In fact, we not only estimate the exponential rate,
 but also the second time scale governing the time one must wait before one begins to see the exponential relaxation in the $L^1$ distance. 
This waiting time phenomenon, with a long plateau before the exponential decay ``kicks in'' when starting from initial data that is well-concentrated in phase space,
 is familiar from work of Aldous and Diaconis on Markov chains, but is new in our continuous phase space setting. 
Our strategies are based on the entropy and spectral methods, and we introduce a new ``index of hypocoercivity'' that is relevant to models of our type involving jump processes and not only diffusion. 
At the heart of our method is a decomposition technique that allows us to adapt Lyapunov's direct method to our continuous phase space setting in order to construct our entropy functionals. 
These are used to obtain precise information on linearized BGK models.
Finally, we also prove local asymptotic stability of a nonlinear BGK model.

\medskip \noindent
\textbf{keywords:} kinetic equations, BGK models, hypocoercivity, Lyapunov functionals, perturbation methods for matrix equations 


\section{Introduction}

This paper is concerned with the large time behavior of nonlinear BGK models (named after the physicists Bhatnagar-Gross-Krook \cite{BGK54}) and their linearizations around their Maxwellian steady state. 
With respect to position, we consider here only models on $\tilde\T^d:=\big(\frac{L}{2\pi}\T\big)^d$, the $d$-dimensional torus of side length $L$ without confinement potential. 
Then, the usual BGK model for a phase space density $f(x,v,t)$; $x\in\tilde\T^d,\,v\in\Rd$ satisfies the kinetic evolution equation
\begin{equation}\label{bgk}
  \partial_t f+v\cdot\nabla_x f= \Q f := M_f(x,v,t) -f(x,v,t)\ ,\quad t\ge0\ ,
\end{equation}
where $M_f$ denotes the local Maxwellian corresponding to $f$; i.e., the
local Maxwellian with the same hydrodynamic moments as $f$:
$$
  M_f(x,v,t)=\frac{\rho(x,t)}{(2\pi T(x,t))^{\frac{d}2}}\,e^{-\frac{|v-u(x,t)|^2}{2T(x,t)}}
  =\frac{\rho(x,t)^{1+\frac{d}{2}}}{(2\pi P(x,t))^\frac{d}2}\,e^{-\frac{|v-u(x,t)|^2\rho(x,t)}{2P(x,t)}}\ ,
$$
with density 
$$
  \rho(x,t):=\intRd f(x,v,t)\,\d[v]\ ,
$$
mean velocity
$$
  u(x,t):=\frac1{\rho(x,t)}\intRd vf(x,v,t)\,\d[v]\ ,
$$
temperature
$$
  T(x,t):=\frac1{d\rho(x,t)}\intRd |v-u(x,t)|^2 f(x,v,t)\,\d[v]\ ,
$$
and pressure (setting the gas constant $R=1$)
$$
  P(x,t):=T(x,t)\rho(x,t)=\frac1{d}\intRd |v-u(x,t)|^2 f(x,v,t)\,\d[v]\ .
$$

Let $\d[\tilde x]:=L^{-d}\d$ denote the normalized Lebesgue measure on $\tilde\T^d$,
 and consider normalized initial data $f^I(x,v)$ such that 
 \begin{equation}\label{f0-normal}
  \int_{\tilde\T^d\times\R^d} f^I(x,v)\d[\tilde x] \d[v] =1\ ,\quad 
  \int_{\tilde\T^d\times\R^d} vf^I(x,v)\d[\tilde x] \d[v] =0\ ,\quad 
  \int_{\tilde\T^d\times\R^d} |v|^2f^I(x,v)\d[\tilde x] \d[v] =d\ .
 \end{equation}
This means, our system has unit mass, zero mean momentum,
 and unit position-averaged pressure (w.l.o.g.\ this can be obtained by a Galilean transformation and choice of units).
One easily checks that this normalization is conserved under the flow of \eqref{bgk}. 
Hence the system \eqref{bgk} is expected to have the unique, space-homogeneous steady state 
 \[ f^\infty(v) = M_1(v):=(2\pi)^{-\frac{d}2}e^{-\frac{|v|^2}2} \ , \]
the centered Maxwellian at unit temperature, which clearly has the same normalization as \eqref{f0-normal}. 
A standard argument involving the Boltzmann entropy confirms that this is indeed the case, but it gives no information on the rate of convergence to equilibrium, nor does it even prove convergence. 
We remark that \eqref{bgk} involves two different time scales: the generic transport time is $O(L)$, while the relaxation time is $O(1)$.
The goal of this paper is to prove the large time convergence to this $f^\infty$ for solutions of \eqref{bgk} and its linearizations in 1, 2, and 3D with explicitly computable exponential rates. 

This extends our previous work \cite{AAC16}, which considered the 1D linear BGK model:
\begin{equation}\label{bgk:linear}
  \partial_t f +v\cdot\nabla_x f = \Qlin f := M_{T}(v)\,\int_\R f(x,v,t)\,\d[v] -f(x,v,t)\ ,\quad t\ge0\ ,
\end{equation}
where $M_{T}$ denotes the normalized Maxwellian at some temperature $T>0$:
\begin{equation*}
  M_T(v) = (2\pi T)^{-1/2}e^{-|v|^2/2T}\ .
\end{equation*}
In \cite{AAC16} we studied the rate at which normalized solutions of 
\eqref{bgk:linear} approach the steady state $f^\infty = M_T$ as $t\to\infty$. 
This problem is interesting since the collision mechanism drives the local velocity distribution towards $M_T$,
 but a more complicated mechanism involving the interaction of the streaming term $v \partial_x$
 and the collision operator $\Qlin$ is responsible for the emergence of spatial uniformity. 

To elucidate this key point, let us define the operator $\LL$ by
$$ \LL f(x,v) := -v\ \partial_x f (x,v) + \Qlin f(x,v)\ .$$
The evolution equation \eqref{bgk:linear} can be written $\partial_t f = \LL f$.  
Let $\mathcal{H}$ denote the weighted space $L^2(\tilde\T\times\R^d;M_T^{-1}(v)\d[\tilde x]\d[v])$, where
 in the current discussion $d=1$.
Then $\Qlin$ is self-adjoint on $\mathcal{H}$, $\LL f^\infty = 0$, 
 and a simple computation shows that if $f(t)$ is a solution of \eqref{bgk:linear},
\begin{equation*}
  \frac{{\rm d}}{{\rm d}t} \| f(t) - f^\infty\|_\mathcal{H}^2 = 2 \langle f(t), \LL f(t)\rangle_\mathcal{H} = 2 \langle f(t), \Qlin f(t)\rangle_\mathcal{H}
    = -2\| f  - M_T \rho\|_{\mathcal{H}}^2\ ,
\end{equation*}
where, as before, $\rho(x,t) := \int_{\R}f(x,v,t) \d[v]$. Thus, while the norm $ \| f(t)- f^\infty\|_\mathcal{H}$ is monotone decreasing, the derivative is zero
whenever $f(t)$ has the form $f(t) = M_T \rho$ for {\em any} smooth density $\rho$. In particular, the inequality
\begin{equation}\label{coercive}
\langle f - f^\infty , \LL (f - f^\infty) \rangle_\mathcal{H} \leq - \lambda \|f - f^\infty\|_\mathcal{H}^2
\end{equation}
is valid in general for $\lambda = 0$, but for no positive value of $\lambda$.  If (\ref{coercive}) were valid for some $\lambda>0$, we would have had
$\|f(t) - f^\infty\|_\mathcal{H}^2 \leq e^{-t\lambda}\|f^I - f^\infty\|_\mathcal{H}^2$ for all solutions of our equation, and we would say that
the evolution equation is {\em coercive}. However, while this is not the case, it does turn out that one still has constants $1< c < \infty$ and $\lambda> 0$
such that
\begin{equation*}
\|f(t) - f^\infty\|_\mathcal{H}^2 \leq ce^{-t\lambda}\|f^I - f^\infty\|_\mathcal{H}^2\ .
\end{equation*}
(The fact that there exist initial data $f(0) \neq f^\infty $ for which the derivative of the norm is zero shows that necessarily $c>1$.) In Villani's terminology
(see \S3.2 of \cite{ViH06}),
this means that our evolution equation is {\em hypocoercive}. 

Since $f(t)$ and $f^\infty$ are probability densities, a natural norm in which to measure the distance between them is the $L^1$ distance, or, what is the same up to a factor of $2$, the total variation distance between the corresponding probability measures. However, as is well known, the norm $\|\cdot \|_\mathcal{H}$ controls the $L^1$ norms. Specifically, by the Cauchy-Schwarz inequality, 
\begin{eqnarray}\label{L1dom}
\|f(t) - f^\infty\|_\mathcal{H}^2 &=&  \int_{\tilde\T\times\R^d}  | f(x,v,t)M_T^{-1}(v)   -1|^2 M_T(v)\d[\tilde x] \d[v]\nonumber\\
&\geq& \left(\int_{\tilde\T\times\R^d}  | f(x,v,t)M_T^{-1}(v)   -1| M_T(v)\d[\tilde x] \d[v]\right)^2\nonumber\\
&=& \| f(t) - f^\infty\|^2_{L^1(\tilde\T\times\R^d,\d[\tilde x]\d[v])}\ .
\end{eqnarray}


Many hypocoercive equations have been studied in recent years \cite{ViH06, He06, DoMoScH09, DoMoScH10, ArEr14}, 
including BGK models in \S 1.4 and \S 3.1 of \cite{DoMoScH10}, 
 but sharp decay rates were rarely an issue there.
In our earlier work~\cite{AAC16}, we 
established hypocoercivity for such models in 1D by an approach that yields explicit -- and quite reasonable -- values for $c$ and $\lambda$. 
To this end, our main tools have been variants of the \emph{entropy--entropy production method}.

The articles \cite{AAC16} and \cite{DoMoScH10} only consider BGK models with conserved mass, and partly with also conserved energy. But the tools presented there did not apply to BGK equations that also conserve momentum. This is in fact an important structural restriction that we shall formalize in \S\ref{sec-hypo-index} with the notion \emph{hypocoercivity index}. The common feature of all models analyzed in \cite{AAC16} as well as in \cite{DoMoScH10} is that their hypocoercivity index is~1. The main goal of this paper is to extend the methods from \cite{AAC16} (i.e.\ constructing feasible Lyapunov functionals) to models with higher hypocoercivity index. Applied to BGK equations this then also includes models with conserved momentum.


\bigskip
The existence of global solutions for the Cauchy problem of~\eqref{bgk} has been proven in case of unbounded domains~\cite{Pe89}
 and bounded domains~\cite{Ri91, PePu93}, respectively.
In case of bounded domains (such as $x\in\tilde\T^d$), these solutions are essentially bounded and unique~\cite{PePu93}. 
For a space-inhomogeneous nonlinear BGK model with an external confinement potential,
 the global existence of solutions for its Cauchy problem
 and their strong convergence in $L^1$ to a Maxwellian equilibrium state has been proven recently~\cite{BoCa09}.
 
\bigskip
In the first part of this paper we shall study the linearization of the BGK equation \eqref{bgk} around the centered Maxwellian with constant-in-$x$ temperature equal to one. 
To this end we consider $f$ close to the global equilibrium $M_1(v)$, with  $h$ defined by $f(x,v,t) = M_1(v)+h(x,v,t)$. Then 
\begin{align} 
  \rho(x,t) &= 1 + \sigma(x,t) \qquad\text{with }\: \sigma(x,t):=\int_{\Rd} h(x,v,t) \d[v]\ , \nonumber \\
  (\rho u)(x,t) &= \int_{\Rd} vf(x,v,t) \d[v] =\mu(x,t) \qquad\text{with the vector function }\: \mu(x,t):=\int_{\Rd} vh(x,v,t) \d[v]\ ,\label{f:perturb} \\ 
  P(x,t) &= \frac1d \int_{\Rd} (|v-u|^2) f(x,v,t) \d[v] 
  = 1+\frac1d\big[\tau(x,t)-\frac{|\mu(x,t)|^2}{1+\sigma(x,t)}\big] \quad\text{with }\: \tau(x,t):=\intRd |v|^2 h(x,v,t) \d[v]\ . \nonumber 
\end{align}
The conservation of the normalizations \eqref{f0-normal} implies
\begin{equation}\label{mloc2}
\int_{\tilde\T^d} \sigma(x,t)\d[\tilde x] =0\ ,  \qquad \int_{\tilde\T^d} \mu(x,t)\d[\tilde x] =0 \ , \qquad{\rm and}\qquad   \int_{\tilde\T^d} \tau(x,t)\d[\tilde x] =0\ .
\end{equation}

The perturbation $h$ then satisfies
$$
 \partial_t h (x,v,t) +  v\cdot \nabla_x h(x,v,t) = [M_f(x,v,t)  -M_1(v)]  - h(x,v,t)\ , \quad t\geq 0 \ .
$$
For $\sigma$, $\mu$, and $\tau$ small we have
\begin{align}\label{taylor}
 & M_f(x,v) - M_1(v) \nonumber \\
 &= \frac{(1+\sigma)^{1+\frac{d}2}(x)}{\left(2\pi \big\{1+\frac1d\big[\tau(x)-\frac{|\mu|^2(x)}{1+\sigma(x)}\big]\big\}\right)^\frac{d}2} 
 \,\exp\Big\{-\frac{|v(1+\sigma(x))-\mu(x)|^2}{2\big(1+\frac1d\big[\tau(x)-\frac{|\mu|^2(x)}{1+\sigma(x)}\big]\big) (1+\sigma(x))}\Big\}
  - (2\pi)^{-\frac{d}2}e^{-\frac{|v|^2}2}\nonumber \\
  &\approx M_1(v)\left[\big( 1+\frac{d}2 - \frac{|v|^2}{2}\big)\sigma(x)+v\cdot \mu(x)
      + \big( -\frac{1}{2}+ \frac{|v|^2}{2d}\big)\tau(x)\right] \ ,
\end{align}
which yields the linearized BGK model that we shall analyze in dimensions 1, 2, and 3 in this paper:
\begin{align} \label{linBGK:torus}
 & \partial_t h (x,v,t) +  v\cdot \nabla_x h(x,v,t) \\
 &\quad = M_1(v)\left[\big( 1+\frac{d}2 - \frac{|v|^2}{2}\big)\sigma(x,t)+v\cdot \mu(x,t)
      + \big( -\frac{1}{2}+ \frac{|v|^2}{2d}\big)\tau(x,t)\right]  - h(x,v,t)\ , \quad t\geq 0 \ .\nonumber
\end{align}
Here and in the sequel we only have $h(x,v,t)\approx f(x,v,t)-M_1(v)$, but for simplicity of notation we shall still denote the perturbation by $h$.

\begin{theorem}[{\bf decay estimate for the linearized BGK \eqref{linBGK:torus} in dimensions 1, 2, and 3}] \label{linBGK-decay}
For each side length $L>0$ and for dimensions $d=1,2,3$, there exists a (quadratic) entropy functional $\cE^d(f)$ and a decay rate $\lambda^d(L)>0$ satisfying
 \begin{equation} \label{entropy-equiv}
 c_d(L) \cE^d(f) \leq \| f - M_1\|^2_\mathcal{H} \le C_d(L) \cE^d(f)\ ,
 \end{equation}
with some positive constants $c_d,\,C_d$ 
given explicitly in the proofs.
Moreover, any solution $h(t)$ to \eqref{linBGK:torus} with $\cE^d(h^I+M_1)<\infty$, normalized according to \eqref{mloc2}, then satisfies
 \begin{equation} \label{ineq:linBGK-decay}
   \cE^d(h(t)+M_1) \leq e^{-\lambda^d(L)\ t}\ \cE^d(h^I+M_1)\ ,\qquad t\ge0\ .
 \end{equation}
\end{theorem}

\begin{remark}
\begin{enumerate}[label=(\alph*)]
\item Combining \eqref{L1dom} and the bound on the right in \eqref{entropy-equiv}, we obtain a 
{\em Pinsker type inequality} \cite{Pi64} for our entropy. Let $\tilde f := h+M_1$. Then
\begin{equation}\label{pinsker}
 \| \tilde f - M_1\|_{L^1(\tilde\T\times\R^d,\d[\tilde x]\d[v])}  \le \sqrt{C_d(L) \cE^d(\tilde  f)}\ .
\end{equation}
\item For any solution $h(t)$ to \eqref{linBGK:torus} with $\cE^d(h^I+M_1)<\infty$, normalized according to \eqref{mloc2}, the function $\tilde f(t):=h(t)+M_1$ satisfies
\begin{equation}\label{twoscale1}
\| \tilde f(t) - M_1\|_{L^1(\tilde\T\times\R^d,\d[\tilde x]\d[v])} \le \sqrt{C_d(L) \cE^d(\tilde f^I)} e^{-\lambda^d(L)\ t / 2}\,,
\end{equation}
 due to~\eqref{ineq:linBGK-decay} and~\eqref{pinsker} with $\tilde f^I :=h^I +M_1$.
However, since $\tilde f(t)$ and $M_1$ are both probability measures, we also have
\begin{equation}\label{twoscale2}
\| \tilde f(t) - M_1\|_{L^1(\tilde\T\times\R^d,\d[\tilde x]\d[v])}  \le 2
\end{equation}
for all $t$. Moreover, if most of the mass density is initially located in a small portion of $\tilde\T$; 
 e.g., if the gas molecules are initially released from a small container into a vacuum in the rest of $\tilde\T$,
 then $\| \tilde f(t) - M_1\|_{L^1(\tilde\T\times\R^d,\d[\tilde x]\d[v])}$ will be close to $2$ until the streaming has had time to distribute the particles more uniformly over $\tilde \T$. 
Our estimates bound the time that it takes for this to happen. 

Combining \eqref{twoscale1} with \eqref{twoscale2} yields
\begin{equation} \label{Hnorm-decay-est}
 \| \tilde f(t) - M_1\|_{L^1(\tilde\T\times\R^d,\d[\tilde x]\d[v])}  \le \min\left\{2,\,\sqrt{C_d(L)\,\cE^d(\tilde f^I)}\,e^{-\lambda^d(L)\ t/2} \right\}\ ,\qquad t\ge0\ .
\end{equation}
Our bound~\eqref{twoscale1} improves the trivial bound~\eqref{twoscale2} 
only for $t> t_{{\rm init}}$ where
\begin{equation*}
t_{{\rm init}} := \frac{ \log C_d(L)  +  \log \cE^d(\tilde f^I)-2\log 2}{  \lambda^d(L) } \ .
\end{equation*}

For the one dimensional case, it is shown in Remark~\ref{rem3.2} that $\lambda^1(L) = O(1/L^2)$ in the limit $L\to\infty$. 
Moreover, the constant $C_1(L)$ approaches $1$ in the limit $L\to\infty$
 by using the limiting behavior $\alpha_*(L)=O(1/L)$ in expression~\eqref{norm-const:1D}. 
For  initial data $\tilde f^I$ with all of the gas molecules initially located in a small region of $\tilde\T$ with a volume fraction of order $\epsilon$, the initial entropy $\cE^1(\tilde f^I)$ will satisfy $\cE^1(\tilde f^I) = O(\epsilon^{-2})$. In this case, $t_{{\rm init}}$ is approximately given by
\[ O(-L^2 (C +\log \epsilon) )\quad \mbox{for } \epsilon\ll 1,\;L\gg1 \]
and some positive constant $C$.
Thus one time scale in our problems is given, or at least bounded, by $t_{{\rm init}}$. 
After this time, the solution satisfies
\begin{equation}\label{twoscale3b}
\| \tilde f(t) - M_1\|_{L^1(\tilde\T\times\R^d,\d[\tilde x]\d[v])}  \leq 2 e^{-\lambda^d(L)\ (t- t_{{\rm init}})/2}\ ,
\end{equation}
and the second time scale, is given by $2/\lambda^d(L)$, the  waiting time after $t_{{\rm init}}$ 
for $\| \tilde f(t) - M_1\|_{L^1(\tilde\T\times\R^d,\d[\tilde x]\d[v])}$ to decrease by a factor of $1/e$; 
see Figure \ref{fig:H-decay}. 

These two times scales are quite similar to what one observes in interacting particle systems or even in card shuffling; see \cite{AlDi86,Di96}. 
In particular, \cite[Fig. 2]{AlDi86} is quite similar to our Fig. \ref{fig:H-decay} below. 

\item The resemblance of \eqref{twoscale3b} to the results of Aldous and Diaconis for finite Markov chains in \cite{AlDi86,Di96}, and in particular for card shuffling, is not a coincidence. 
The equation \eqref{bgk:linear} can be interpreted as the Kolmogorov forward equation for a Markov process. 
Exponential rates for related Markov process with exponential rates have been obtained by probabilistic methods;
 see \cite{BL} for an early study of this type. 
However, the approach in \cite{BL} relies on compactness arguments and does not yield explicit values for $c$ or $\lambda$. 
One difference between our results and those for finite Markov chains is that in our case, the initial relative entropy can be infinite. 
In card shuffling, starting form a perfectly ordered deck of cards,
 one starts from a state of maximal---but finite---relative entropy, and the waiting time for uniformization from this state dominates that of any other starting point. 
For this reason, the initial waiting time for finite Markov chains is a universal ``worst case'', while this is impossible in our setting; our result must refer to $\cE^d(\tilde f^I)$. 

\item Our bound on the decay rate is monotonically decreasing in $L$ and satisfies $\lambda^d(L=0)>0$ and  $\lambda^d(L=\infty)=0$ (for $d=1$ see Fig.\ \ref{fig:muL} below). Moreover $c_d(L=0)=C_d(L=0)=1$ (see \eqref{alpha-limit}, \eqref{norm-const:1D} below).
\end{enumerate}
\end{remark}
\begin{figure}[ht!]
\begin{center}
 \includegraphics[scale=.6]{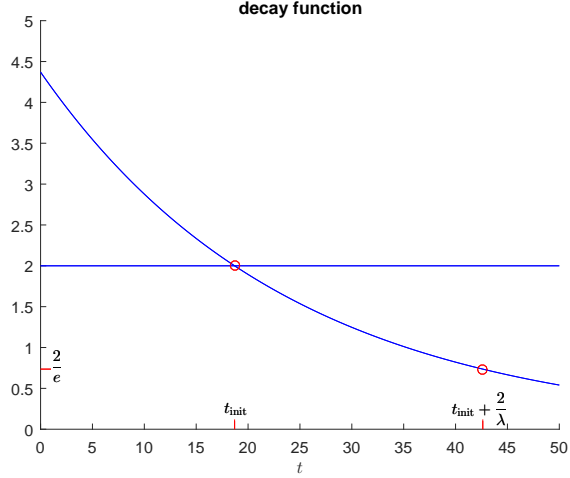}
 \caption{These two functions illustrate the time dependent decay estimate from \eqref{Hnorm-decay-est}. The values of $C_d,\,\lambda^d$ correspond to the 1D case with $L=2\pi$, and we chose $\cE^d(\tilde f^I)=15$. We also show the two time scales of the BGK equation: $t_{{\rm init}}$ marks the intersection point of the two blue curves and it corresponds to the generic transport time. $t_2:=t_{{\rm init}}+\frac2\lambda$ marks the intersection point of the exponential curve with the value $2/e$, and $t_2-t_{{\rm init}}$ corresponds to the relaxation time scale.  For larger values of $L$, $t_{{\rm init}}$ will be much larger. 
 }
 \label{fig:H-decay}
\end{center}
\end{figure}
\medskip

To prove local asymptotic stability for the nonlinear BGK equation \eqref{bgk} in 3D,
 we make use of another set of norms:
For $\gamma \geq 0$, let $H^\gamma(\tilde\T^3)$ be the Sobolev space
 consisting of the completion of smooth functions $\varphi$ on $\tilde\T^3$ in the Hilbertian norm 
 \[ \|\varphi\|_{H^\gamma}^2 := \sum_{k\in \Z^3} (1+|k|^2)^{\gamma}|\varphi_k|^2\ , \]
 where $\varphi_k$ ($k\in\Z^3$) is the $k$th Fourier coefficient of $\varphi$. 
Let $\mathcal{H}_\gamma$ denote the Hilbert space $H^\gamma(\tilde\T^3)\otimes L^2(\R^3;M_1^{-1})$.  
Then the inner product in $\mathcal{H}_\gamma$ is given by
 \[ \langle f,g\rangle_{\mathcal{H}_\gamma} = \int_{\tilde\T^3\times\R^3} \overline{f}(x,v) \left[\left(1 - \Delta_x\right)^{\gamma} g(x,v) \right]M_1^{-1}(v) \d[\tilde x] \d[v] \ . \]

\begin{theorem}[{\bf decay estimates for the linearized and nonlinear BGK equation \eqref{bgk} in 3D}] \label{BGK-decay}
Let $L=2\pi$ and let the initial data $f^I$ satisfy the normalization~\eqref{f0-normal}. 
\begin{enumerate}[label=(\alph*)]
\item \label{BGK-decay:a}
For all $\gamma \ge 0$ there is an entropy functional $\cE_\gamma(f)$ satisfying
 \begin{equation*}
  \frac34 \cE_\gamma(f)  \leq \|f-M_1\|_{\mathcal{H}_\gamma}^2 \leq \frac32 \cE_\gamma(f)
 \end{equation*}
  such that, if $h$ is a solution of the linearized BGK equation \eqref{linBGK:torus} in 3D
 with initial data $h^I$ and ${\cE_\gamma(h^I+M_1)< \infty}$, then
 \begin{equation*} 
  \cE_\gamma(h(t)+M_1) \leq e^{-t/2820}\cE_\gamma(h^I+M_1)\ ,\qquad t\ge0\ .
 \end{equation*}
\item \label{BGK-decay:b}
Moreover, for all $\gamma>3/2$, there is an explicitly computable  $\delta_\gamma>0$ such that, 
 if $f$ is a solution of the nonlinear BGK equation \eqref{bgk} with initial data $f^I$
 and $\|f^I-M_1\|_{\mathcal{H}_\gamma} < \delta_\gamma$,
 then for the same entropy function~$\cE_\gamma$, the following decay estimate holds:
$$
 \cE_\gamma(f(t)) \leq e^{-t/2820}\cE_\gamma(f^I)\ ,\qquad t\ge0\ .
$$ 
\end{enumerate}
\end{theorem}
Note that part~\ref{BGK-decay:a} of this theorem generalizes Theorem~\ref{linBGK-decay} to the Sobolev-type entropies $\cE_\gamma(f)$ in the case $d=3$, $L=2\pi$.

\bigskip
This paper is organized as follows: 
In \S\ref{sec:P-method} we review from \cite{AAC16} a Lyapunov-type method for hypocoercive ODEs that yields their sharp exponential decay rate. 
While this approach requires all eigenvectors of the system matrix, we also develop an approach using simplified Lyapunov functionals. 
This alternative strategy comes at the price of yielding only a suboptimal decay rate,
 but it can be extended to infinite dimensional systems and BGK equations. 
In \S\ref{sec:linBGK:1D} we apply the second strategy to the linearized BGK model \eqref{linBGK:torus} in 1D,
 proving exponential decay of the solution towards the spatially uniform Maxwellian, 
 as stated in Theorem \ref{linBGK-decay}. 
This is based on decomposing \eqref{linBGK:torus} into spatial Fourier modes and introducing a Hermite function basis in velocity direction. 
In the Sections \ref{sec:linBGK:2D} and \ref{sec:linBGK:3D} we extend our result to 2D and 3D, respectively. 
But this is not straightforward, as it is already not obvious how to choose a convenient Hermite function basis in multi dimensions. 
Finally, in \S\ref{sec:6} we prove local exponential stability of the nonlinear BGK equation \eqref{bgk} in 3D as stated in Theorem \ref{BGK-decay}(b).



\section{Decay of hypocoercive ODEs}
\label{sec:P-method}

The local convergence result in Theorem~\ref{BGK-decay}(b) is obtained from the global convergence result in Theorem~\ref{linBGK-decay} and a relatively straightforward control of the errors involved in linearization. Therefore, the essential content of the paper concerns the linearized BGK equations. To this end we shall rewrite them as ODEs -- of infinite dimension -- in fact. We therefore begin this section with a discussion of  the hypocoercivity structure of ODEs and review (from \cite{AAC16}) a Lyapunov-type method that yields their sharp decay rate.

\subsection{Lyapunov's direct method}

To illustrate the method we start with linear, finite dimensional ODEs.
Consider an ODE for a vector $f(t)=(f_1(t),\,f_2(t),\ldots,f_n(t))^\top\in \C^n$:
 \begin{equation} \label{ODE:general:n}
  \begin{cases}
   \ddt{f}= -\CC f,\quad t\ge 0\,,\\
   f(0)   = f^I \in\C^n\,,
  \end{cases}
 \end{equation}
for some (typically non-Hermitian) matrix $\CC\in\C^{n\times n}$. 
The stability of the steady state $f^0\equiv 0$ is determined by the eigenvalues of matrix $\CC$:
\begin{theorem} 
 Let $\CC\in\C^{n\times n}$ and let $\lambda_j$ ($j=1,\ldots,n)$ denote the eigenvalues of $\CC$
  (counted with their multiplicity).
 \begin{enumerate}[label=(S\arabic*)]
  \item The equilibrium $f^0$ of \eqref{ODE:general:n} is stable 
    if and only if (i) $\Re (\lambda_j) \geq 0$ for all $j=1,\ldots,n$;
    and (ii) all eigenvalues with $\Re(\lambda_j)=0$ are 
    non-defective\footnote{An eigenvalue is {\em defective} if its geometric multiplicity is strictly less than its algebraic multiplicity. This difference is called {\em defect}.}.
  \item 
   The equilibrium $f^0$ of \eqref{ODE:general:n} is asymptotically stable 
    if and only if $\Re (\lambda_j)> 0$ for all $j=1,\ldots,n$.
  \item The equilibrium $f^0$ of \eqref{ODE:general:n} is unstable 
    in all other cases.
 \end{enumerate}
\end{theorem}

For positive definite Hermitian matrices $\CC$, 
 using the Lyapunov functional $\|f\|^2$ in the energy method allows to obtain the sharp decay rate,
 which is the smallest eigenvalue $\mu$ of $\CC$: 
The derivative of $\|f\|^2$ along solutions $f(t)$ of \eqref{ODE:general:n} satisfies
 \begin{equation*} 
  \ddt \|f(t)\|^2  = -\ip{f(t)}{(\CC^* + \CC)f(t)} = -2 \ip{f(t)}{\CC f(t)} \leq -2\mu \|f(t)\|^2 \,, 
 \end{equation*}
 where $\CC^*$ denotes the Hermitian transpose of $\CC$.
Note that the derivative of $\|f\|^2$ depends only on the Hermitian part $\tfrac12 (\CC^* + \CC)$ of matrix $\CC$,
 such that for a Hermitian matrix~$\CC$ there is no loss of information.
 
But for non-Hermitian matrices it is more natural to use a modified norm: 
 \[ \|f\|^2_\P := \langle f,\P f\rangle\,, \]
 for some positive definite Hermitian matrix $\P\in\C^{n\times n}$, to be derived from $\CC$. 
The derivative of $\|f\|^2_\P$ along solutions $f(t)$ of \eqref{ODE:general:n} satisfies
 \begin{equation*}
  \ddt \|f(t)\|^2_\P
     = -\ip{f(t)}{(\CC^* \P +\P \CC)f(t)} \,. 
 \end{equation*}
Then, $f^0\equiv 0$ is asymptotically stable,
 if there exists a positive definite Hermitian matrix~$\P$ such that $\CC^* \P +\P \CC$ is positive definite. 
To determine the decay rate to $f^0$, and to choose $\P$ conveniently
 we shall use the following algebraic result. 
\begin{lemma}[{\cite[Lemma 2]{AAC16}}] \label{lemma:Pdefinition}
For any fixed matrix $\CC\in\C^{n\times n}$,
 let $\mu:=\min\{\Re(\lambda)|\lambda$ is an eigenvalue of $\CC\}$. 
Let $\{\lambda_{j}|1\leq j\leq j_0\}$ be all the eigenvalues of $\CC$ with $\Re(\lambda_j)=\mu$, only counting their geometric multiplicity.
%
 If all $\lambda_j$ ($j=1,\dots,j_0$) are non-defective, 
  then there exists a positive definite Hermitian matrix $\P\in\C^{n\times n}$ with
  \begin{align} \label{matrixestimate1}
   \CC^*\P+\P \CC &\geq 2\mu \P\,.
  \end{align}
But $\P$ is not uniquely determined.
 Moreover, if all eigenvalues of $\CC$ are non-defective, such matrices $\P$ satisfying~\eqref{matrixestimate1} are given by
 \begin{align} \label{simpleP1} 
  \P:= \sum\limits_{j=1}^n b_j \,\overline{w_j}\otimes w_j^\top \,,
 \end{align}
 where $w_j$ ($j=1,\dots,n$) denote the left eigenvectors of $\CC$, and $b_j\in\R^+$ ($j=1,\dots,n$) are arbitrary weights.
\end{lemma}

\begin{remark}
 \begin{enumerate}
  \item[(i)]
This result was proven in \cite[Lemma 4.3]{ArEr14} for real matrices $\CC\in\R^{n\times n}$,
 and in \cite[Lemma 2]{AAC16} for complex matrices $\CC$.
In particular, if $\CC$ is a real matrix, 
 then the inequality~\eqref{matrixestimate1} of Lemma~\ref{lemma:Pdefinition} holds true
 for some real positive definite symmetric matrices $\P\in\R^{n\times n}$.
  \item[(ii)] For the extension of the above lemma to the case of defective eigenvalues see~\cite[Lemma 4.3(i)]{ArEr14} and~\cite[Prop. 2.2]{AAS15}. But the construction of~$\P$ then involves also the generalized eigenvectors,
  \item[(iii)] The Lyapunov inequality \eqref{matrixestimate1} is a special case of a linear matrix inequality.
In a standard reference of system and control theory~\cite{Boyd94},
the problem of finding the maximal positive constant $\mu$ and a positive definite matrix $\P$ satisfying \eqref{matrixestimate1} is formulated as a generalized eigenvalue problem, see \cite[\S 5.1.3]{Boyd94}.
The optimal value for the constant $\mu$ is pointed out,
but the associated matrices $\P$ (like in our construction \eqref{simpleP1}) are not specified.
 \end{enumerate}
\end{remark}
Now we consider examples, where all eigenvalues of $\CC\in\C^{n\times n}$ are non-defective and have positive real parts. Then the origin is the unique and asymptotically stable steady state $f^0=0$ of~$\eqref{ODE:general:n}$:
Due to Lemma~\ref{lemma:Pdefinition}, 
 there exists a positive definite Hermitian matrix $\P\in\C^{n\times n}$ such that $\CC^* \P +\P\CC \geq 2\mu \P$
 where $\mu = \min \Re(\lambda_j) >0$.
Thus, the derivative of $\|f\|^2_\P = \ip{f}{\P f}$ along solutions of \eqref{ODE:general:n} satisfies
 \begin{align*}
  \ddt \|f(t)\|^2_\P 
    &\leq -2\mu \|f(t)\|^2_\P \qquad \text{with} \quad \mu = \min \Re(\lambda_j),
 \end{align*} 
which implies 
\begin{equation} \label{ODE-decay}
 \|f(t)\|^2_\P\leq e^{-2\mu t} \|f^I \|^2_\P \,,\qquad t\ge0\,.
\end{equation}  
Let $\lambda^\P_j$ $(j=1,\ldots,n)$ denote the positive eigenvalues of the positive definite Hermitian matrix~$\P$
 being ordered by magnitude such that $0 <\lambda^\P_1 \leq\ldots \leq\lambda^\P_n$.
Then the matrix inequality $\lambda^\P_1 \II \leq \P \leq\lambda^\P_n \II$ implies the equivalence of norms
\[ \lambda^\P_1 \|v\|^2 \leq \|v\|_\P^2 \leq \lambda^\P_n \|v\|^2 \qquad \forall v\in\C^n \,. \]
Thus the decay in $\P$-norm~\eqref{ODE-decay} translates into a decay in the Euclidean norm
\begin{equation} \label{ODE:exponentialDecay}
 \norm{f(t)}^2\leq c e^{-2\mu t} \norm{f^I}^2\,,
\end{equation}
with the constant $c=\lambda^\P_n /\lambda^\P_1 \ge1$, i.e.\ the condition number of $\P$.
Note that $c=1$ if and only if $\P=\II$.
\begin{remark}
In a popular textbook on linear systems theory~\cite{Hes09},
 the exponential decay~\eqref{ODE:exponentialDecay} is obtained as follows~\cite[\S 8.5]{Hes09}:
For a stable matrix $-\CC$ (i.e. all eigenvalues of $-\CC$ have negative real part) and a matrix~$\Q$, 
 the unique solution~$\P$ of Lyapunov's equation
 \begin{equation*} 
  \CC^* \P +\P\CC = \Q
 \end{equation*}
 is given by \[ \P := \int_0^\infty e^{-\CC^* t}\ \Q\ e^{-\CC t} \d[t]\ . \]
If $\Q$ is a positive definite symmetric matrix,
 then the unique solution~$\P$ is also symmetric and positive definite.
Moreover, the $\P$-norm of any solution $f(t)$ of \eqref{ODE:general:n} satisfies
\[
   \ddt \|f(t)\|^2_\P 
    = -\ip{f(t)}{(\CC^* \P +\P \CC)f(t)} 
    = -\ip{f(t)}{\Q f(t)}   
    \leq -\min \lambda^\Q_j \|f(t)\|^2
    \leq -\frac{\min \lambda^\Q_j}{\max \lambda^\P_j} \|f(t)\|^2_\P \ ,
\]
where $\lambda^\Q_j$ and $\lambda^\P_j$ are the positive eigenvalues of the positive definite symmetric matrices $\Q$ and $\P$.
This implies \eqref{ODE-decay} with $2\mu ={\min \lambda^\Q_j}/{\max \lambda^\P_j}$.
However, only a suitable choice for $\Q$ would allow to recover the optimal decay rate as achieved in Lemma~\ref{lemma:Pdefinition}.
\end{remark}
The preceding discussion allows to characterize coercive and hypocoercive systems of linear ODEs (as well as matrices) according to the definition in the introduction:
Equation~\eqref{ODE:general:n} with matrix~$\CC$ is \emph{coercive}, if the Hermitian part of $\CC$ is positive definite,
i.e. 
 \begin{equation*} 
  \exists\kappa>0 \quad \text{such that} \quad \CC_H:=\tfrac12 (\CC+\CC^*) \ge \kappa \II \,.
 \end{equation*}
In this case, the trivial energy method (i.e.\ multiplying \eqref{ODE:general:n} by $\overline{f}(t)^\top$
 and using $\|f\|^2$ as a Lyapunov functional) shows decay of $f(t)$ with rate $\kappa$ and $c=1$. 
But this exponential rate is not necessarily sharp, e.g. for some non-Hermitian matrices $\CC$.

Equation~\eqref{ODE:general:n} with matrix~$\CC$ is \emph{hypocoercive (with trivial kernel)},
if there exists $\mu>0$ such that all eigenvalues of $\CC$ satisfy
 \[ \Re(\lambda_j)\ge\mu\,, \qquad j=1,...,n\,. \]
While this notion was originally coined for operators in PDEs, such matrices are typically also called \emph{positively stable}.

Comparing the spectrum of $\CC$ and $\CC_H$, it is well known that the maximum constants $\kappa$ and $\mu$ satisfy $\kappa\le\mu$.
If all eigenvalues of $\CC$ with $\Re(\lambda_j)=\mu$ are \emph{non-defective},
 then $f(t)$ decays at least with rate $\mu$. 
However, if $\CC$ has a defective eigenvalue with $\Re(\lambda)=\mu$, then $f(t)$ decays ``slightly slower'',
 i.e.\ with rate $\mu-\eps$, for any $\eps>0$ (see \cite[Proposition 4.5]{ArEr14} and \cite[Proposition 2.2]{AAS15} for details -- applied to hypocoercive Fokker-Planck equations). 
Very recently this decay result has been improved as follows: In this case there is still a positive definite matrix~$\P$, but it cannot be given by the simple formula~\eqref{simpleP1}, and~\eqref{ODE-decay} becomes
\begin{equation}\label{defective-decay}
 \|f(t)\|^2_\P \leq C (1+t^{2m}) e^{-2\mu t}  \|f^I\|^2_\P\ 
\end{equation}
for some $C>0$, where~$m$ is the maximal defect of the eigenvalues of~$\CC$ with $\Re(\lambda_j)=\mu$.
See \cite{ArEiWo} for more information. 

\subsection{Index of hypocoercivity}\label{sec-hypo-index}

For the BGK models analyzed below we intend to construct convenient Lyapunov functionals of the form $\langle  f,\P f\rangle$, where the matrix $\P$ does not necessarily have to reveal the sharp spectral gap of $\CC$ (in the sense of Lemma \ref{lemma:Pdefinition}).
To this end we first give a definition of the structural complexity of a hypocoercive equation of the form
\begin{equation}\label{matrix-ODE}
  \ddt f +i\; \CC_1 f = -\CC_2 f\ ,\quad t\ge0\ .
\end{equation}
Here we decomposed the matrix $\CC\in\C^{n\times n}$ as $\CC=i\CC_1+\CC_2$
with Hermitian matrices $\CC_1$ and $\CC_2$ with $\CC_2\ge0$.
In the special case $\CC_1=0$, $(\ker \CC_2)^\perp$ corresponds to the subspace of decaying solutions $f(t)$, and $\ker \CC_2$ to the non-decaying subspace. 
In hypocoercive equations, the semigroup generated by the skew-Hermitian matrix $i\CC_1$ may turn non-decaying directions into decaying directions,
 hence allowing for an exponential decay of all solutions.
More precisely, we assume
\begin{equation} \label{hypocoercive:2}
 \exists \tau\in \N_0 \quad \text{and}\quad \exists \kappa>0 \ : \qquad \sum_{j=0}^\tau \CC_1^j \CC_2 (\CC_1)^j \geq \kappa \II \ . 
\end{equation}

\begin{definition}\label{hyp-index}
 For Hermitian matrices $\CC_1$ and $\CC_2$ with $\CC_2\ge0$,
 the \emph{hypocoercivity index} of the matrix $\CC$ (and of the ODE \eqref{matrix-ODE}) is the smallest $\tau\in\N_0$, such that \eqref{hypocoercive:2} holds.
\end{definition}
Clearly, $\tau=0$ corresponds to coercive matrices $\CC$;
 i.e., those for which all eigenvalues of its Hermitian part $\tfrac12(\CC+\CC^*)$ are strictly positive.
A simple computation shows that this definition is invariant under a change of basis. 
We note that condition \eqref{hypocoercive:2} is identical to the matrix condition in Lemma 2.3 of \cite{ArEr14}, which characterizes the hypoellipticity of degenerate Fokker-Planck operators of the form $\LL f=\diver(\DD\nabla f+\CC xf)$ 
(using the matrix correspondence $\DD=\CC_2$, $\CC=\CC_1$). 
Hence, condition \eqref{hypocoercive:2} for the ODE \eqref{matrix-ODE} and its hypocoercivity index can be seen as an analogue of the \emph{finite rank H\"ormander condition} for hypoelliptic and degenerate diffusion equations \cite[Th.\ 1.1]{Ho67}. 
While the hypocoercivity index of degenerate parabolic equations determines the algebraic regularization rate (e.g.\ from $L^2$ into $H^1$, see Theorem A.12 in \cite{ViH06} and Theorem 4.8 in \cite{ArEr14}), its role in hypocoercive ODEs is not yet clear.

\subsubsection{Equivalent hypocoercivity conditions} 

Next, we collect several statements which are equivalent to condition~\eqref{hypocoercive:2}. They will be useful for the analysis in \S\ref{sec:Pmatrix}.
\begin{proposition} \label{prop:equivalence}
 Suppose that $\CC_1\in\C^{n\times n}$ and $\CC_2\in\C^{n\times n}$ are Hermitian matrices.
 Suppose furthermore that $\CC_2$ is positive semi-definite.
 Then the following conditions are equivalent:
 \begin{enumerate}[label=(B\arabic*)]
 	\item \label{cond:fullrank} There exists $\tau\in\N_0$ 
 	 such that \[ \rank\{\sqrt{\CC_2},\CC_1 \sqrt{\CC_2}, \ldots, \CC_1^\tau \sqrt{\CC_2}\} =n \ , \]
	 which is often called \emph{Kalman rank condition}.
	\item \label{cond:hypocoercive} The matrices $\CC_1$ and $\CC_2$ satisfy condition~\eqref{hypocoercive:2}.
 \item \label{cond:invariance} No non-trivial subspace of $\ker \CC_2$ is invariant under $\CC_1$.
	\item \label{cond:eigenvector} No eigenvector of $\CC_1$ lies in the kernel of $\CC_2$.
  \item \label{cond:compMatrix} There exists a skew-Hermitian matrix $\K$ such that $\CC_2 +[\K,\CC_1] =\CC_2 +(\K\CC_1 -\CC_1\K)$ is positive definite.
 \end{enumerate}
Moreover, the smallest possible $\tau$ in~\ref{cond:fullrank} and~\ref{cond:hypocoercive} coincides; it is the hypocoercivity index of $\CC$.
\end{proposition}
\begin{proof}
 The equivalence of~\ref{cond:fullrank} and~\ref{cond:hypocoercive} (with the same $\tau$) follows from~\cite[Lemma 2.3]{AAS15}.
 The equivalence of~\ref{cond:hypocoercive}--\ref{cond:eigenvector} follows from~\cite[Lemma 2.3]{ArEr14}.
 The equivalence of~\ref{cond:eigenvector} and~\ref{cond:compMatrix} follows by the same arguments as for real symmetric matrices in~\cite[Theorem 2.5]{ShKa85}.
\end{proof}
\begin{remark}\label{rem:hypo:cond:B1'}
\item
\begin{enumerate}[label=(\alph*)]
\item In order to use condition \ref{cond:fullrank} later on also for ``infinite matrices'' we give here an equivalent version:
 \begin{enumerate}[label=(B1')] 
 \item \label{cond:trivkernel} There exists $\tau\in\N_0$ 
  such that $\bigcap_{j=0}^\tau \ker\big(\sqrt{\CC_2}\CC_1^j\big) =\{0\}$.
 \end{enumerate}
 \item
  If $\tau\in\N_0$ is such that 
  \begin{equation}\label{rank-tau1} 
    \rank\{\sqrt{\CC_2},\CC_1 \sqrt{\CC_2}, \ldots, \CC_1^\tau \sqrt{\CC_2}\} = \rank\{\sqrt{\CC_2},\CC_1 \sqrt{\CC_2}, \ldots, \CC_1^\tau \sqrt{\CC_2}, \CC_1^{\tau+1} \sqrt{\CC_2}\} \ ,
  \end{equation}
  then for all $k\in\N$ 
  \[ \rank\{\sqrt{\CC_2},\CC_1 \sqrt{\CC_2}, \ldots, \CC_1^\tau \sqrt{\CC_2}\} = \rank\{\sqrt{\CC_2},\CC_1 \sqrt{\CC_2}, \ldots, \CC_1^{\tau+k} \sqrt{\CC_2}\} \,. \]
  Condition \eqref{rank-tau1} implies that the columns of $\CC_1^{\tau+1} \sqrt{\CC_2}$ are linear combinations of the columns of $\CC_1^j \sqrt{\CC_2}$, $j\in\{0,\ldots,\tau\}$.
  This implies that $\CC_1^{\tau+k} \sqrt{\CC_2}$ are linear combinations of the columns of $\CC_1^j \sqrt{\CC_2}$, $j\in\{k-1,\ldots,\tau+k-1\}$.
  Hence, for a hypocoercive matrix we have to gain with each added term in \eqref{rank-tau1} at least one rank until we reach full rank, i.e.\ space dimension~$n$.
  Thus, for hypocoercive matrices its hypocoercivity index is bounded from above by the dimension of $\ker\CC_2$ (or equivalently corank of~$\CC_2$). 
\end{enumerate}
\end{remark}

In~\cite[Remark 17]{ViH06} the connections of the above conditions to \emph{Kawashima's nondegeneracy condition} for the study of degenerate hyperbolic-parabolic systems \cite{Ka87} and \emph{H\"ormander's rank condition} for hypoelliptic equations \cite{Ho67} are noted. 

For real symmetric matrices $\CC_1, \CC_2\in \R^{n\times n}$ with $\CC_2\geq 0$,
 condition~\ref{cond:eigenvector} is equivalent to the condition that $\CC :=i\CC_1 +\CC_2$ has only eigenvalues with positive real part, see~\cite[Theorem 1.1]{ShKa85}. And the latter statement is equivalent to the exponential stability of \eqref{matrix-ODE}.
Using Proposition~\ref{prop:equivalence}, we shall now prove a similar statement for Hermitian matrices:

\begin{lemma} \label{lem:hypocoercive}
 Hermitian matrices $\CC_1$ and $\CC_2$ with $\CC_2\ge0$ satisfy condition~\eqref{hypocoercive:2}
 if and only if all eigenvalues~$\lambda_{\CC}$ of $\CC :=i\CC_1 +\CC_2$ have positive real part $\Re (\lambda_{\CC}) >0$. 
\end{lemma}
To show Lemma~\ref{lem:hypocoercive} for Hermitian matrices,
 we will follow the proofs of~\cite[Prop. 2.4]{UmKaSh84} and~\cite[Lemma 3.2]{ShKa85} for real symmetric matrices.
\begin{proof}[Proof of Lemma~\ref{lem:hypocoercive}]
 First, we show that condition~\eqref{hypocoercive:2} implies
  that all eigenvalues $\lambda_{\CC}$ of $\CC :=i\CC_1 +\CC_2$ have positive real part $\Re (\lambda_{\CC}) >0$:
 Let $\phi$ be an eigenvector of $\CC$ corresponding to an eigenvalue $\lambda$, i.e.
 \begin{equation} \label{eq:EVP}
  \lambda\phi = \CC \phi = (i\CC_1 +\CC_2) \phi \ .
 \end{equation}
 Take the complex inner product of this equation with $\phi$, to obtain
 \[ \overline{\lambda} \langle\phi,\phi\rangle = \langle \CC \phi, \phi\rangle \ , \]
 using $\langle\phi,\psi\rangle = \overline{\phi}^\top \psi$ for all $\phi,\psi\in\C^n$.
 Its real part satisfies 
 \begin{equation} \label{eq:Re:lambda}
  \Re (\lambda) \langle\phi,\phi\rangle = \langle\CC_2 \phi,\phi\rangle \ ,
 \end{equation}
 due to the assumptions on the matrices $\CC_1$ and $\CC_2$.
 Moreover, there exists a skew-Hermitian matrix~$\K$ such that $\CC_2 +[\K,\CC_1]$ is positive definite by Proposition~\ref{prop:equivalence}.
 We multiply equation~\eqref{eq:EVP} with $i\K$ and take the inner product with $\phi$ such that
 \[
  \overline{\lambda} \langle i\K\phi,\phi\rangle = \langle i\K\CC\phi,\phi\rangle \ .
 \]
 Its real part satisfies
 \begin{equation} \label{eq:Re:lambda:2}
  2\Re (\lambda) \langle i\K\phi,\phi\rangle = \langle(\CC_1 \K-\K\CC_1)\phi,\phi\rangle - i\langle(\CC_2 \K+\K\CC_2)\phi,\phi\rangle \ , 
 \end{equation}
 since $\CC_1$, $\CC_2$ and $i\K$ are Hermitian matrices.
 Moreover,
 \begin{multline} \label{est:1}
  2\Re (\langle i\K\CC_2 \phi,\phi\rangle)
   = \langle(\CC_2 \ i\K+ i\K\CC_2)\phi,\phi\rangle
   = \langle \sqrt{\CC_2}\ i\K\phi,\sqrt{\CC_2}\phi\rangle +\langle \sqrt{\CC_2}\phi,\sqrt{\CC_2}\ i\K\phi\rangle \\
	 \leq 2\norm{\sqrt{\CC_2}\phi} \norm{\sqrt{\CC_2}\ i\K\phi}
   \leq 2M \norm{\sqrt{\CC_2}\phi} \norm{\phi}
	 \leq \epsilon \norm{\phi}^2 +\frac{M^2}{\epsilon} \langle\CC_2 \phi,\phi\rangle
 \end{multline}
 for any positive $\epsilon$. Here we used $M:=\|\sqrt{\CC_2}\ i\K\|$ and $\norm{\sqrt{\CC_2}\phi}^2 = \langle\CC_2 \phi,\phi\rangle$ since $\CC_2\geq 0$.
 Combining equations~\eqref{eq:Re:lambda} and~\eqref{eq:Re:lambda:2} as $2\cdot$\eqref{eq:Re:lambda}$-\alpha\cdot$\eqref{eq:Re:lambda:2}
  for some constant $\alpha>0$ to be chosen later, we derive
 \begin{equation}\label{eq1}
  2\Re (\lambda) \big(\norm{\phi}^2 -\alpha \langle i\K\phi,\phi\rangle \big)
	  =\langle (\CC_2 +\alpha (\K\CC_1 -\CC_1 \K))\phi,\phi\rangle +\langle\CC_2 \phi,\phi\rangle
		 +i\alpha \langle(\CC_2 \K+\K\CC_2)\phi,\phi\rangle \ .
 \end{equation}
 There exists $\alpha_0>0$ such that $\Phi_\alpha := \norm{\phi}^2 -\alpha \langle i\K\phi,\phi\rangle$
 satisfies 
 \begin{equation}\label{eq2}
   \norm{\phi}^2/2\leq \Phi_\alpha\leq 2\norm{\phi}^2\quad \forall\; \alpha\in(-\alpha_0,\alpha_0)\,,
 \end{equation}
 since $i\K$ is a Hermitian matrix.
 Recall that the skew-Hermitian matrix~$\K$ was chosen such that $\CC_2 +[\K,\CC_1]$ is positive definite by Proposition~\ref{prop:equivalence}.
 Therefore, the estimate 
 \begin{equation}\label{eq3}
   \langle (\CC_2 +\alpha (\K\CC_1-\CC_1 \K))\phi,\phi\rangle \geq \alpha m \norm{\phi}^2 
 \end{equation}
 holds for all $\alpha\in[0,1]$, where $m>0$ is the smallest eigenvalue of the positive definite Hermitian matrix $\CC_2 +(\K\CC_1-\CC_1 \K)$.
 Thus we deduce from \eqref{eq1} and the estimates \eqref{eq3}, \eqref{est:1} that
 \begin{equation*}
  2\Re (\lambda) \Phi_\alpha \geq \alpha (m -\epsilon) \norm{\phi}^2 +(1 -\alpha\frac{M^2}{\epsilon}) \langle\CC_2 \phi,\phi\rangle \ .
 \end{equation*}
 Choosing $\epsilon =m/2$ and $\alpha =\min\{1,\alpha_0,\epsilon/M^2\}$, we finally derive with \eqref{eq2}
 \[ \Re (\lambda) \geq \frac{\alpha m}8 >0 \ . \]

 Finally, we show the reverse implication via a proof of its negation.
 If condition~\ref{cond:eigenvector} does not hold, then there exists a $\phi\in\ker \CC_2$ and an (eigenvalue) $\mu\in\R$ such that $\CC_1 \phi =\mu \phi$.
 This implies $(i\CC_1 +\CC_2)\phi =i\mu \phi$. Thus $\phi$ is an eigenvector of $\CC :=i\CC_1 +\CC_2$ for the purely imaginary eigenvalue $i\mu$.
 Thus not all eigenvalues $\lambda_{\CC}$ of $\CC$ have positive real part. 
 
 We conclude that, if all eigenvalues $\lambda_{\CC}$ of $\CC$ have positive real part $\Re (\lambda_{\CC}) >0$, then condition~\ref{cond:eigenvector} -- and equivalently~\eqref{hypocoercive:2} -- must hold.
\end{proof}

\begin{remark}\label{kdep-estimate}
In the study of hypocoercivity for discrete velocity BGK models,
a family of matrices $\CC^{(k)} := ik\ \CC_1 +\CC_2$ $(k\in\N)$ 
 for some real symmetric matrices $\CC_1, \CC_2\in \R^{n\times n}$ with $\CC_2\geq 0$
 has to be considered, see~\cite[\S4.1-\S4.2]{AAC16}.
Following the proof of~\cite[Prop. 2.4]{UmKaSh84},
 a uniform bound for the real parts of the eigenvalues $\lambda_{\CC^{(k)}}$ of these matrices $\CC^{(k)}$ ($k\in\N$) can be proven:
 \[ \Re (\lambda_{\CC^{(k)}}) \geq \frac{\alpha m}8 \frac{k^2}{1+k^2} >0 \qquad \forall k\in\N \ . \]
\end{remark}

\begin{remark}
Next we relate our study of equation~\eqref{matrix-ODE} to the one of $\ddt f + \LL f =0$ in~\cite{ViH06}.
In the first part of~\cite{ViH06}, operators $\LL =\A^* \A + \B$ with a skew-symmetric operator $\B$ are considered.
Our operator/matrix $\CC = i\CC_1 +\CC_2$ (for some Hermitian matrices $\CC_1,\CC_2 \in\C^{n\times n}$ with $\CC_2\geq 0$)
 is of the form $\LL =\A^* \A + \B$ for the choice $\A =\sqrt{\CC_2}$ and $\B=i\CC_1$ acting on the complex Hilbert space $\C^n$.
First, we notice that $\cK :=\ker \LL =\ker\A \cap \ker\B$, see~\cite[Prop. I.2]{ViH06}.
There, the study of hypocoercivity is based on the assumptions \cite[(3.4)--(3.5)]{ViH06}:
\begin{equation} \label{cond:Villani:3.4}
 \exists \tau\in\N_0 \ : \quad \ker\big(\sum_{k=0}^\tau \DD_k^* \DD_k \big) =\ker \LL =:\cK \ ,
\end{equation}
or more clearly, 
\begin{equation} \label{cond:Villani:3.5}
 \exists \tau\in\N_0 \ : \quad \sum_{k=0}^\tau \DD_k^* \DD_k \quad\text{is coercive on } \cK^\perp \ ,
\end{equation}
where the iterated commutators $\DD_k$ ($k\in\N_0$) are defined recursively as
\begin{equation*} 
 \DD_0 :=\A \ , \qquad \DD_k := [\DD_{k-1},\B] = \DD_{k-1}\B -\B\DD_{k-1} \ , \quad k\in\N \ .
\end{equation*}
In~\cite[Remark 17]{ViH06} it is noted (without a proof) that on finite dimensional Hilbert spaces,
 condition~\eqref{cond:Villani:3.5} is equivalent to~\ref{cond:invariance} in Proposition~\ref{prop:equivalence} (with credit to Denis Serre).
\end{remark}

The following simple example shows that this ``equivalence'' needs a small modification in complex Hilbert spaces:
Consider the matrices
\begin{equation*}
\A = \begin{pmatrix} 1 & 0 \\ 0 & 0 \end{pmatrix}, \quad
\B = i\A = \begin{pmatrix} i & 0 \\ 0 & 0 \end{pmatrix}.
\end{equation*}
Matrix $\A$ has kernel $\ker\A = \spn \{\binom{0}{1}\}$.
Moreover $\DD_0 =\A$ and $\DD_k = \Null$ for all $k\in\N$.
Hence, $\cK =\ker\A\cap\ker\B =\ker\A$ and conditions~\eqref{cond:Villani:3.4} and~\eqref{cond:Villani:3.5} are satisfied for all $\tau\in\N_0$. But \ref{cond:invariance}  does \emph{not} hold.

Now we give a proof of a slightly modified equivalence. 
On finite dimensional Hilbert spaces, conditions~\eqref{cond:Villani:3.4} and \eqref{cond:Villani:3.5} are obviously equivalent.
Moreover, we will make use of Proposition~\ref{prop:equivalence}
 and only show the equivalence of~\ref{cond:fullrank} and a modified~\eqref{cond:Villani:3.4}:
\begin{lemma}
Let the matrices $\CC_1$ and $\CC_2\ge0$ be Hermitian and define $\A: =\sqrt{\CC_2}$ and $\B:=i\CC_1$. Then $(\A,\B)$ satisfies
 \begin{equation} \label{cond:Villani:3.4:modified} \text{condition~\eqref{cond:Villani:3.4} together with $\ker\A \cap\ker\B =\{0\}$} \end{equation}
 if and only if $(\CC_1,\CC_2)$ satisfies~\ref{cond:fullrank}.
Moreover, the smallest possible $\tau$ in \eqref{cond:Villani:3.4:modified} and \ref{cond:fullrank} coincides.
\end{lemma}
\begin{proof}
First, notice that for all $\tau\in\N_0$
\begin{equation} \label{eq4}
  \bigcap_{k=0}^\tau \ker \DD_k = \ker \sum_{k=0}^\tau \DD_k^* \DD_k \ . 
\end{equation}
Defining $\cK':=\bigcap_{k\geq 0} \ker \DD_k$, the inclusion $\cK \subset \cK'$ is proven in~\cite[Prop I.15]{ViH06}.
Next we prove that 
\begin{equation} \label{Villani+fullrank}
 w\in \bigcap_{k=0}^\tau \ker \DD_k =: \cK'_\tau \qquad \text{is equivalent to} \qquad  \A \B^k w = 0 \quad \forall k\in\{0,\ldots,\tau\} 
\end{equation}
by induction:
For $\tau=0$, $w\in \ker \DD_0 =\ker \A$ holds.
Assume now condition~\eqref{Villani+fullrank} for $\tau$ and prove it for $\tau+1$.
Operator $\DD_{\tau+1}$ is defined as $\DD_{\tau+1} = [\DD_\tau,\B] = \DD_\tau \B -\B \DD_\tau$
 and using $w\in\bigcap_{k=0}^{\tau+1} \ker \DD_k$ yields
 \begin{equation*} 
 \begin{split}
    0 &=\DD_{\tau+1}w = [\DD_\tau,\B]w = (\DD_\tau \B -\B \DD_\tau)w = \DD_\tau \B w \\
      &=(\DD_{\tau-1}\B -\B\DD_{\tau-1})\B w = \DD_{\tau-1} \B^2 w - \B^2 \DD_{\tau-1} w
			=\DD_{\tau-1}\B^2 w = \ldots = \DD_0 \B^{\tau+1} w = \A \B^{\tau+1} w \ .
 \end{split}
 \end{equation*}
The converse, $0=\A \B^{\tau+1} w=\DD_{\tau+1}w$, is proven similarly.
Thus the equivalence~\eqref{Villani+fullrank} holds.

Finally we prove the equivalence of~\ref{cond:fullrank} and~\eqref{cond:Villani:3.4:modified}: 
If condition~\ref{cond:fullrank} holds for one $\tau_0$, then $\A \B^k w = 0$ for all $k\in\{0,\ldots,\tau_0\}$ implies $w=0$.
Due to the equivalence in~\eqref{Villani+fullrank}, $\cK'\subset\cK'_{\tau_0} =\{0\}$.
Hence, $\{0\}\subset\cK\subset\cK'_{\tau_0}=\{0\}$. With \eqref{eq4} this proves condition~\eqref{cond:Villani:3.4} with $\tau=\tau_0$ and $\ker\LL=\ker\A \cap\ker\B =\{0\}$.

If condition~\eqref{cond:Villani:3.4} holds together with $\ker\A \cap\ker\B =\{0\}$,
 then $\cK =\ker\LL =\ker\A \cap\ker\B =\{0\}$ and
 there exists $\tau\in\N_0$ such that $\ker\big(\sum_{k=0}^\tau \DD_k^* \DD_k \big) =\ker \LL =\cK$.
Due to~\eqref{eq4},
\[ \bigcap_{k=0}^\tau \ker \DD_k =\ker \sum_{k=0}^\tau \DD_k^* \DD_k =\ker \LL =\cK =\{0\} \ . \]
{}From the equivalence~\eqref{Villani+fullrank} we then obtain: If some $w\in\C^n$ satisfies $\A\B^k w=0$ for all $k\in\{0,\ldots,\tau\}$, it follows that $w=0$.
Therefore condition~\ref{cond:fullrank} holds with the same index~$\tau$.
This finishes the proof.
\end{proof}

\subsection{Ansatz for the transformation matrix~$\P$} \label{sec:Pmatrix}

For finite dimensional matrices with non-defective eigenvalues,
an \emph{optimal} transformation matrix $\P$ (yielding the sharp spectral gap and thus the sharp decay rate) can be constructed as stated in Lemma~\ref{lemma:Pdefinition}. 
But for ``infinite matrices'' the eigenfunctions $w_j$ will not be known in general. 
Hence, an optimal matrix $\P$ cannot be obtained from formula~\eqref{simpleP1}. 
Even for finite dimensional systems with $n$ large, it may not be possible to explicitly construct the matrix~$\P$ defined in~\eqref{simpleP1}.
However, Lemma~\ref{lemma:Pdefinition} still provides a guide to the construction of a non-optimal choice of $\P$ that can still be used to prove hypocoercivity and to give a quantitative decay rate.
We shall exploit this in \S\ref{sec:linBGK:1D}--\ref{sec:6} to prove hypocoercivity for BGK equations. To this end we shall only consider \emph{minimal} matrices~$\P$,
 i.e.\ matrices with a minimal number of non-zero entries in $\P-\II$, such that Lemma \ref{lemma:Pdefinition} still allows to deduce hypocoercivity (but then with a suboptimal rate $\mu$). 

Our focus will be to find a usable and simple ansatz 
for $\P$ and to prove that such an ansatz will give rise to a matrix inequality of the form \eqref{matrixestimate1}. The structure of these ansatzes shall be derived from the \emph{connectivity 
structure} of the matrix $\CC$: We consider examples of equations \eqref{matrix-ODE},
where we assume w.l.o.g.\ that the Hermitian matrix $\CC_2$ is diagonal and hence real.
Next we consider how the zero and negative diagonal elements of $-\CC_2$
(or equivalently the non-decaying and decaying eigenmodes of $\ddt f  = -\CC_2 f$) are coupled via a (non-zero) off-diagonal pair in the Hermitian matrix $\CC_1$. 
More precisely, a non-zero off-diagonal element of $\CC_1$ at $j,k$ (and hence also at $k,j$) couples, in the evolution equation, the $j$-th mode of $\CC_2$ to its $k$-th mode (or diagonal element).
In the sequel we shall use a simple graphical representation of such connections: there the dots $\circ$ and $\bullet$ represent, respectively, zero and negative diagonal elements of $-\CC_2$, and an arrow between such dots represents their connection (or coupling). 

For each zero element in the diagonal of $\CC_2$, we next consider a \emph{shortest connection graph} to a non-zero element in $\diag(\CC_2)$ -- realized by a sequence of non-zero off-diagonal elements of $\CC_1$. 
This leads to a guideline to find a simple ansatz for a minimal transformation matrix of the form $\P=\II+\A$: The ansatz parameters of the Hermitian matrix $\A\in\C^{n\times n}$ should be put at the positions of the non-zero off-diagonal coupling elements of $\CC_1$ that are needed to establish the shortest connection graphs -- choosing only one graph per zero element in $\diag(\CC_2)$.

Next we shall list some hypocoercive cases with low dimensionality of $\ker\CC_2$, because these are the most important cases in kinetic equations (as discussed in \S\ref{sec:linBGK:1D}--\ref{sec:6}). For those cases we shall then prove that the above mentioned ansatzes indeed allow to establish a spectral gap of $\CC$.

\subsubsection{Hypocoercive matrix with $\dim(\ker\CC_2)=1$} 
In this situation there exists only one (structurally relevant) case.
For \eqref{matrix-ODE} to be hypocoercive, the only zero element of the diagonal of $\CC_2$ (w.l.o.g.\ say with index $j=1$) needs to be coupled (via $\CC_1$) to a positive element of the diagonal of $\CC_2$ . \newline
Due to our assumptions,
 \[ \CC_2=\diag\{0,c_2,\ldots,c_n\} \quad \text{with }c_j>0;\,j=2,...,n;\quad \text{and} \quad \CC_1=(c_{j,k})_{j,k\in\{1,\ldots,n\}} \,. \]
The matrix $\CC=i\CC_1 +\CC_2$ is hypocoercive if and only if~\ref{cond:invariance} holds.
Since $\ker\CC_2 =\spn\{e_1\}$, Condition \ref{cond:invariance} reads here $\CC_1 e_1\not\in \spn\{e_1\}$.
Thus, we conclude from $\CC_1 e_1 =(c_{1,1},\ldots,c_{n,1})^\top$ that $c_{j,1}\ne 0$ for some $j\in\{2,\ldots,n\}$.
Of course, $j$ does not have to be unique, but we now fix one such index~$j_0$. 
This means that $c_{1,j_0}= \overline{c_{j_0,1}}\ne 0$.
In this case the hypocoercivity index is always~$1$,
since Remark \ref{rem:hypo:cond:B1'}(b) yields here that the hypocoercivity index is less or equal $\dim(\ker \CC_2)=1$.

W.l.o.g.\ we assume $j_0=2$.
The coupling within the relevant $2\times2$-subspace (i.e.\ the upper left $2\times2$ block of the matrix $\CC$) can then be symbolized as $\circ\!\!\longrightarrow\!\!\bullet$ .
Such an example was analyzed in \S4.3 of \cite{AAC16} (representing a linear BGK equation in 1D) using a transformation matrix with the ansatz
\begin{equation}\label{ansatz:Pmatrix:1D}
  \sbox0{$\begin{matrix}0 & \lambda \\ \bar\lambda & 0\end{matrix}$}
  \P=\II + \left(\begin{array}{c|c}
                   \usebox{0} & {\bf 0}\\
                   \hline
                   {\bf 0 }& {\bf 0}
  \end{array}\right)\ ,
\end{equation}
for some $\lambda\in\C$. 
Here, $\P$ and $\II$ are square matrices of the same size as $\CC$, possibly even infinite. 
The second matrix on the r.h.s.\ has the same size, but only its upper left $2\times2$ block is non-zero. 

While the above transformation matrix $\P$ is not optimal, this approach is important in practice: 
in theory, Lemma \ref{lemma:Pdefinition} provides the optimal transformation matrix $\P$ to deduce the optimal ODE-decay \eqref{ODE-decay} or \eqref{defective-decay}. But in practice, its computation is tedious, particularly when the system matrix involves a parameter, which is the case for the BGK-models to be analyzed below (cf.\ Remark \ref{kdep-estimate}). For large systems, there is therefore a need to design a method that does not require all eigenvectors, even if the resulting decay rates are then sub-optimal.
For the case $\dim(\ker \CC_2)=1$, an approximate transformation matrix $\P$ of the simple structure \eqref{ansatz:Pmatrix:1D} is sufficient, and it \emph{always} allows to prove an explicit exponential decay of the ODE \eqref{matrix-ODE}: the following theorem shows that $\CC$ and $\P$ satisfy a matrix inequality of form \eqref{matrixestimate1}, but not necessarily with the optimal constant $\mu$. Moreover, it shows that the ansatz \eqref{ansatz:Pmatrix:1D} from \S4.3 of \cite{AAC16} was not a ``wild guess'' but rather a systematic approach.\\

\begin{theorem}\label{lemma:ansatzP:1D}
 Let $\CC_1$ and $\CC_2$ be Hermitian matrices with $\CC_2\ge0$, $\dim(\ker \CC_2)=1$ such that $\CC:=i\CC_1 +\CC_2$ is hypocoercive.
 For $|\lambda|<1$ the Hermitian matrix~$\P$ in~\eqref{ansatz:Pmatrix:1D} is positive definite.
 If a sufficiently small $\lambda\in\C$ is chosen such that $\Im(\overline{\lambda}c_{1,2})>0$, then the Hermitian matrix $\CC^*\P+\P\CC$ is also positive definite.
\end{theorem}
\begin{proof}
 We set $\P=\II+r\A$ with
 \[ \lambda =r e^{i\phi} \quad \text{and} \quad \sbox0{$\begin{matrix}0 & e^{i\phi} \\ e^{-i\phi} & 0\end{matrix}$}
  \A=  \left(\begin{array}{c|c}
                   \usebox{0} & {\bf 0}\\
                   \hline
                   {\bf 0 }& {\bf 0}
  \end{array}\right)\ .\]
 Then we consider $\CC^*\P+\P\CC = 2\CC_2 +r(\CC^*\A+\A\CC)$ 
 as a perturbation of the matrix $2\CC_2$ for sufficiently small $r\geq 0$.
 In particular, zero is a simple eigenvalue of $\CC_2$ with eigenvector~$e_1$.
 For small $r\geq 0$, the eigenvalues of $\CC^*\P+\P\CC$ are close to the eigenvalues of $2\CC_2$.
 Therefore, we only need to study the evolution of the zero eigenvalue w.r.t.\ $r$.
 Due to~\cite[Thm. 6.3.12]{HoJo13}, the lowest eigenvalue~$\mu(r)$ is a continuous function satisfying $\lim_{r\to 0} \mu(r)=0$.
 Moreover, it is differentiable at $r=0$ with
 \begin{multline*}
  \frac{\d[\mu]}{\d[r]} \big|_{r=0} = \frac{e_1^* (\CC^*\A+\A\CC) e_1}{e_1^* e_1} = (\CC^*\A+\A\CC)_{1,1}
   = -i e^{-i\phi} c_{1,2} +i (e^{-i\phi} c_{1,2})^* = 2\Im (e^{-i\phi} c_{1,2}) \,.
 \end{multline*}
 Due to our assumptions, $c_{1,2}\ne 0$. Hence, we can choose $\phi$ such that $\Im (e^{-i\phi} c_{1,2})$ is positive.
 For such a choice, the smallest eigenvalue $\mu(r)$ of $2\CC_2 +r(\CC^*\A+\A\CC)$ will be positive.
 This finishes the proof.
\end{proof}


\subsubsection{Hypocoercive matrix with $\dim(\ker\CC_2)=2$}\label{sec:dim-ker2}

Up to a change in basis of $\C^n$, we consider the Hermitian matrices
 \begin{equation} \label{setup:DimKerC2:2}
    \CC_2=\diag\{0,0,c_3,\ldots,c_n\}\geq 0 \quad \text{and} \quad 
    \CC_1=(c_{j,k})_{j,k\in\{1,\ldots,n\}}\in\C^{n\times n} 
 \end{equation}
such that $c_j>0$ for $j\geq 3$ and $c_{j,j}\in\R$ for all $j\in\{1,\ldots,n\}$. 
%
We only consider hypocoercive matrices $\CC=i\CC_1 +\CC_2$.
Then, $\CC_1$ cannot have a block-diagonal structure of partition size $(2,n-2)$ as, otherwise, 
the kernel of $\CC_2$ would be invariant under $\CC_1$ in contradiction to condition~\ref{cond:invariance}. 
Hence, we shall assume in the sequel w.l.o.g.\ that $c_{2,3}\ne0$.

In order to construct (later on) appropriate transformation matrices $\P$ we shall distinguish two cases
 depending on the rank of the upper right submatrix $\CC_1^{ur} =(c_{j,k})_{j\in\{1,2\},\ k\in\{3,\ldots,n\}}$ of $\CC_1$.
These cases with appropriate ansatz for the matrix~$\P$ are summarized in Table \ref{table:ansatzP}.\\
\begin{table}
\begin{align}
(2A) &\qquad 
  \sbox0{$\begin{matrix}\ast & \ast \\ \ast & \ast \end{matrix}$}
  \sbox1{$\begin{matrix}\ast & \bullet & \ast & \cdots & \ast \\ \bullet & \ast & \ast & \cdots & \ast \end{matrix}$}
  \sbox2{$\begin{matrix}\ast & \bullet \\ \bullet & \ast \\ \ast & \ast \\ \vdots & \vdots \\ \ast & \ast \end{matrix}$}
  \CC_1 = \left(\begin{array}{c|c}
                   \usebox{0}& \usebox{1}\\
                   \hline
                   \usebox{2}& \mbox{\bf *}
  \end{array}\right)\,,
&&\quad 
  \sbox0{$\begin{matrix}0 & 0 & 0 & \lambda_1 \\ 0 & 0 & \lambda_2 & 0 \\ 0 & \overline{\lambda_2} & 0 & 0 \\ \overline{\lambda_1} & 0 & 0 & 0 \end{matrix}$}
  \P=\II + \left(\begin{array}{c|c}
                   \usebox{0}& {\bf 0}\\
                   \hline
                   {\bf 0}& {\bf 0}
  \end{array}\right)\,, \label{P-ansatz-2A} 
\intertext{\hspace{19mm} where the upper right submatrix $\CC_1^{ur}\in\C^{2\times(n-2)}$ has rank 2. Here, we assume w.l.o.g. that \smallskip\newline 
           \hspace*{19mm} $|c_{1,4}c_{2,3}|\geq |c_{1,3}c_{2,4}|$ and $c_{1,4}\,c_{2,3} \ne c_{1,3}\,c_{2,4}$, such that $c_{2,3}\ne 0$ and $c_{1,4}\ne 0$.}
(2B) &\qquad 
  \sbox0{$\begin{matrix}\ast & \ast \\ \ast & \ast \end{matrix}$}
  \sbox1{$\begin{matrix}\ast & \ast & \cdots & \ast \\ \bullet & \ast & \cdots & \ast \end{matrix}$}
  \sbox2{$\begin{matrix}\ast & \bullet \\ \ast & \ast \\ \vdots & \vdots \\ \ast & \ast \end{matrix}$}
  \CC_1 = \left(\begin{array}{c|c}
                   \usebox{0}& \usebox{1}\\
                   \hline
                   \usebox{2}& \mbox{\bf *}
  \end{array}\right)\,,
&&\quad
  \sbox0{$\begin{matrix}0 & \lambda_1 & 0 \\ \overline{\lambda_1} & 0 & \lambda_2 \\ 0 & \overline{\lambda_2} & 0 \end{matrix}$}
  \P=\II +\U \left(\begin{array}{c|c} \usebox{0}& {\bf 0}\\ \hline {\bf 0}& {\bf 0} \end{array}\right) \U^* \,, \label{P-ansatz-2B}
\intertext{\hspace{19mm} where the upper right submatrix $\CC_1^{ur}\in\C^{2\times(n-2)}$ has rank 1. Again, we assume w.l.o.g. that \smallskip\newline 
           \hspace*{19mm} $c_{2,3}\ne 0$. The right choice for the unitary matrix $\U$ depends on the structure of $\CC_1$:}
(2B1) &\qquad 
  \sbox0{$\begin{matrix}\ast & \bullet \\ \bullet & \ast \end{matrix}$}
  \sbox1{$\begin{matrix}0 & 0 & \cdots & 0 \\ \bullet & \ast & \cdots & \ast \end{matrix}$}
  \sbox2{$\begin{matrix}0 & \bullet \\ 0 & \ast \\ \vdots & \vdots \\ 0 & \ast \end{matrix}$}
  \CC_1 = \left(\begin{array}{c|c}
                   \usebox{0}& \usebox{1}\\
                   \hline
                   \usebox{2}& \mbox{\bf *}
  \end{array}\right)\,,
&&\quad
  \U=\II\,, \label{U-ansatz-2B1} \\
(2B2) &\qquad 
  \sbox0{$\begin{matrix}\ast & \ast \\ \ast & \ast \end{matrix}$}
  \sbox1{$\begin{matrix}\bullet & \ast & \cdots & \ast \\ \bullet & \ast & \cdots & \ast \end{matrix}$}
  \sbox2{$\begin{matrix}\bullet & \bullet \\ \ast & \ast \\ \vdots & \vdots \\ \ast & \ast \end{matrix}$}
  \CC_1 = \left(\begin{array}{c|c}
                   \usebox{0}& \usebox{1}\\
                   \hline
                   \usebox{2}& \mbox{\bf *}
  \end{array}\right)\,,
&&\quad
  \U= \left(\begin{array}{c|c} \U^{ul} & {\bf 0}\\ \hline {\bf 0}& \II \end{array}\right) \,, \label{U-ansatz-2B2}
\intertext{\hspace{19mm} with upper left submatrix $\U^{ul}=\tfrac1{\sqrt{|c_{1,3}|^2 +|c_{2,3}|^2}}\begin{pmatrix}\overline{c_{2,3}} & c_{1,3} \\ -\overline{c_{1,3}} & c_{2,3} \end{pmatrix}$.}
  \nonumber
\end{align}
\vspace{-18mm}
\caption{We give a classification of Hermitian matrices $\CC_1$,
 such that the associated matrix~$\CC=i\CC_1 +\diag(0,0,c_2,\ldots,c_n)$ is hypocoercive.
  The restrictions on the coefficients of $\CC_1$ are depicted as $0$ if zero, $\bullet$ if non-zero, and $\ast$ if there is no restriction. 
 Furthermore, we give the corresponding two-parameter ansatz for the transformation matrix $\P=\II+\A$. 
 The guideline to construct an admissible Hermitian perturbation matrix $\A$, is to put the parameters $\lambda_j$ at the positions of the (non-zero) coupling elements of $\CC_1$.
 In case (2B2) this will be apparent after a suitable transformation, see the proof of Theorem~\ref{th:P-admissible2D}.
 }
 \label{table:ansatzP}
\end{table}

\noindent
\underline{Case 2A:} In this case the upper right submatrix $\CC_1^{ur}\in\C^{2\times(n-2)}$ has rank 2.
Its hypocoercivity index is $1$ which can be inferred from condition \ref{cond:fullrank}: Using 
\[
  \CC_1 \sqrt{\CC_2} = \Bigg(\Null,\Null,\sqrt{c_3} \begin{pmatrix} c_{1,3} \\ \vdots \\ c_{n,3} \end{pmatrix},\ldots,
                                          \sqrt{c_n} \begin{pmatrix} c_{1,n} \\ \vdots \\ c_{n,n} \end{pmatrix} \Bigg)
 \]
we see that 
 \[
  \rank\big(\sqrt{\CC_2},\CC_1 \sqrt{\CC_2}\big) =\rank\Bigg(e_3,\ldots,e_n,\sqrt{c_3} \begin{pmatrix} c_{1,3} \\ \vdots \\ c_{n,3} \end{pmatrix},\ldots,    \sqrt{c_n} \begin{pmatrix} c_{1,n} \\ \vdots \\ c_{n,n} \end{pmatrix} \Bigg) \,.
 \]
Due to $\rank \CC_1^{ur}=2$, we have $\rank\big(\sqrt{\CC_2},\CC_1 \sqrt{\CC_2}\big)=n$. Hence, the hypocoercivity index of $\CC$ is 1. 
Such an example (a linearized BGK equation in 1D) was analyzed in \S4.4 of \cite{AAC16} using a transformation matrix with ansatz~\eqref{P-ansatz-2A}.

Up to a renumbering of the indices $\{j\ge3\}$, we assume $c_{1,4}\,c_{2,3} \ne c_{1,3}\,c_{2,4}$.
Moreover, up to a renumbering of the indices $j\in\{3,4\}$, we assume $|c_{1,4}c_{2,3}|\geq |c_{1,3}c_{2,4}|$ such that $c_{1,4}\ne 0$ and $c_{2,3}\ne 0$.
Thus, w.l.o.g.\ we assume that the zero in the diagonal of $\CC_2$ at $j=1$ is connected to $j=4$, and the zero at $j=2$ is connected to $j=3$.

The two zeros in the diagonal of $\CC_2$ are connected (via $\CC_1$) to two \emph{different} positive entries in the diagonal of $\CC_2$, i.e.\ to two decaying modes (and possibly, in addition, also to the same). Hence, this case can occur only for $n\ge4$. Here, the two connections in the relevant (upper left) $4\times4$-subspace can be symbolized as $\circ\!\!\longrightarrow\!\!\bullet\;\;\circ\!\!\longrightarrow\!\!\bullet$\,. 
\medskip

\noindent
\underline{Case 2B:} 
In this case the upper right submatrix $\CC_1^{ur}\in\C^{2\times(n-2)}$ has rank 1.
Then $\rank\big(\sqrt{\CC_2},\CC_1 \sqrt{\CC_2}\big)=n-1$.
Hence, the hypocoercivity index of $\CC$ is 2 since it is bounded from above by $\dim(\ker \CC_2)=2$,
 see Remark~\ref{rem:hypo:cond:B1'}(b).

\begin{lemma}
 Let $\CC_1$ be a Hermitian matrix whose upper right submatrix $\CC_1^{ur}\in\C^{2\times(n-2)}$ has rank 1,
  and let $\CC_2$ be a positive semi-definite Hermitian matrix with $\dim(\ker \CC_2)=2$.
 Up to a change of basis,
  the Hermitian matrices $\CC_1$ and $\CC_2$ satisfy \eqref{setup:DimKerC2:2} with $c_{2,3} =\overline{c_{3,2}} \ne 0$.
 Then,
  the matrix $\CC:=i\CC_1 +\CC_2$ is hypocoercive if and only if 
  \begin{equation} \label{2B:condition:hypo}
   c_{1,3}\ c_{2,3}\ (c_{1,1} -c_{2,2}) -c_{1,3}^2\ c_{2,1} +c_{2,3}^2\ c_{1,2} \ne 0 \,.
  \end{equation}
\end{lemma}
\begin{proof}
 Up to a change of basis, 
  the Hermitian matrices $\CC_1$ and $\CC_2$ satisfy \eqref{setup:DimKerC2:2}.
 The upper right submatrix $\CC_1^{ur}\in\C^{2\times(n-2)}$ has rank 1,
  therefore at least one coefficient of $\CC_1^{ur}$ is non-zero.
 Another change of basis
  moves this non-zero coefficient to position $(2,3)$,
  hence, w.l.o.g. let $c_{2,3} =\overline{c_{3,2}} \ne 0$.
 To prove that condition~\eqref{2B:condition:hypo} is necessary and sufficient,
  we use the characterization in Proposition~\ref{prop:equivalence}.
 Condition~\ref{cond:eigenvector} for one-dimensional subspaces of $\ker \CC_2$ reads
 \[ \forall \,(\alpha,\beta)\in\C^2\setminus\{(0,0)\}: \qquad \CC_1 (\alpha e_1 +\beta e_2)
     =\alpha \begin{pmatrix} c_{1,1} \\ c_{2,1} \\ \vdots \\ c_{n,1} \end{pmatrix}
      +\beta \begin{pmatrix} c_{1,2} \\ c_{2,2} \\ \vdots \\ c_{n,2} \end{pmatrix}
  \notin \spn\{\alpha e_1 +\beta e_2\} \,.
 \]
 This is equivalent to the following condition:
 \begin{align} 
  \text{For all $(\alpha,\beta)\in\C^2\setminus\{(0,0)\}$, } &
     (\alpha c_{1,1} +\beta c_{1,2})\beta \ne (\alpha c_{2,1} +\beta c_{2,2})\alpha \label{HC-cond:1} \\ 
    &\text{or } \exists j\in\{3,\ldots,n\}\,:\ \alpha c_{j,1} +\beta c_{j,2} \ne 0 \text{ holds.} \label{HC-cond:2}
 \end{align}
 Due to the assumption~$\rank \CC_1^{ur} =1$, 
  there exists a unique $\gamma\in\C$ (namely $\gamma = -c_{3,1} /c_{3,2}$, since $c_{2,3} =\overline{c_{3,2}} \ne 0$) such that $c_{j,1} +\gamma\ c_{j,2} = 0$ for all $j\in\{3,\ldots,n\}$.
 Therefore, the second condition~\eqref{HC-cond:2} holds if and only if $\beta\ne \alpha\gamma$.
 If $\beta =\alpha\gamma$ then the first condition~\eqref{HC-cond:1} has to hold.
 Inserting $\beta =\alpha\gamma$ in~\eqref{HC-cond:1} yields 
\begin{equation}\label{alpha-beta} 
  0 \ne \alpha\beta (c_{1,1} -c_{2,2}) +\beta^2 c_{1,2} -\alpha^2 c_{2,1}
     = \alpha^2 (\gamma (c_{1,1} -c_{2,2}) +\gamma^2 c_{1,2} -c_{2,1}) \,.
\end{equation}
Using $\gamma = -c_{3,1} /c_{3,2}$, the r.h.s.\ of \eqref{alpha-beta} reads
$$
     \frac{\alpha^2}{c_{3,2}^2} \big(-c_{3,1}\ c_{3,2} (c_{1,1} -c_{2,2}) +c_{3,1}^2 c_{1,2} -c_{3,2}^2\ c_{2,1} \big)
     = -\frac{\alpha^2}{c_{3,2}^2} \overline{(c_{1,3}\ c_{2,3} (c_{1,1} -c_{2,2}) -c_{1,3}^2 c_{2,1} +c_{2,3}^2\ c_{1,2})} \,.
$$
 Thus, matrix $\CC$ is hypocoercive if and only if condition~\eqref{2B:condition:hypo} holds.
\end{proof}

 
This finishes the complete classification of the situation when $\dim(\ker \CC_2)=2$. 
Our ansatz for matrix~$\P$ depends on the structure of matrix~$\CC_1$.
Therefore we distinguish between the subcases (2B1) and (2B2), see also Table \ref{table:ansatzP}.
We shall prove that these ansatzes will allow for a matrix inequality of the form \eqref{matrixestimate1} and hence for an explicit exponential decay~\eqref{ODE-decay} in the ODE~\eqref{matrix-ODE}. 
As in Theorem \ref{lemma:ansatzP:1D} we shall construct $\P$ as a perturbation of $\II$. 
To verify, then, a matrix inequality of the form \eqref{matrixestimate1} we shall use the following perturbation result on multiple eigenvalues:

\begin{lemma}[Theorem II.2.3 in \cite{Kato76}] \label{lemma:ansatzP:nD}
 Let $\CC_1$ and $\CC_2$ be Hermitian matrices with $\CC_2\ge0$ and $\dim(\ker\CC_2)=k\in\N_0$,
 such that the associated matrix~$\CC=i\CC_1 +\CC_2$ is hypocoercive.
 Let $\{v_j;\,j=1,\ldots,k\}$ be an orthonormal basis of the kernel $\ker\CC_2$
 and let $\A$ be a Hermitian matrix (which makes $\P(r):=\II+r\A$ a positive definite Hermitian matrix for sufficiently small $r\geq0$).
 Then, for sufficiently small~$r>0$, 
 the $k$ lowest eigenvalues~$\mu_j(r)$ of the Hermitian matrix $\CC^*\P(r)+\P(r)\CC$ satisfy 
 \begin{equation} \label{eigenvalue:expansion}
  \mu_j(r) =r \xi_j +o(r)\ ,\quad j=1,...,k\,, 
 \end{equation}
 where $\xi_j$ are the eigenvalues of $\RR^*(\CC^*\A+\A\CC)\RR$ and $\RR:=(v_1,\ldots,v_k)\in\C^{n\times k}$. 
\end{lemma}
We will use this result to construct perturbation matrices $\A$ and to check the admissibility of the various ansatzes for the transformation matrices~$\P$ -- mostly in the case $\dim(\ker\CC_2)=2$. The two matrices in Lemma \ref{lemma:ansatzP:nD} are related via
\begin{equation} \label{matrix-rel}
 \CC^*\P(r)+\P(r)\CC =\CC^*(\II+r\A)+(\II+r\A)\CC =2\CC_2 +r(\CC^*\A+\A\CC)\,, 
\end{equation}
and $\CC_2$ has a $k$-fold 0-eigenvalue by assumption. Now, if $\A$ is chosen such that all eigenvalues $\xi_j,\,j=1,...,k$ in \eqref{eigenvalue:expansion} are positive, then we deduce the positive definiteness of $\CC^*\P(r)+\P(r)\CC$ for sufficiently small~$r>0$.\\ 
We remark that the positivity of $\xi_1,...,\xi_k$ is first of all a sufficient condition for the positive definiteness of $\CC^*\P(r)+\P(r)\CC$ (for sufficiently small~$r>0$). But one sees easily from \eqref{matrix-rel} that it is also necessary.\\
\medskip

\begin{theorem}\label{th:P-admissible2D}
 Let $\CC_1$ and $\CC_2$ be Hermitian matrices with $\CC_2\ge0$ and $\dim(\ker\CC_2)=2$,
 such that the associated matrix~$\CC=i\CC_1 +\CC_2$ is hypocoercive.
 Then there exists a two-parameter ansatz for a positive definite matrix~$\P=\P(\lambda_1,\lambda_2)$, according to Table~\ref{table:ansatzP},
 such that $\CC^*\P+\P\CC$ is positive definite (for an appropriate choice of $\lambda_1,\lambda_2$).
\end{theorem}
\begin{proof}
 First, one easily checks that all matrices $\P$ from Table~\ref{table:ansatzP} are positive definite if $|\lambda_1|^2+|\lambda_2|^2<1$. 
Thus, $\P(r):=\II+r\A$ with $\A:=\P-\II$ yields for $r\in[0,1]$ a family of positive definite Hermitian matrices~$\P(r)$. 

Now, up to a change of basis in $\C^n$,
we assume without loss of generality that $\CC_2$ is a diagonal matrix of the form $\CC_2 =\diag(0,0,c_3,\ldots,c_n)$ with $c_j>0$.
Then, $\ker\CC_2=\spn\{e_1,e_2\}$ and we choose $\RR=(e_1,e_2)\in\R^{n\times 2}$.
According to Lemma \ref{lemma:ansatzP:nD}, the positive definiteness of $\CC^*\P+\P\CC$ (for sufficiently small~$r>0$) can be inferred from the positive definiteness of $\RR^*(\CC^*\A+\A\CC)\RR$. 

Next we deal with each case of $\CC_1$ and its corresponding ansatz $\P=\II+\A$ (as listed in Table \ref{table:ansatzP}) separately: 
we need to prove that $\lambda_1$ and $\lambda_2$ can be chosen such that $\RR^*(\CC^*\A+\A\CC)\RR$ is indeed positive definite. 
%
\begin{enumerate}[label=(2\Alph*)]
 \item 
We consider $\CC_1=(c_{j,k})_{j,k\in\{1,\ldots,n\}}$ satisfying w.l.o.g.
\begin{equation}\label{cond:2A:c}
 |c_{1,4}\ c_{2,3}|\geq |c_{1,3}\ c_{2,4}| \quad \mbox{and} \quad c_{1,4}\ c_{2,3} \ne c_{1,3}\ c_{2,4} \,,
\end{equation}
such that $c_{2,3}=\overline{c_{3,2}}\ne0$ and $c_{1,4}=\overline{c_{4,1}}\ne0$.
For
 \[ \RR^*(\CC^*\A+\A\CC)\RR = i
     \begin{pmatrix}
      - c_{1,4}\ \overline{\lambda_1} + \overline{c_{1,4}}\ \lambda_1 & +\overline{c_{2,4}}\ \lambda_1 - c_{1,3}\ \overline{\lambda_2} \\
      - c_{2,4}\ \overline{\lambda_1} + \overline{c_{1,3}}\ \lambda_2 & -c_{2,3}\ \overline{\lambda_2}+\overline{c_{2,3}}\ \lambda_2
     \end{pmatrix} \,
 \]
to be positive definite, all three of its minors have to be positive for appropriately chosen $\lambda_1$ and $\lambda_2$. We set
\begin{equation}\label{lambda-polar}
  \lambda_1 := -i \ell_1 c_{1,4}\,, \quad \lambda_2 := -i \ell_2 c_{2,3} \,,
\end{equation}
for some positive numbers $\ell_1$ and $\ell_2$. Then, the minors of first order satisfy 
\begin{align*}
  -i(c_{1,4}\ \overline{\lambda_1} -\overline{c_{1,4}}\ \lambda_1) &= 2\Im(c_{1,4}\ \overline{\lambda_1})=2\ell_1\ |c_{1,4}|^2 >0 \,, \\
  -i(c_{2,3}\ \overline{\lambda_2} -\overline{c_{2,3}}\ \lambda_2) &= 2\Im(c_{2,3}\ \overline{\lambda_2})=2\ell_2\ |c_{2,3}|^2 >0 \,.
\end{align*}
The minor of second order reads (using~\eqref{lambda-polar})
\begin{align*}
 \det(\RR^*(\CC^*\A &+\A\CC)\RR)
   = 4\ell_1 \ell_2 |c_{1,4}|^2 |c_{2,3}|^2 -|\ell_1\ c_{2,4}\ \overline{c_{1,4}} +\ell_2\ \overline{c_{1,3}}\ c_{2,3}|^2 \\
 &\quad = 4\ell_1 \ell_2 |c_{1,4}|^2 |c_{2,3}|^2 -|\ell_1\ c_{2,4}\ \overline{c_{1,4}}|^2 -|\ell_2\ \overline{c_{1,3}}\ c_{2,3}|^2 \\
  &\qquad -\ell_1 \ell_2\ c_{2,4}\ \overline{c_{1,4}}\ c_{1,3}\ \overline{c_{2,3}} -\ell_1 \ell_2\ \overline{c_{2,4}}\ c_{1,4}\ \overline{c_{1,3}}\ c_{2,3} \\
 &\quad = -\big(\ell_1 |c_{1,4}\ c_{2,4}| -\ell_2 |c_{1,3}\ c_{2,3}|\big)^2 \\
  &\qquad +\ell_1 \ell_2 \Big[ 4|c_{1,4}\ c_{2,3}|^2 -2 |c_{1,4}\ c_{2,4}\ c_{1,3}\ c_{2,3}|
                                -c_{2,4}\ \overline{c_{1,4}}\ c_{1,3}\ \overline{c_{2,3}} -\overline{c_{2,4}}\ c_{1,4}\ \overline{c_{1,3}}\ c_{2,3} \Big] \\
 &\quad = -\big(\ell_1 |c_{1,4}\ c_{2,4}| -\ell_2 |c_{1,3}\ c_{2,3}|\big)^2 \\
  &\qquad +\ell_1 \ell_2 \Big[ (3 |c_{1,4}\ c_{2,3}| +|c_{1,3}\ c_{2,4}|) (|c_{1,4}\ c_{2,3}| -|c_{1,3}\ c_{2,4}|)
                                +|c_{1,4}\ c_{2,3} -c_{1,3} c_{2,4}|^2 \Big] \,.
\end{align*} 
Then, the minor of second order is positive for the choice $\ell_1 =\epsilon |c_{1,3}\ c_{2,3}|$ and $\ell_2 =\epsilon |c_{1,4}\ c_{2,4}|$ with any $\epsilon>0$, 
 due to our assumption~\eqref{cond:2A:c}.
Finally, for sufficiently small $\epsilon>0$ the Hermitian matrix~$\P$ is positive definite.

\item \label{case:2B}
First, we verify that the ansatz for $\P$ in~\eqref{P-ansatz-2B} is admissible in case (2B1).

\noindent
\underline{In case (2B1)}, we consider w.l.o.g.
 \[ \CC_1=(c_{j,k})_{j,k\in\{1,\ldots,n\}} \quad \text{with} \quad
     c_{1,2}=\overline{c_{2,1}}\ne0\,,\ 
     c_{2,3}=\overline{c_{3,2}}\ne0\,,\ 
     c_{1,3}=\overline{c_{3,1}}=0 \,. 
 \]
Then, the connections in the relevant (upper left) $3\times3$-subspace can be symbolized as $\circ\!\!\longrightarrow\!\!\circ\!\!\longrightarrow\!\!\bullet$\,; see \eqref{U-ansatz-2B1}. 
To prove that the ansatz for $\P$ in~\eqref{P-ansatz-2B} with $\U=\II$ 
is admissible, we use Lemma~\ref{lemma:ansatzP:nD} and we need to check the positive definiteness of 
\begin{equation}\label{RCAR}
 \RR^*(\CC^*\A+\A\CC)\RR = i
     \begin{pmatrix}
      -c_{1,2}\overline{\lambda_1} +\overline{ c_{1,2}}\lambda_1 & (c_{2,2} - c_{1,1})\lambda_1 \\
       (c_{1,1} -c_{2,2})\overline{\lambda_1} & c_{1,2}\overline{\lambda_1} -\overline{ c_{1,2}}\lambda_1 -c_{2,3}\overline{\lambda_2} +\overline{ c_{2,3}}\lambda_2
     \end{pmatrix}
\end{equation}
for appropriately chosen $\lambda_1$ and $\lambda_2$. 
The minors of first order are 
\[ -i( c_{1,2}\overline{\lambda_1}-\overline{ c_{1,2}}\lambda_1) =2\Im( c_{1,2}\overline{\lambda_1}) \quad \text{and} \quad
   i( c_{1,2}\overline{\lambda_1}-\overline{ c_{1,2}}\lambda_1- c_{2,3}\overline{\lambda_2} +\overline{ c_{2,3}}\lambda_2) =-2\Im( c_{1,2}\overline{\lambda_1})+2\Im( c_{2,3}\overline{\lambda_2}) \,.
\]
They are positive if and only if 
\begin{equation}\label{1st-minor-cond}
  0<\Im( c_{1,2}\overline{\lambda_1})<\Im( c_{2,3}\overline{\lambda_2}) \,. 
\end{equation}
Due to our assumptions $ c_{1,2}\ne0$ and $ c_{2,3}\ne0$,
we can choose $\lambda_1$ and $\lambda_2$ such that this condition is satisfied.
The minor of second order reads
\[ \det(\RR^*(\CC^*\A+\A\CC)\RR)
    = 4\Im( c_{1,2}\overline{\lambda_1}) \big(\Im( c_{2,3}\overline{\lambda_2})-\Im( c_{1,2}\overline{\lambda_1})\big) -|c_{1,1} -c_{2,2}|^2 |\lambda_1|^2 \,,
\] 
where the first summand is positive due to \eqref{1st-minor-cond}.
First we choose $\lambda_1$ and $\lambda_2$ such that the minors of first order are positive.
Then we consider $r \lambda_1$ for $r\in(0,1)$ instead of $\lambda_1$, and we choose $r\in(0,1)$ sufficiently small such that the second minor
becomes positive, and hence \eqref{RCAR} is positive definite.

\underline{In case (2B2)}, we consider w.l.o.g.
 \[ \CC_1=(c_{j,k})_{j,k\in\{1,\ldots,n\}} \quad \text{with} \quad
     c_{1,3}=\overline{c_{3,1}}\ne0\,,\ 
     c_{2,3}=\overline{c_{3,2}}\ne0\,, 
 \]
 and recall the hypocoercivity condition~\eqref{2B:condition:hypo}.
The guideline to construct a simple ansatz for~$\P$ at the beginning of this section would suggest to connect each non-decaying mode to the same decaying mode.
However, for some examples in subcase (2B2) this ansatz is not admissible. 
Therefore this guideline is not universally true.

The motivation for the (alternative) $\P$-ansatz \eqref{P-ansatz-2B} with unitary matrix~$\U$ in~\eqref{U-ansatz-2B2} is
 that the transformation $\tilde{\CC}_1 =\U^{-1} \CC_1 \U$ yields a matrix of form (2B1)
 with $\tilde c_{1,j}=0$ for $j\ge3$ (since $\rank(\CC_1^{ur})=1$), $\tilde c_{2,3}=1$ and
 \[
  \big( \widetilde{\CC}_1 \big)_{1,2} = \tilde{c}_{1,2} = \tfrac1{|c_{1,3}|^2 +|c_{2,3}|^2} \big( (c_{1,1} -c_{2,2}) c_{1,3} c_{2,3} +c_{1,2} c_{2,3}^2 -\overline{c_{1,2}} c_{1,3}^2 \big) \ne 0 \,,
 \]
due to the hypocoercivity condition~\eqref{2B:condition:hypo}.
To prove that the ansatz for $\P$ in~\eqref{P-ansatz-2B} with $\U$ in~\eqref{U-ansatz-2B1} is admissible,
 we consider
 \[ \CC^* \P +\P\CC = \CC^{*} (\II+\U\A\U^* ) +(\II+\U\A\U^* )\CC = 2\CC_2 +\CC^{*} \U\A\U^* +\U\A\U^* \CC \,. \]
Due to Lemma~\ref{lemma:ansatzP:nD},
 we need to check the positive definiteness of $\widetilde{\RR}^*(\CC^{*} \U\A\U^* +\U\A\U^* \CC)\widetilde{\RR}$
 for appropriately chosen $\lambda_1$ and $\lambda_2$. 
Using $\widetilde{\RR} =\U\RR$, we deduce
 \begin{align*} 
 \widetilde{\RR}^*(\CC^{*} \U\A\U^* +\U\A\U^* \CC)\widetilde{\RR}
  &= i\RR^* \U^* (-\CC_1^{*} \U\A\U^* +\U\A\U^* \CC_1)\U\RR \\
  &= i\RR^* \big(-(\U^* \CC_1^{*} \U)\A +\A(\U^* \CC_1 \U)\big)\RR \,.
 \end{align*}
Recalling that $\U^* \CC_1 \U =\U^{-1} \CC_1 \U$ is of form (2B1),
 the positive definiteness of $i\RR^* \big(-(\U^* \CC_1^{*} \U)\A +\A(\U^* \CC_1 \U)\big)\RR$ for suitable $\lambda_1$ and $\lambda_2$ follows as in case (2B1).
\end{enumerate}
\end{proof}

For $\dim(\ker\CC_2)=1$ or 2, we just listed all possible cases. 
But for $\dim(\ker\CC_2)=3$ we will next only consider the one situation relevant below for the linearized BGK equation in 1D, i.e.\ \eqref{linap4}, \eqref{L1L2}.


\subsubsection{Hypocoercive matrix with $\dim(\ker\CC_2)=3$}\label{sec:dim-ker3}

If the three zeros in the diagonal of $\CC_2$ are connected (via $\CC_1$) only \emph{consecutively} to a positive entry in the diagonal of $\CC_2$, the relevant $4\times4$-subspace can be symbolized as 
$\circ\!\!\longrightarrow\!\!\circ\!\!\longrightarrow\!\!\circ\!\!\longrightarrow\!\!\bullet$. 
Proceeding as in \S\ref{sec:dim-ker2} one easily checks that
 \[
  \rank\big(\sqrt{\CC_2},\CC_1 \sqrt{\CC_2},\CC_1^2 \sqrt{\CC_2}\big) = n-1 \,,\quad
  \rank\big(\sqrt{\CC_2},\CC_1 \sqrt{\CC_2},\CC_1^2 \sqrt{\CC_2},\CC_1^3 \sqrt{\CC_2}\big) = n \,,
 \]
and hence the hypocoercivity index of $\CC$ is 3.

With $\CC_1$ of the form
\begin{equation}\label{C1-form-3}
  \sbox0{$\begin{matrix}\ast & \bullet & 0\\ \bullet & \ast & \bullet\\ 0 & \bullet & \ast\end{matrix}$}
  \sbox1{$\begin{matrix}0 & 0 & \cdots & 0 \\ 0 & 0 & \cdots & 0 \\ \bullet & \ast & \cdots & \ast \end{matrix}$}
  \sbox2{$\begin{matrix}0 & 0 & \bullet \\ 0 & 0 & \ast \\ \vdots & \vdots & \vdots \\ 0 & 0 & \ast \end{matrix}$}
  \CC_1 = \left(\begin{array}{c|c}
                   \usebox{0}& \usebox{1}\\
                   \hline
                   \usebox{2}& \mbox{\bf *}
  \end{array}\right)\,,
\end{equation}
a natural ansatz for a simple transformation matrix is given by
\begin{equation}\label{P-ansatz-3}
 \sbox0{$\begin{matrix}0 & \lambda_1 & 0 & 0 \\ \overline{\lambda_1} & 0 & \lambda_2 & 0 \\ 0 & \overline{\lambda_2} & 0 & \lambda_3 \\ 0 & 0 & \overline{\lambda_3} & 0 \end{matrix}$}
  \P=\II + \left(\begin{array}{c|c}
                   \usebox{0}& {\bf 0}\\
                   \hline
                   {\bf 0}& {\bf 0}
  \end{array}\right)\ ,
\end{equation}
with some $\lambda_1,\,\lambda_2,\,\lambda_3\in\C$. 

Indeed, this ansatz always yields a useful Lyapunov functional and hence a quantitative exponential decay rate, as we shall now show under the simplifying restriction $c_{1,1}=c_{2,2}=c_{3,3}$ (which is the relevant situation in \S\ref{sec:linBGK:1D}):
\begin{theorem}\label{th:p-ansatz-3}
 Let $\CC_2=\diag(0,0,0,c_4,...,c_n)$ with $c_j>0$, and $\CC_1$ be a Hermitian matrices of form \eqref{C1-form-3} and satisfying $c_{1,1}=c_{2,2}=c_{3,3}$.
 Then there exists a three-parameter ansatz for a positive definite matrix~$\P=\P(\lambda_1,\lambda_2,\lambda_3)$ of form \eqref{P-ansatz-3},
 such that $\CC^*\P+\P\CC$ is positive definite (for an appropriate choice of $\lambda_1,\lambda_2,\lambda_3$).
\end{theorem}
\begin{proof}
First, the matrix $\P$ is positive definite if $|\lambda_{1}|^2+|\lambda_{2}|^2+|\lambda_{3}|^2<1$. 
Thus, $\P(r):=\II+r\A$ with $\A:=\P-\II$ yields for $r\in[0,1]$ a family of positive definite Hermitian matrices~$\P(r)$.

We have $\ker\CC_2=\spn\{e_1,e_2,e_3\}$ and $\RR=(e_1,e_2,e_3)\in\R^{n\times 3}$.
According to Lemma \ref{lemma:ansatzP:nD}, the positive definiteness of $\CC^*\P+\P\CC$ (for sufficiently small~$r>0$) can be inferred from the positive definiteness of $\RR^*(\CC^*\A+\A\CC)\RR$. 

As in the proof of Theorem \ref{th:P-admissible2D} we search for conditions on $\lambda_j$ ($j=1,2,3$) such that the eigenvalues of 
 \begin{equation}\label{R-matrix-3D}
  \RR^*(\CC^*\A+\A\CC)\RR =  
   \begin{pmatrix}
    2 \Im(c_{1,2}\overline{\lambda_1}) & 0 & i(c_{2,3} \lambda_1 -c_{1,2}\lambda_2) \\
    0 & -2\Im(c_{1,2}\overline{\lambda_1}-c_{2,3}\overline{\lambda_2}) & 0 \\
    i(\overline{c_{1,2}} \overline{\lambda_2} -\overline{c_{2,3}}\overline{\lambda_1}) & 0 & -2\Im(c_{2,3}\overline{\lambda_2} -c_{3,4}\overline{\lambda_3}) 
   \end{pmatrix} \,
 \end{equation}
 are positive.
If all minors are positive, then the matrix will be positive definite (by Sylvester's criterion).
{}From the three minors of first order we deduce the conditions 
\begin{equation}\label{minor1-3D}
 0<\Im(c_{1,2}\overline{\lambda_1}) <\Im(c_{2,3}\overline{\lambda_2}) <\Im(c_{3,4}\overline{\lambda_3}) \,,
\end{equation}
which also imply the positivity of the second minor, i.e. 
\[ 
 4 \Im(c_{1,2}\overline{\lambda_1}) \Im(c_{2,3}\overline{\lambda_2}-c_{1,2}\overline{\lambda_1}) >0 \,.
\]
To satisfy the former conditions it is convenient to choose
\begin{equation}\label{lambda-polar-3D}
  \arg(\lambda_1)=\arg(c_{1,2})-\frac{\pi}{2}\ ,\quad \arg(\lambda_2)=\arg(c_{2,3})-\frac{\pi}{2}\ ,\quad \arg(\lambda_3)=\arg(c_{3,4})-\frac{\pi}{2}\ ,
\end{equation}
just as in \eqref{lambda-polar}.
The determinant of \eqref{R-matrix-3D} reads
\begin{equation}\label{det-3D}
  2 \Im(c_{2,3}\overline{\lambda_2}-c_{1,2}\overline{\lambda_1}) 
	\left[4 \Im(c_{1,2}\overline{\lambda_1}) \Im(c_{3,4}\overline{\lambda_3}-c_{2,3}\overline{\lambda_2})
	-|c_{2,3} \lambda_1-c_{1,2}\lambda_2|^2\right]\ .
\end{equation}
Now the parameters $\lambda_j$ ($j=1,2,3$) can be chosen in analogy to the proof of Theorem \ref{th:P-admissible2D},
case \ref{case:2B} to satisfy the conditions \eqref{minor1-3D}. Once $\lambda_1$, $\lambda_2$, and $\arg(\lambda_3)$ are fixed, we can choose $|\lambda_3|$ large enough to also satisfy the positivity of \eqref{det-3D}.
\end{proof}
\medskip

This analysis to construct appropriate matrices $\P$ could, of course, also be extended to higher dimensions of $\ker\CC_2$, but this gets more cumbersome. In \S\ref{sec:linBGK:2D} and \S\ref{sec:linBGK:3D} we have $\dim(\ker\CC_2)=4$ and 5, respectively.


\section{Linearized BGK equation in 1D}
\label{sec:linBGK:1D}

\bigskip
In this section we shall analyze the large time behavior of the linearized BGK equation~\eqref{linBGK:torus} in 1D, 
\begin{multline} \label{linBGK:torus:1D}
 \partial_t h (x,v,t) +  v\ \partial_x h(x,v,t) \\
      = M_1(v) \,\left[\left( \frac32 -\frac{v^2}{2}\right)\sigma(x,t)
                       +v \mu(x,t)
                       + \left( -\frac{1}{2}+ \frac{v^2}{2}\right)\tau(x,t)\right] -h(x,v,t)\ , \quad t\geq 0 \ ,
\end{multline}
for the perturbation $h(x,v,t) \approx f(x,v,t) -M_1(v)$. To prepare for the proof of Theorem \ref{linBGK-decay} we shall use an expansion in $v$--modes, as in \cite{AAC16}.
Using the probabilists' Hermite polynomials,
\begin{equation} \label{hermite-polynom:1D}
  H_m(v) := (-1)^m e^{\frac{v^2}{2}}\frac{{\rm d}^m}{{\rm d} v^m}e^{-\frac{v^2}{2}}\ ,\quad m\in\N_0\ ,
\end{equation}
we define the normalized Hermite functions corresponding to $T=1$:
\begin{equation} \label{hermite-fct:1D}
  g_m(v):=(2\pi m!)^{-1/2} H_m(v)\,e^{-\frac{v^2}{2}}.
\end{equation}
They satisfy
$$
  \int_\R g_m(v) g_n(v) M_1^{-1}(v)\,{\rm d} v = \delta_{mn}\,,\quad m,n\in\N_0\,,
$$
and the recurrence relation
\begin{equation}\label{recur:1D}
  v \, g_m(v) = \sqrt{m+1}\, g_{m+1}(v) + \sqrt{m} \,g_{m-1}(v)\,,\quad m\in\N\ .
\end{equation}
The first three normalized Hermite functions $g_m(v)$ are 
$$
  {g}_0(v) = M_1(v)\ , \qquad
  {g}_1(v) = v M_1(v) \quad\text{and}\quad
  {g}_2(v) = \frac{ v^2- 1}{\sqrt{2}}\ M_1(v)\ .
$$
With this notation, \eqref{linBGK:torus:1D} reads
$$
  \partial_t h(x,v,t) + v\ \partial_x h(x,v,t) = \left({g}_0(v) -\frac{1}{\sqrt2} {g}_2(v)\right)\sigma(x,t)
    + {g}_1(v)\mu(x,t)
    + \frac{1}{\sqrt2} {g}_2(v)\tau(x,t)-h(x,v,t)\ .
$$
We start with the $x$--Fourier series of $h$: 
 \begin{equation*}
  h(x,v,t) = \sum_{k\in\Z} h_k(v,t)\,e^{ik\frac{2\pi}{L}x}\ .
 \end{equation*}
Each spatial mode $h_k(v,t)$ is decoupled and evolves according to
\begin{equation} \label{linBGK:Fourier} 
 \ddt h_k +  ik\tfrac{2\pi}{L}v\ h_k = {g}_0(v)\sigma_k(t) 
      + {g}_1(v)\mu_k(t)
      + {g}_2(v)\frac{1}{\sqrt2} \left(\tau_k(t) -\sigma_k(t)\right)-h_k\ ,\quad k\in\Z;\,t\ge0\ .
\end{equation}
Here, $\sigma_k$, $\mu_k$ and $\tau_k$ denote the spatial modes of the $v$--moments $\sigma$, $\mu$ and $\tau$ defined in \eqref{mloc2}; hence
\[ \sigma_k := \int_\R h_k(v,t)\d[v]\ , \quad 
   \mu_k := \int_\R v\ h_k(v,t)\d[v]\ , \quad 
   \tau_k := \int_\R v^2\ h_k(v,t)\d[v]\ .
\]

\noindent
Next we expand $h_k(\cdot,t)\in L^2(\R;M_1^{-1})$ in the orthonormal basis $\{{g}_m(v)\}_{m\in\N_0}$:
$$
  h_k(v,t)=\sum_{m=0}^\infty \hat h_{k,m}(t)\,{g}_m(v)\ ,\quad\text{with}\quad
  \hat h_{k,m} = \langle h_k(v),{g}_m(v)\rangle_{L^2(M_1^{-1})}\ .
$$
For each $k\in\Z$,
 the ``infinite vector'' $\hat {\bf h}_k(t) = ( \hat h_{k,0}(t),\,\hat h_{k,1}(t),\ ...)^\top\in \ell^2(\N_0)$ contains all Hermite coefficients of $h_k(\cdot,t)$. 
In particular we have
$$
  \hat h_{k,0} = \int_\R h_k(v) {g}_0(v) M_1^{-1}(v) \d[v] = \sigma_k \ , \qquad 
  \hat h_{k,1} = \int_\R h_k(v) {g}_1(v) M_1^{-1}(v) \d[v] = \mu_k \ ,
$$
and
$$
  \hat h_{k,2} = \int_\R h_k(v) {g}_2(v) M_1^{-1}(v) \d[v] = \frac{1}{\sqrt2} \left(\tau_k-\sigma_k\right)\ .
$$
Hence, \eqref{linBGK:Fourier} can be written equivalently as
$$
  \ddt h_k(v,t) +  ik\tfrac{2\pi}{L} v\ h_k(v,t) = {g}_0(v) \hat h_{k,0}(t) 
      + {g}_1(v) \hat h_{k,1}(t)
      + {g}_2(v) \hat h_{k,2}(t) -h_k(v,t)\ , \quad k\in\Z\ ;\ t\ge0\ .
$$
Thus, the vector of its Hermite coefficients satisfies
\begin{equation}\label{linap4}
\ddt\hat {\bf h}_k(t) + ik \tfrac{2\pi}{L}\,\LL_1 \hat{\bf h}_k(t) = -\LL_2 \hat{\bf h}_k(t)\ ,\quad k\in\Z\ ;\ t\ge0\ ,
\end{equation}
where the operators $\LL_1,\,\LL_2$ are represented by ``infinite matrices'' on $\ell^2(\N_0)$:
\begin{equation}\label{L1L2}
  \LL_1= \begin{pmatrix}
  0 &  \sqrt 1 &  0 &  \cdots \\
  \sqrt 1 & 0 &  \sqrt 2 &  0 \\
  0 &  \sqrt 2 & 0 &  \sqrt 3 \\
  \vdots &  0 &  \sqrt 3 & \ddots 
  \end{pmatrix}
\ ,\quad \LL_2=\diag(0,\, 0,\, 0,\, 1,\, 1,\cdots)\ .
\end{equation}
\begin{remark} The bi-diagonal form of $\LL_1$ is a direct expression of the two-term recursion relation (\ref{recur:1D}). This is not special to the Hermite polynomials; a similar expression holds for the orthogonal polynomials with respect to any even reference measure.
\end{remark} 
Equation \eqref{linap4} provides a decomposition of the generator in its skew-symmetric part $-ik \tfrac{2\pi}{L}\LL_1$ and its symmetric part $-\LL_2$, the latter introducing the decay in the evolution.

We remark that \eqref{linap4} simplifies for the spatial mode $h_0$ with $k=0$. 
One easily verifies that, for all $d$, the flow of 
\eqref{linBGK:torus} preserves \eqref{mloc2},
 i.e. $\sigma_0(t)=0$, $\mu_0(t)=0$, $\tau_0(t)=0$ for all $t\ge0$. 
Hence, \eqref{linBGK:Fourier} yields 
\begin{equation}\label{h0-decay}
  \ddt h_0(v,t) = -h_0(v,t)\ ,\quad t\ge0\ .
\end{equation}

For $k\neq 0$, we note that the linearized BGK equation is very similar to the equation specified in \cite[\S 4.4]{AAC16}:
The only difference is that $\LL_2$ now has one more zero -- at the second entry on the diagonal, which corresponds to the conservation of momentum. For $k\ne0$, \eqref{linap4} has the structure of the example in \S\ref{sec:dim-ker3}, and thus hypocoercivity index 3. This has a simple interpretation: The mass-conservation mode is coupled to the momentum-conservation mode, which is coupled to the energy-conservation mode. Finally, the latter is coupled to the decreasing mode that corresponds to $g_3(v)$. The hypocoercivity index of \eqref{linap4} can also be obtained directly from Definition \ref{hyp-index}, in its equivalent formulation \ref{cond:trivkernel} that also applies to ``infinite matrices'': With $\corank \LL_2=3$, $\ker\LL_2=\spn\{e_0,\,e_1,\,e_2\}$, and the relations 
$$
  \LL_1 e_0=e_1\ ,\quad \LL_1 e_1=e_0+\sqrt2 e_2\ ,\quad \LL_1 e_2=\sqrt2 e_1+\sqrt3 e_3\ ,\quad 
$$
we again find $\tau=3$. \\

We define the matrices $\CC_k := ik \tfrac{2\pi}{L} \LL_1 + \LL_2$, $k\in\Z$ which determine the evolution of the spatial modes of the BGK equation in 1D, cf.\ \eqref{linap4}. Our next goal is to establish a spectral gap of $\CC_k$, uniformly in $k\ne0$. This will prove Theorem \ref{linBGK-decay} in 1D. Clearly, this matrix corresponds to $\CC=i\CC_1+\CC_2$ in \S\ref{sec:Pmatrix}. There, the construction of the transformation matrix $\P(r)=\II+r\A$ was based on Lemma \ref{lemma:ansatzP:nD}, and hence on proving the positive definiteness of
$$
  \RR^*(\CC^*\A+\A\CC)\RR = i\RR^*(-\CC_1^*\A+\A\CC_1)\RR\,.
$$
Here, the operator $\LL_1$ carries the coefficient $ik \tfrac{2\pi}{L}$ with $k\in\Z\setminus\{0\}$. To compensate for $k$, it is natural to choose the perturbation matrix $\A$ proportional to $\frac1k$.
Following \S\ref{sec:dim-ker3} we hence use the ansatz \eqref{P-ansatz-3} for the $k$--dependent transformation matrices $\P_k$: For parameters $\lambda_j;\,j=1,2,3$ 
to be chosen below, we define $\P_k,\, k\ne0$ to be the infinite matrix that has
\begin{equation}\label{P:1D}
\begin{pmatrix}
1  & \lambda_1/k  & 0 & 0 \\ 
\overline{\lambda_1}/k & 1  & \lambda_2/k &0\\
0 & \overline{\lambda_2}/k &1 & \lambda_3/k \\
0 & 0 & \overline{\lambda_3}/k & 1 
\end{pmatrix}
\end{equation}
as its upper-left $4\times 4$ block, with all other entries being those of the identity. 
Under the assumption $|\lambda_1|^2 +|\lambda_2|^2 +|\lambda_3|^2 <1$, 
the matrix $\P_k$ will be positive definite for all $k\ne 0$.
Recalling that $\LL_1$ is an (infinite) real matrix as well as the parameter choice in \eqref{lambda-polar-3D}, it is natural to choose also here $\arg(\lambda_j)=-\frac\pi2$. Hence \eqref{P:1D} turns into
\begin{equation}\label{P:1D'}
\begin{pmatrix}
1  & -i\alpha/k  & 0 & 0 \\ 
i\alpha/k & 1  & -i\beta/k &0\\
0 & i\beta/k &1 & -i\gamma/k \\
0 & 0 & i\gamma/k & 1 
\end{pmatrix} \ ,
\end{equation}
with $\alpha:=|\lambda_1|$, $\beta:=|\lambda_2|$, $\gamma:=|\lambda_3|$.

Now, (the infinite dimensional analog of) Theorem \ref{th:p-ansatz-3} asserts that the above ansatz will yield an admissible transformation matrix $\P$ and hence an exponential decay rate for \eqref{linap4}, uniformly in $k$. But, as a perturbation result, it neither provides an explicit value for the decay rate $\mu$, nor does it yield a rather natural ratio between the parameters $\lambda_j$. These two aspects will be our next task.

\begin{remark}
To justify the infinite dimensional analog of Theorem \ref{th:p-ansatz-3},
 we decompose 
 \[ \CC_k^*\P(r)+\P(r)\CC_k = 2\LL_2 +r\ (\CC_k^*\A+\A\CC_k) = 2\II +(2\LL_2 -2\II) +r\ (\CC_k^*\A+\A\CC_k) \,. \]
To investigate the spectrum of the Hermitian operator~$\CC_k^*\P(r)+\P(r)\CC_k$ in $\ell^2(\N_0)$,
 it is sufficient to compute the spectrum of the compact operator $(2\LL_2 -2\II) +r\ (\CC_k^*\A+\A\CC_k)$.
The compact operators $2\LL_2 -2\II=2\diag(-1,-1,-1,0,\ldots)$ and $\CC_k^*\A+\A\CC_k$ act on a common finite-dimensional subspace of $\ell^2(\N_0)$,
 hence we can use Lemma~\ref{lemma:ansatzP:nD} to analyze the restriction of the compact operators on this finite-dimensional subspace:
The three lowest eigenvalues of $(2\LL_2 -2\II) +r\ (\CC_k^*\A+\A\CC_k)$, for sufficiently small $r\geq0$, satisfy
 \[
   \widetilde{\mu_j}(r) = -2 +r\xi_j +o(r) \, ;\quad j=1,2,3\,,
 \]
 where $\xi_j$ are the eigenvalues of $\RR^*(\CC_k^*\A+\A\CC_k)\RR$ and $\RR=(e_1,e_2,e_3)\in\C^{n\times 3}$ (recall that $\ker(\LL_2) =\spn\{e_1,e_2,e_3\}$).
Then the three lowest eigenvalues of $\CC_k^*\P(r)+\P(r)\CC_k$, for sufficiently small $r\geq0$, satisfy
 \[ \mu_j(r) = r\xi_j +o(r) \,. \]
\end{remark}

Next we search for conditions on $\alpha,\,\beta,\,\gamma>0$ 
such that the eigenvalues $\xi_j$ of 
 \[ 
  \RR^*(\CC_k^*\A+\A\CC_k)\RR = \tfrac{2\pi}{L}\  
   \begin{pmatrix}
    2 \alpha & 0 & \sqrt2\alpha-\beta \\
    0 & 2(\sqrt{2}\beta-\alpha) & 0 \\
    \sqrt2\alpha-\beta & 0 & 2(\sqrt{3}\gamma -\sqrt{2}\beta) 
   \end{pmatrix} \,
 \]
 are positive.
If all minors are positive, then the matrix will be positive definite (by Sylvester's criterion).
We deduce the conditions 
 \[ 
    0<\alpha<\sqrt2\beta<\sqrt3\gamma \quad \text{and} \quad 
    0 <4\alpha(\sqrt{3}\gamma -\sqrt{2}\beta) - |\sqrt{2} \alpha -\beta|^2 \, ,
 \]
which are special cases of \eqref{minor1-3D}, \eqref{det-3D}.
In fact, the matrix $\tfrac{L}{2\pi}\ \RR^*(\CC_k^*\A+\A\CC_k)\RR$ has the eigenvalues $2(\sqrt2\beta-\alpha)$ 
and
 \[
  \sqrt{3}\gamma-\sqrt2\beta+\alpha
   \pm \sqrt{(\sqrt{3}\gamma-\sqrt2\beta-\alpha)^2 +(\sqrt{2}\alpha-\beta)^2} \,.
 \]
We note that the special choice $\beta:=\sqrt{2}\alpha$ and $\gamma:=\sqrt{3}\alpha$ makes all eigenvalues of $\RR^*(\CC_k^*\A+\A\CC_k)\RR$ equal, which seems to be beneficial to obtain eventually a good decay estimate. Moreover, it will simplify the proof of Lemma~\ref{lem:mu:1D}.

In the following lemma we establish an infinite dimensional analog of Lemma \ref{lemma:Pdefinition} -- for \eqref{linap4}, the transformed linearized BGK equation in 1D. However, here we shall not aim at obtaining the optimal decay constant $\mu$ in the matrix inequality \eqref{matrixestimate1}. 
Still, $\mu$ will be independent of the modal index $k\in\Z$, thus providing exponential decay of the full solution.

\begin{lemma} \label{lem:mu:1D}
For each cell length $L>0$ we consider $\alpha^{(3)}=\alpha^{(3)}(L)>0$ defined in~\eqref{root:3:1D}.
If the matrices $\P_k$ are chosen with some $\alpha \in (0,\alpha^{(3)})$, $\beta =\sqrt{2} \alpha$, and $\gamma =\sqrt{3} \alpha$ uniformly for all $|k|\in\N$,
 then $\P_k$ from \eqref{P:1D'} and $\CC_k^*\P_k + \P_k\CC_k$ are positive definite for all $k\in\Z\setminus\{0\}$.
Moreover,
\begin{equation}\label{goodbnd}
\CC_k^*\P_k + \P_k\CC_k  \geq 2\mu \P_k \qquad \text{ uniformly in } |k|\in\N \ , 
\end{equation}
with  
$$
\mu := \frac{\delta_3(1,\alpha)}{8(1-2\pi\alpha/L)^2 (1+\alpha \sqrt{3+\sqrt{6}})}>0\ , 
$$
 where $\delta_3(1,\alpha):= \det \DD_{1,\alpha,\sqrt{2} \alpha,\sqrt{3} \alpha}^{(3)}>0$
 with the matrix $\DD_{k,\alpha,\sqrt{2} \alpha,\sqrt{3} \alpha}^{(3)}$ defined in~\eqref{Dka3}.
\end{lemma} 
The proof of this lemma is deferred to Appendix~\ref{A1}.

\begin{remark}\label{rem3.2}
\begin{enumerate}[label=(\alph*)]
\item Consider 
 \[ \alpha_* = \argmax_{\alpha\in[0,\alpha^{(3)}]} \frac{\delta_3(1,\alpha)}{8(1-2\pi\alpha/L)^2 (1+\alpha \sqrt{3+\sqrt{6}})} \ . \]
 Choosing $\P_k$ with $\alpha=\alpha_*$, $\beta =\sqrt{2} \alpha$, and $\gamma =\sqrt{3} \alpha$ uniformly for all $|k|\in\N$, yields~\eqref{goodbnd} with the 
 constant
\begin{equation}\label{goodbnd2}
  \mu_* = \frac{\delta_3(1,\alpha_* )}{8(1-2\pi\alpha_* /L)^2 (1+\alpha_* \sqrt{3+\sqrt{6}})}
      = \max_{\alpha\in[0,\alpha^{(3)}]} \frac{\delta_3(1,\alpha)}{8(1-2\pi\alpha/L)^2 (1+\alpha \sqrt{3+\sqrt{6}})} \ .
\end{equation}
\item In the limit $L\to \infty$, the matrix $\CC_k^*\P_k + \P_k\CC_k$ has zero eigenvalues,
 which is apparent from its upper left submatrix $\DD_{k,\alpha,\sqrt{2} \alpha,\sqrt{3} \alpha}$ defined in~\eqref{Dkaaa}.
 Accordingly, $\alpha^{(3)} \to 0$ with $\alpha^{(3)} =O(\frac1L)$ and $\mu_*=O(\frac1{L^2})$ in the limit $L\to \infty$. It is no surprise that the exponential decay rate vanishes in this limit, as the limiting whole space problem only exhibits algebraic decay (cf.\ \cite{BDMMS17} for the large-time analysis of \eqref{bgk:linear} on $\R^d$).
 
 In the limit $L\to 0$, again $\alpha^{(3)} \to 0$ with $\alpha^{(3)} =O(L)$. 
Using
\begin{equation}\label{alpha-limit}
   \lim_{L\to0} \frac{\alpha_*(L)}{L} 
   = \frac{4-\sqrt{13}}{6\pi}\,,
\end{equation}
 we obtain
 $$
   \lim_{L\to0} \mu_*(L)= 3 (4-\sqrt{13}) \frac{(3-\sqrt{13})^2}{(1-\sqrt{13})^2} = 0.06391670961...
 $$
 (cf.\ Fig.\ \ref{fig:muL}).
\end{enumerate}
\end{remark}

\begin{figure}[ht!]
\begin{center}
 \includegraphics[scale=.5]{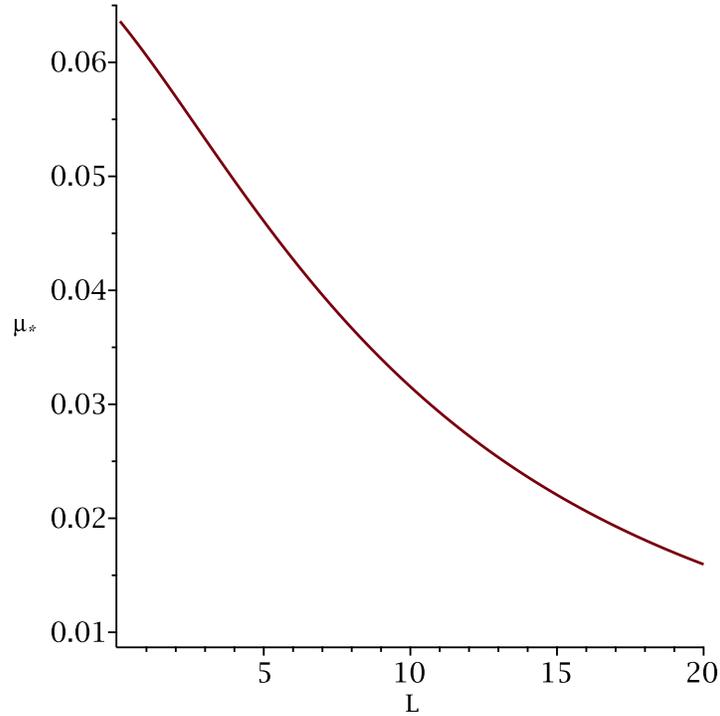}
 \caption{For each cell length $L$ the constant $2\mu_*(L)$ obtained from Lemma \ref{lem:mu:1D} and Remark \ref{rem3.2}(a) 
 yields a bound for the entropy decay rate in Theorem \ref{linBGK-decay}.
 }
 \label{fig:muL}
\end{center}
\end{figure}


\medskip
Applying Lemma \ref{lem:mu:1D} to each $x$-Fourier mode $\hat{\bf h}_k(t),\,k\ne0$ from \eqref{linap4} allows to prove exponential decay of the linearized BGK equation in 1D:
\begin{proof}[Proof of Theorem~\ref{linBGK-decay} in 1D]
We consider a solution $h$ of \eqref{linBGK:torus}, and let the entropy functional $\cE(\tilde{f})$ be defined by
\begin{equation*}
  \cE(\tilde{f}) := \sum_{k\in \Z} \langle h_k(v), \P_k h_k(v)\rangle_{L^2(M_1^{-1})}\ ,
\end{equation*}
with $\tilde{f}(t) :=M_1 +h(t)$.
Here, the ``infinite matrices'' $\P_0 := \II$ and $\P_k$ defined in \eqref{P:1D} for $k\neq 0$
are regarded as bounded operators on $L^2(M_1^{-1})$.
Then 
\begin{equation*}
  \ddt \cE(\tilde{f}) = -\sum_{k\in \Z} \langle h_k(v), (\CC_k^*\P_k + \P_k\CC_k) h_k(v)\rangle_{L^2(M_1^{-1})} 
   \leq - 2\min\{1,\,\mu_*\} \,\cE(\tilde{f}) \ , 
\end{equation*}
where $1$ is the decay rate of $h_0$, cf.~\eqref{h0-decay}.
This implies~\eqref{ineq:linBGK-decay} with $\lambda^1(L):=2\min\{1,\,\mu_*\}$ and $\mu_*$ from \eqref{goodbnd2}. 

The constants $c_1$ and $C_1$ in the estimate \eqref{entropy-equiv} follow from \eqref{Pbound}: 
\begin{equation}\label{norm-const:1D}
  c_1(L)=\left(1+\alpha_*(L)\sqrt{3+\sqrt6}\right)^{-1}\,,\quad
  C_1(L)=\left(1-\alpha_*(L)\sqrt{3+\sqrt6}\right)^{-1}\,
\end{equation}
and this finishes the proof of Theorem \ref{linBGK-decay} in 1D.
\end{proof}

To appreciate the above decay estimate, let us compare it to the spectral gap obtained in numerical tests for $L=2\pi$. In this case the estimate from Remark \ref{rem3.2} yields the analytic bound with $\mu_*=0.041812...$. As a comparison we computed the spectrum of finite dimensional approximation matrices to $\LL_2+ik\LL_1$ up to the matrix size $n=500$. Apparently the spectral gap is determined by the lowest spatial modes $k=\pm1$. With increasing $n$ it grows monotonically to $\mu_{num}=0.558296...$. So, our estimate is off by a factor of about 13. Following the strategy from \S4.3 in \cite{AAC16}, i.e.\ by maximizing $\mu$ in the matrix inequality $\CC_k^*\P_k+\P_k\CC-2\mu\P_k\ge0$, the above estimate on the decay rate could be improved further. But we shall not pursue this strategy here again.

Let us briefly compare this gap to the situation in the two 1D BGK models analyzed in \S4.3 and \S4.4 of \cite{AAC16}. They only differ from the 1D model \eqref{linap4}-\eqref{L1L2} of this section, concerning the matrix $\LL_2$: there we had $\LL_2=\diag(0,\,1,\,...)$ and $\LL_2=\diag(0,\,1,\,0,\,1,\,...)$, resp. We recall from \S\ref{sec-hypo-index} that both models have hypocoercivity index 1, and their spectral gaps are 0.6974... and 0.3709717660..., resp. One might expect that removing 1 entries from $\LL_2$ and hence increasing the hypocoercivity index would decrease the spectral gap. But this is obviously not always the case.


\section{Linearized BGK equation in 2D}
\label{sec:linBGK:2D}

Next we shall analyze the linearized BGK equation~\eqref{linBGK:torus} in 2D:
\begin{multline*} 
 \partial_t h(x,v,t) +  v\cdot \nabla_x h(x,v,t) \\
  = M_1(v)\left[\big( 2 - \frac{|v|^2}{2}\big)\sigma(x,t)+v\cdot \mu(x,t)
    + \big( -\frac{1}{2}+ \frac{|v|^2}{4}\big)\tau(x,t)\right]  - h(x,v,t)\ , \quad t\geq 0 \ ,
\end{multline*}
 for the perturbation $h(x,v,t) \approx f(x,v,t) -M_1(v)$ with $x\in\tilde\T^2$, $v\in\R^2$.

 Again we consider the $x$--Fourier series of $h$: 
 \begin{equation*}
  h(x,v,t) = \sum_{k\in\Z^2} h_k(v,t)\,e^{i\frac{2\pi}{L} k\cdot x}\ ;
 \end{equation*}
 each spatial mode $h_k(v,t)$ is decoupled and evolves as
\begin{equation} \label{linBGK:Fourier:2D} 
 \ddt h_k +  i\tfrac{2\pi}{L} k\cdot v\ h_k =
      M_1(v)\left[\big( 2 - \frac{|v|^2}{2}\big)\sigma_k(t)+v\cdot \mu_k(t)
      + \big( -\frac{1}{2}+ \frac{|v|^2}{4}\big)\tau_k(t)\right]  - h_k(v,t)\ , \quad t\geq 0 \ .
\end{equation}
Here, $\sigma_k$, $\mu_k$ and $\tau_k$ denote the spatial modes of the $v$--moments $\sigma$, $\mu$ and $\tau$ defined in \eqref{mloc2}; hence
\[ \sigma_k := \int_{\R^2} h_k(v,t)\d[v]\ , \quad 
   \mu_k := \int_{\R^2} v\ h_k(v,t)\d[v]\ , \quad 
   \tau_k := \int_{\R^2} |v|^2\ h_k(v,t)\d[v]\ .
\]

Next we shall introduce an orthonormal basis in $v$-direction, to represent the spatial modes $h_k(\cdot,t)\in L^2(\R^2;M_1^{-1})$, $k\in\Z^2$. As in 1D we shall again use Hermite functions. But their multi-dimensional generalization is not unique, and we shall present two options that seem to be practical:\\

\noindent
\underline{Basis 1 (``pure tensor-basis''):}
  A complete set of orthogonal polynomials in $d$ variables can be formed as products of $d$ such polynomials, each in a single variable. 
  Using the Hermite polynomials $H_m$ in 1D, i.e.\
  \[ H_0(\upsilon) =1\ , \quad H_1(\upsilon) =\upsilon\ , \quad H_2(\upsilon) =\upsilon^2 -1\ , \quad H_3(\upsilon) =\upsilon^3 -3\upsilon\ ,\ ... \quad \text{with } \upsilon\in\R\ , \]
  we construct Hermite polynomials on $\R^d$ as
  \begin{equation} \label{hermite-polynom:2D:a}
    H_m(v) := \prod_{j=1}^d H_{m_j}(v_j)\ ,\quad v\in\R^d\ ,
  \end{equation}
with the multi-index $m=(m_1,\ldots,m_d)\in\N_0^d$. 
They are also generated by a simple generalization of formula~\eqref{hermite-polynom:1D}:
  \begin{equation*} 
   H_m(v) = (-1)^{|m|} e^{\frac{|v|^2}{2}}\frac{\partial^{|m|}}{\partial v^m}e^{-\frac{|v|^2}{2}}\ ,\quad m\in\N_0^d\ ,
  \end{equation*}
 with $|m|=\sum_{j=1}^d m_j$ (see \cite{Gr49}, e.g.).
  For $d=2$, we obtain
  \begin{align*}
   H_{0,0}(v) &=H_0(v_1) H_0(v_2) =1\ , \quad H_{1,0}(v) =H_1(v_1) H_0(v_2) =v_1\ , \quad H_{0,1}(v) =H_0(v_1) H_1(v_2) =v_2\ , \\
   H_{2,0}(v) &=H_2(v_1) H_0(v_2) =v_1^2 -1\ , \quad H_{1,1}(v) =H_1(v_1) H_1(v_2) =v_1 v_2\ , \quad H_{0,2}(v) =H_0(v_1) H_2(v_2) =v_2^2 -1\ ,  \\
   H_{3,0}(v) &=H_3(v_1) H_0(v_2) =v_1^3 -3v_1\ , \quad H_{2,1}(v) =H_2(v_1) H_1(v_2) =(v_1^2 -1) v_2\ , \\
   H_{1,2}(v) &=H_1(v_1) H_2(v_2) =v_1 (v_2^2 -1)\ , \quad H_{0,3}(v) =H_0(v_1) H_3(v_2) =v_2^3 -3v_2\ .
  \end{align*}
	Using definition~\eqref{hermite-fct:1D} of normalized Hermite functions in 1D,
	 we define the normalized Hermite functions in $d$ dimensions as 
	 \begin{equation} \label{hermite-fct:dD}
	   g_m(v) := \prod_{j=1}^d g_{m_j} (v_j) \qquad \text{for } m=(m_1,\ldots,m_d)\in\N_0^d \ .
	 \end{equation}
	Then, $g_m$ ($m\in \N_0^d$) form an orthonormal basis of $L^2(\R^d;M_1^{-1})$
	 and inherit a simple recurrence relation:
	For $k\in\{1,\ldots,d\}$, $m\in\N_0^d$, and the Euclidean basis vectors $e_k =(\delta_{kj})_{j=1,\ldots,d}$ in $\R^d$, 
	 the recurrence relation
	 \begin{equation} \label{recur:dD1}
	   v_k g_m(v) = \sqrt{m_k+1}\ g_{m+e_k}(v) +\sqrt{m_k}\ g_{m-e_k}(v)
	 \end{equation}
	 holds.

In order to give a vector representation of \eqref{linBGK:Fourier:2D}, the evolution equation of the spatial modes $h_k(v,t)$, 
we first need to introduce a linear ordering of the velocity basis $g_m$ ($m\in \N_0^2$).
We shall use a lexicographic order, i.e.\ first (increasingly) with respect to the total order $|m|$, and
 within a set of order $|m|$ we order w.r.t. $m_1$  (decreasingly) (for $d=2$).
Thus, we obtain the linearly ordered basis
\begin{align*} 
   g_0(v) &:=g_{0,0}(v) =M_1(v)\ , \quad
	 g_1(v) :=g_{1,0}(v) =v_1 M_1(v)\ , \quad
	 g_2(v) :=g_{0,1}(v) =v_2 M_1(v)\ , \\
   g_3(v) &:=g_{2,0}(v) =\tfrac1{\sqrt{2}} (v_1^2 -1) M_1(v)\ , \quad
	 g_4(v) :=g_{1,1}(v) =v_1 v_2 M_1(v)\ , \quad
	 g_5(v) :=g_{0,2}(v) =\tfrac1{\sqrt{2}} (v_2^2 -1) M_1(v)\ ,  \\
   g_6(v) &:=g_{3,0}(v) =\tfrac1{\sqrt{3!}} (v_1^3 -3v_1) M_1(v)\ , \quad
	 g_7(v) :=g_{2,1}(v) =\tfrac1{\sqrt{2}} (v_1^2 -1) v_2 M_1(v)\ , \\
   g_8(v) &=g_{1,2}(v) =\tfrac1{\sqrt{2}} v_1 (v_2^2 -1) M_1(v)\ , \quad
	 g_9(v) :=g_{0,3}(v) =\tfrac1{\sqrt{3!}} (v_2^3 -3v_2) M_1(v)\ , \quad ...
\end{align*}
Given a multi-index $m\in\N_0^2$, its lexicographic index is computed as $|m| (|m|+1)/2 +m_2$ with $|m| =m_1 +m_2$.\\

\noindent
\underline{Basis 2 (``energy-basis''):}
The second basis is a simple variant of the first one. We recall that the evolution with the BGK equation~\eqref{bgk} conserves the (kinetic) energy and mass. Hence, their difference is also conserved and it is related to the polynomial $\frac{|v|^2}2 -1$. In analogy to the 1D case from \S\ref{sec:linBGK:1D} it is thus a natural option to construct a basis of orthogonal polynomials $\tilde{H}_m(v)$, $m\in\N_0^d$, such that $\frac{|v|^2}2 -1$ is a basis element.
Compared to $\{H_m(v)\}$, in fact, we only have to modify the Hermite polynomials of second order. For $d=2$ they read:
  \begin{equation}\label{Hermite:basis:2D}
     \tilde{H}_m(v) :=
       \begin{cases}
         H_m(v) & \text{if } |m|\ne 2\ , \\
         \frac12 \big(H_{2,0}(v) +H_{0,2}(v)\big) = \frac{|v|^2}2 -1 & \text{if } m=(2,0)\ , \\
         H_{1,1}(v) = v_1 v_2 & \text{if } m=(1,1)\ , \\ 
         \frac12 \big(H_{2,0}(v) -H_{0,2}(v)\big) = \frac{v_1^2 -v_2^2}2 & \text{if } m=(0,2)\ .
       \end{cases}
  \end{equation}
	Similarly, we define normalized Hermite functions
  \begin{equation*}
     \tilde{g}_m(v) :=
       \begin{cases}
         g_m(v) & \text{if } |m|\ne 2\ , \\
         \frac1{\sqrt2} \big(g_{2,0}(v) +g_{0,2}(v)\big) & \text{if } m=(2,0)\ , \\
         g_{1,1}(v) & \text{if } m=(1,1)\ , \\ 
         \frac1{\sqrt2} \big(g_{2,0}(v) -g_{0,2}(v)\big) & \text{if } m=(0,2)\ .
       \end{cases}
  \end{equation*}
	The elements $\tilde{g}_m$ satisfy a recurrence relation similar to~\eqref{recur:dD1},
	 except for identities involving $\tilde{g}_{2,0}$ or $\tilde{g}_{0,2}$.
	For example,
	 \begin{eqnarray*}
	    v_1 \tilde{g}_{2,0}(v)
	       &=& \frac1{\sqrt{2}} v_1\ \big(g_{2,0}(v) +g_{0,2}(v)\big) 
	       = \frac1{\sqrt{2}} \big(\sqrt{3}\ g_{3,0}(v) +\sqrt{2}\ g_{1,0}(v) +g_{1,2}(v)\big) \\
	       &=& \frac1{\sqrt{2}} \big(\sqrt{3}\ \tilde{g}_{3,0}(v) +\sqrt{2}\ \tilde{g}_{1,0}(v) +\tilde{g}_{1,2}(v)\big)\ .
	 \end{eqnarray*}
While the first basis $g_m$ $(m\in\N_0^d)$ inherits a simple recurrence formula with three elements,
 the recurrence formulas for the second basis $\tilde{g}_m$ $(m\in\N_0^d)$ involve four elements, when including $\tilde{g}_{2,0}(v)$ or $\tilde{g}_{0,2}(v)$.\\

To derive the vector representation of \eqref{linBGK:Fourier:2D}, 
 it is preferable to use the first basis with the linear ordering $g_m$ $(m\in\N_0)$. 
With the identity $(g_3(v) +g_5(v))/\sqrt{2} = (|v|^2/2 -1) M_1(v)$ we rewrite~\eqref{linBGK:Fourier:2D} as
 \begin{multline} \label{linBGK:Fourier:2Da} 
 \ddt h_k +  i\tfrac{2\pi}{L} k\cdot v\ h_k =
      \big( g_0(v) -\tfrac1{\sqrt{2}} g_3(v) -\tfrac1{\sqrt{2}} g_5(v)\big)\sigma_k(t)
			+ \binom{g_1(v)}{g_2(v)}\cdot \mu_k(t) \\
      + \tfrac1{2\sqrt{2}} \big( g_3(v) +g_5(v)\big)\tau_k(t)  - h_k(v,t)\ , 
			\quad k\in\Z^2\ , \quad t\geq 0 \ .
 \end{multline}
First we consider the spatial mode $h_0$ with $k=0$. With the same argument as in 1D we again obtain \eqref{h0-decay}, i.e.\ $\ddt h_0(v,t)=-h_0(v,t)$.
  Next we expand $h_k(\cdot,t)\in L^2(\R^2;M_1^{-1})$ in the orthonormal basis $\{g_m(v)\}_{m\in\N_0}$:
 $$
   h_k(v,t)=\sum_{m=0}^\infty \hat h_{k,m}(t)\,g_m(v)\ ,\quad\text{with}\quad
   \hat h_{k,m} = \langle h_k(v),g_m(v)\rangle_{L^2(M_1^{-1})}\ .
 $$
 For each spatial mode $k\in\Z^2$,
  the ``infinite vector'' $\hat {\bf h}_k(t) = ( \hat h_{k,0}(t),\,\hat h_{k,1}(t),\ ...)^\top\in \ell^2(\N_0)$ contains all 2D--Hermite coefficients of $h_k(\cdot,t)$. 
 In particular we have
 $$
   \hat h_{k,0} = \int_{\R^2} h_k(v) g_0(v) M_1^{-1}(v) \d[v] = \sigma_k \ , \qquad 
   \binom{\hat h_{k,1}}{\hat h_{k,2}} = \int_{\R^2} h_k(v) \binom{g_1(v)}{g_2(v)} M_1^{-1}(v) \d[v] = \mu_k \in\R^2 \ ,
 $$
 and
 $$
   \tfrac1{\sqrt{2}} (\hat h_{k,3} +\hat h_{k,5})
	   = \int_{\R^2} h_k(v) \tfrac1{\sqrt{2}} (g_3(v) +g_5(v)) M_1^{-1}(v) \d[v]
	   = \frac12\tau_k -\sigma_k\ .
 $$
 Thus, we can rewrite \eqref{linBGK:Fourier:2Da} as
 \begin{align} \label{linBGK:Fourier2D} 
 \ddt h_k +  i\tfrac{2\pi}{L} k\cdot v\ h_k
    &= g_0(v) \hat h_{k,0} 
			+ \binom{g_1(v)}{g_2(v)}\cdot \binom{\hat h_{k,1}}{\hat h_{k,2}}
      +\frac{ g_3(v) +g_5(v)}{2} 
			(\hat h_{k,3} +\hat h_{k,5})
			- h_k(v,t)\ , \quad k\in\Z^2\ , \quad t\geq 0 \ .
\end{align}

Our next goal is to rewrite this system in the Hermite function basis as an infinite vector system -- in analogy to \eqref{linap4} in 1D. In that equation, the operator $\LL_1$ is multiplied by the (scalar and integer) mode number $k$, which is then used in the construction of the transformation matrices $\P_k$. To extend this structure and strategy to 2D, we first have to consider the rotational symmetry of \eqref{linBGK:Fourier2D}: We note that the basis functions $g_0$ and $g_3+g_5$ depend only on $|v|$, and that the interplay between the vectors $k$ and $v$ only occurs via $k\cdot v$. Hence, evolution equations from the family \eqref{linBGK:Fourier2D} having the same modulus $|k|$ are identical in the following sense: Rotating the spatial mode vector $k$ and the $v$-coordinate system by the same angle, leaves \eqref{linBGK:Fourier2D} invariant. Thus it suffices to consider \eqref{linBGK:Fourier2D} for vectors $k=(\kappa,0)^\top$ with the discrete moduli 
$$
  \kappa\in K:=\big\{r\ge1 \, \big|\, \exists k\in\Z^2\setminus\{0\} \mbox{ with } r=|k| \big\} \ . 
$$
We skipped here the mode $h_0$, as it was already analyzed before. In the sequel we also denote $h_\kappa:=h_{\kappa,0}$ and $\hat {\bf h}_\kappa:=\hat {\bf h}_{\kappa,0}$.
With this notation, \eqref{linBGK:Fourier2D} reads
\[ \ddt h_\kappa +  i\tfrac{2\pi}{L} \kappa v_1\ h_\kappa
     = g_0(v) \hat h_{\kappa,0} 
			+ \binom{g_1(v)}{g_2(v)}\cdot \binom{\hat h_{\kappa,1}}{\hat h_{\kappa,2}}
      +\frac{ g_3(v) +g_5(v)}{2} 
			(\hat h_{\kappa,3} +\hat h_{\kappa,5})
			- h_\kappa(v,t)\ , \quad \kappa\in K \ .
\]
Then, the vector of its Hermite coefficients satisfies
\begin{equation}\label{linap4:2D}
\ddt\hat {\bf h}_\kappa(t) + i\,\tfrac{2\pi}{L}\kappa\, \LL_1 \hat{\bf h}_\kappa(t) = -\LL_2 \hat{\bf h}_\kappa(t)\ ,\quad \kappa\in K\ ,\quad\ t\ge0\ ,
\end{equation}
where the operators $\LL_1,\,\LL_2$ are represented by symmetric ``infinite matrices'' on $\ell^2(\N_0)$:
\setcounter{MaxMatrixCols}{15}
\begin{eqnarray*}
  \LL_1&=& \begin{pmatrix}
  0 & \sqrt 1 & 0 & 0 & 0 & 0 & 0 & 0 & 0 & 0 & 0 & 0 & 0 & 0 & \cdots \\
  \sqrt 1 & 0 & 0 & \sqrt 2 & 0 & 0 & 0 & 0 & 0 & 0 & 0 & 0 & 0 & 0 & \cdots \\
  0 & 0 & 0 & 0 & \sqrt 1 & 0 & 0 & 0 & 0 & 0 & 0 & 0 & 0 & 0 & \cdots \\
  0 & \sqrt 2 & 0 & 0 & 0 & 0 & \sqrt 3 & 0 & 0 & 0 & 0 & 0 & 0 & 0 & \cdots \\
  0 & 0 & \sqrt 1 & 0 & 0 & 0 & 0 & \sqrt 2 & 0 & 0 & 0 & 0 & 0 & 0 & \cdots \\
  0 & 0 & 0 & 0 & 0 & 0 & 0 & 0 & \sqrt 1 & 0 & 0 & 0 & 0 & 0 & \cdots \\
  0 & 0 & 0 & \sqrt 3 & 0 & 0 & 0 & 0 & 0 & 0 & \sqrt 4 & 0 & 0 & 0 & \cdots \\
  0 & 0 & 0 & 0 & \sqrt 2 & 0 & 0 & 0 & 0 & 0 & 0 & \sqrt 3 & 0 & 0 & \cdots \\
  0 & 0 & 0 & 0 & 0 & \sqrt 1 & 0 & 0 & 0 & 0 & 0 & 0 & \sqrt 2 & 0 & \cdots \\
  0 & 0 & 0 & 0 & 0 & 0 & 0 & 0 & 0 & 0 & 0 & 0 & 0 & \sqrt 1 & \cdots \\
  \vdots &\vdots &\vdots &\vdots &\vdots &\vdots &\vdots &\vdots &\vdots &\vdots &\vdots &\vdots &\vdots &\vdots & \ddots 
  \end{pmatrix}
\ ,\quad\\
 \LL_2&=&\diag(0,\, 0,\, 0,\, \tfrac12 \begin{pmatrix} 1 & 0 & -1 \\ 0 & 2 & 0 \\ -1 & 0 & 1 \end{pmatrix},\, 1,\, 1,\cdots)\ .
\end{eqnarray*}

To compute the hypocoercivity index of the BGK model in 2D, 
 it is preferable to use the second basis with the linear ordering $\tilde{g}_m$ $(m\in\N_0)$. 
We shall give the matrix representation of the two dimensional BGK equation \eqref{linBGK:Fourier:2D} in the second velocity basis, again only for the spatial modes $k=(\kappa,0)^\top, \,\kappa\in K$. To obtain the corresponding matrices $\tilde\LL_1$ and $\tilde\LL_2$, we simply represent the linear basis transformation by the infinite matrix 
 \[
  \Smatrix = \diag (1,\ 1,\ 1,\ \tfrac1{\sqrt{2}} \begin{pmatrix} 1 & 0 & 1 \\ 0 & \sqrt{2} & 0 \\ 1 & 0 & -1 \end{pmatrix},\ 1,\ 1,\cdots)\ ,
 \]
 which is self-inverse, i.e. $\Smatrix=\Smatrix^{-1}$.
Thus we compute $\tilde{\LL}_1 = \Smatrix^{-1} \LL_1 \Smatrix$ and $\tilde{\LL}_2 = \Smatrix^{-1} \LL_2 \Smatrix$, yielding
\begin{eqnarray*}
 &&\tilde\LL_1= \begin{pmatrix}
  0 & 1 & 0 & 0 & 0 & 0 & 0 & 0 & 0 & 0 & 0 & \cdots \\
  1 & 0 & 0 & 1 & 0 & 1 & 0 & 0 & 0 & 0 & 0 & \cdots \\
  0 & 0 & 0 & 0 & 1 & 0 & 0 & 0 & 0 & 0 & 0 & \cdots \\
  0 & 1 & 0 & 0 & 0 & 0 & \sqrt{3/2} & 0 & \sqrt{1/2} & 0 & 0 & \cdots \\
  0 & 0 & 1 & 0 & 0 & 0 & 0 & \sqrt 2 & 0 & 0 & 0 & \cdots \\
  0 & 1 & 0 & 0 & 0 & 0 & \sqrt{3/2} & 0 & -\sqrt{1/2} & 0 & 0 & \cdots \\
  0 & 0 & 0 & \sqrt{3/2} & 0 & \sqrt{3/2} & 0 & 0 & 0 & 0 & 2 & \cdots \\
  0 & 0 & 0 & 0 & \sqrt 2 & 0 & 0 & 0 & 0 & 0 & 0 & \cdots \\
  0 & 0 & 0 & \sqrt{1/2} & 0 & -\sqrt{1/2} & 0 & 0 & 0 & 0 & 0 & \cdots \\
  0 & 0 & 0 & 0 & 0 & 0 & 0 & 0 & 0 & 0 & 0 & \cdots \\
  0 & 0 & 0 & 0 & 0 & 0 & 2 & 0 & 0 & 0 & 0 & \cdots \\
  \vdots & \vdots &\vdots &\vdots &\vdots &\vdots &\vdots &\vdots &\vdots & \vdots & \vdots & \ddots 
  \end{pmatrix}
 \ ,\\[3mm]
 &&\tilde\LL_2=\diag(0,\, 0,\, 0,\, 0,\, 1,\, 1,\cdots)\ .
\end{eqnarray*}

%
This second basis representation makes it easy of determine the hypocoercivity index of the BGK model in 2D. 
As for the 1D model, we use Definition \ref{hyp-index} in its equivalent formulation \ref{cond:trivkernel}: 
With $\corank \tilde\LL_2=4$, $\ker\tilde\LL_2=\spn\{e_0,\,e_1,\,e_2,\,e_3\}$, and the relations 
\begin{equation} \label{linBGK2D:rangeL1}
  \tilde\LL_1 e_0=e_1\ ,\quad \tilde\LL_1 e_1=e_0+e_3+e_5\ ,\quad \tilde\LL_1 e_2=e_4\ ,\quad \tilde\LL_1 e_3=e_1+\sqrt{3/2}\, e_6+\sqrt{1/2}\, e_8\ ,
\end{equation}
we find the index $\tau=2$. 
At first glance this may come as a surprise, since the analogous 1D model has index 3. 
But in 2D, each of the two momentum-conservation modes (represented by $e_1$ and $e_2$) is directly coupled to a decreasing mode (represented by $e_5$ and $e_4$, respectively).
These modes are quadratic polynomials, but orthogonal to $|v|^2$, where the latter corresponds to the (conserved) kinetic energy, cf.\ \eqref{Hermite:basis:2D}.\\

We define the matrices $\CC_\kappa := i\tfrac{2\pi}{L} \kappa \tilde\LL_1 + \tilde\LL_2$, $\kappa\in K\cup\{0\}$ which determine the evolution of the spatial modes of the BGK equation in 2D, cf.\ \eqref{linap4:2D}. Our next goal is to establish a spectral gap of $\CC_\kappa$, uniformly in $\kappa\in K$. This will prove Theorem \ref{linBGK-decay} in 2D.
To this end we make an ansatz for the transformation matrices: Following the detailed motivation from the 1D analog in \S\ref{sec:linBGK:1D}, let $\P_\kappa$, $\kappa\in K$ be the identity matrix whose upper-left $7\times 7$ block is replaced by
 \begin{equation}\label{P:2D}
  \begin{pmatrix}
  1 & -i \alpha/\kappa & 0 & 0 & 0 & 0 & 0 \\
  i \alpha/\kappa & 1 & 0 & 0 & 0 & -i \beta/\kappa & 0 \\
  0 & 0 & 1 & 0 & -i \gamma/\kappa & 0 & 0 \\
  0 & 0 & 0 & 1 & 0 & 0 & -i \omega/\kappa \\
  0 & 0 & i \gamma/\kappa & 0 & 1 & 0 & 0 \\
  0 & i \beta/\kappa & 0 & 0 & 0 & 1 & 0 \\
  0 & 0 & 0 & i \omega/\kappa & 0 & 0 & 1 
  \end{pmatrix}
 \end{equation}
 with positive parameters $\alpha$, $\beta$, $\gamma$, and $\omega$ to be chosen below.
The distribution of the non-zero off-diagonal elements follows from the pattern in matrix $\tilde\LL_1$,
with the following rationale:
The $\alpha$-term couples the $e_0$-mode to the $e_1$-mode,
 which is the only choice according to~\eqref{linBGK2D:rangeL1}.
The $\beta$-term couples the $e_1$-mode to the decaying $e_5$-mode,
 and the $\gamma$-term couples the $e_2$-mode to the decaying $e_4$-mode.
Finally, the $\omega$-term couples the $e_3$-mode to the $e_6$-mode,
 the first decaying mode according to~\eqref{linBGK2D:rangeL1}.


\begin{lemma} \label{lemma:decay:2D}
If the matrices $\P_\kappa$ are chosen as~\eqref{P:2D} with $\beta =2\alpha$, $\gamma =\alpha$, and $\omega=\sqrt{6}\alpha$ uniformly for all $\kappa\in K$,
 then there exists $0<\alpha_+$ such that $\P_\kappa$ and $\CC_\kappa^*\P_\kappa + \P_\kappa\CC_\kappa$ are positive definite for all $\alpha \in (0,\alpha_+)$ and $\kappa\in K$.
Moreover,
\begin{equation}\label{goodbnd:2D}
\CC_\kappa^*\P_\kappa + \P_\kappa\CC_\kappa  \geq 2\mu \P_\kappa \qquad \text{ uniformly in } \kappa\in K \ , 
\end{equation}
with  
\begin{equation*} 
\mu := \left( \frac{10}{14} \right)^{10} \frac{\delta_{11}(1,\alpha,2\alpha,\alpha,\sqrt{6}\alpha)}{2 \big(1+\sqrt{6} \alpha\big)}>0\ , 
\end{equation*}
 where $\delta_{11}(1,\alpha,2\alpha,\alpha,\sqrt{6}\alpha):= \det \DD_{1,\alpha,2\alpha,\alpha,\sqrt{6}\alpha}$
 with $\DD_{\kappa,\alpha,2\alpha,\alpha,\sqrt{6}\alpha}$ defined in~\eqref{Dka2aa3a}.
\end{lemma}

The proof of this lemma is deferred to Appendix~\ref{A1}.

\begin{remark}
\begin{enumerate}[label=(\alph*)]
\item Consider 
 \[ \alpha_* = \argmax_{\alpha\in[0,\alpha_+]} \frac{\delta_{11}(1,\alpha,2\alpha,\alpha,\sqrt{6}\alpha)}{2 \big(1+\sqrt{6} \alpha\big)} \ . \]
 Choosing $\P_\kappa$ with $\alpha=\alpha_*$, $\beta =2\alpha$, $\gamma =\alpha$, and $\omega=\sqrt{6}\alpha$ uniformly for all $\kappa\in K$, yields~\eqref{goodbnd:2D} with the maximal constant
 \begin{equation}\label{goodbnd2:2D:b}
  \mu = \left( \frac{10}{14} \right)^{10} \frac{\delta_{11}(1,\alpha_*,2\alpha_*,\alpha_*,\sqrt{6}\alpha_*)}{2 \big(1+\sqrt{6} \alpha_*\big)}
      = \max_{\alpha\in[0,\alpha_+]} \left( \frac{10}{14} \right)^{10} \frac{\delta_{11}(1,\alpha,2\alpha,\alpha,\sqrt{6}\alpha)}{2 \big(1+\sqrt{6} \alpha\big)} \ .
 \end{equation}
\item For $L=2\pi$, we compute $\alpha_+ = 0,2102380141$.
 Moreover, the constant $\mu$ is determined as $\mu = 0,003013362117$ with $\alpha_* = 0,1453311384$.
\item In the limit $L\to +\infty$, the matrix $\CC_\kappa^*\P_\kappa + \P_\kappa\CC_\kappa$ has zero eigenvalues,
 which is apparent from its upper left submatrix $\DD_{\kappa,\alpha,2\alpha,\alpha,\sqrt{6}\alpha}$ defined in~\eqref{Dka2aa3a}.
 Accordingly, $\alpha_+ \to 0$ 
  in the limit $L\to \infty$.
 
 Moreover, $\alpha_+ \to 0$ in the limit $L\to 0$. 
\end{enumerate}
\end{remark}

\begin{proof}[Proof of Theorem~\ref{linBGK-decay} in 2D]
We consider a solution $h$ of \eqref{linBGK:torus}, and the entropy functional $\cE(\tilde{f})$ defined by
\begin{equation*} 
  \cE(\tilde{f}) := \sum_{k\in \Z^2} \langle h_k(v), \P_{|k|} h_k(v)\rangle_{L^2(M_1^{-1})}\ ,
\end{equation*}
with $\tilde{f}(t) := M_1 +h(t)$.
Here the matrices $\P_0 = \II$ and $\P_\kappa$ defined in \eqref{P:2D} for $\kappa=|k|\neq 0$
are regarded as bounded operators on $L^2(M_1^{-1})$.
Then 
\begin{equation*} 
  \ddt \cE(\tilde{f}) = -\sum_{k\in \Z^2} \langle h_k(v), (\CC_{|k|}^*\P_{|k|} + \P_{|k|}\CC_{|k|}) h_k(v)\rangle_{L^2(M_1^{-1})} 
   \leq - 2\min\{1,\mu\} \,\cE(\tilde{f}) \ , 
\end{equation*}
where $1$ is the decay rate of $h_0$, cf.~\eqref{h0-decay}.
This implies~\eqref{ineq:linBGK-decay} with $\lambda^2(L):=2\min\{1,\mu\}$ and $\mu$ from~\eqref{goodbnd2:2D:b}.

The constants $c_2$ and $C_2$ in the estimate \eqref{entropy-equiv} follow from \eqref{Pbound2}: 
\begin{equation*} 
  c_2(L)=\left(1 +\sqrt{6} \alpha_* \right)^{-1}\,,\quad
  C_2(L)=\left(1 -\sqrt{6} \alpha_* \right)^{-1}\,.
\end{equation*}
This finishes the proof of Theorem \ref{linBGK-decay} in 2D.
\end{proof}


\section{Linearized BGK equation in 3D}
\label{sec:linBGK:3D}

Next we shall analyze the linearized BGK equation~\eqref{linBGK:torus} in 3D:
\begin{multline*} 
 \partial_t h (x,v,t) +  v\cdot \nabla_x h(x,v,t) \\
     = M_1(v)\left[\big( \frac52 - \frac{|v|^2}{2}\big)\sigma(x,t)+v\cdot \mu(x,t)
       + \big( -\frac{1}{2}+ \frac{|v|^2}{6}\big)\tau(x,t)\right]  - h(x,v,t)\ , \quad t\geq 0 \ ,
\end{multline*}
 for the perturbation $h(x,v,t) \approx f(x,v,t) -M_1(v)$ with $x\in\tilde\T^3$, $v\in\R^3$.

Again we consider the $x$--Fourier series of $h$: 
 \begin{equation*} 
  h(x,v,t) = \sum_{k\in\Z^3} h_k(v,t)\,e^{i\frac{2\pi}{L} k\cdot x}\ .
 \end{equation*}
Each spatial mode $h_k(v,t)$ is decoupled and evolves as
\begin{equation} \label{linBGK:Fourier:3D} 
 \partial_t h_k +  i\tfrac{2\pi}{L} k\cdot v\ h_k =
      M_1(v)\left[\big( \frac52 - \frac{|v|^2}{2}\big)\sigma_k(t)+v\cdot \mu_k(t)
      + \big( -\frac{1}{2}+ \frac{|v|^2}{6}\big)\tau_k(t)\right]  - h_k(v,t)\ , \quad t\geq 0 \ .
\end{equation}
Here, $\sigma_k$, $\mu_k$ and $\tau_k$ denote the spatial modes of the $v$--moments $\sigma$, $\mu$ and $\tau$ defined in \eqref{mloc2}; hence
\[ \sigma_k := \int_{\R^3} h_k(v,t)\d[v]\ , \quad 
   \mu_k := \int_{\R^3} v\ h_k(v,t)\d[v]\ , \quad 
   \tau_k := \int_{\R^3} |v|^2\ h_k(v,t)\d[v]\ .
\]
Next we introduce an orthonormal basis in $v$-direction,
 to represent the spatial modes $h_k(\cdot,t)\in L^2(\R^3;M_1^{-1})$, $k\in\Z^3$.
Again we will use Hermite functions.
As in 2D, their multi-dimensional generalization is not unique,
 and we present two options which seem to be practical:\\

\noindent
\underline{Basis 1 (``pure tensor-basis''):}
This is a straightforward generalization of the 2D case.
  Using~\eqref{hermite-polynom:2D:a} and the normalized 1D-Hermite functions~$g_n$ ($n\in\N_0$),
  we define the normalized Hermite functions in 3D as in~\eqref{hermite-fct:dD}, 
  \begin{equation*} 
   g_m(v) := \prod_{j=1}^3 g_{m_j} (v_j) \qquad \text{for } m=(m_1,\ldots,m_3)\in\N_0^3 \ .
  \end{equation*}
  Then, $g_m$ ($m\in\N_0^3$) form an orthonormal basis of $L^2(\R^3;M_1^{-1})$
  and inherit a simple recurrence relation~\eqref{recur:dD1}.

As in 2D,
 we shall use a lexicographic order,
 i.e.\ we order $\{g_m\}$ first (increasingly) with respect to the total order $|m|$,
 and within a set of order $|m|$,
 we order first (decreasingly) with respect to $m_1$, and then $m_2$.
Thus, we obtain the linearly ordered basis
{\footnotesize
\begin{align*} 
   g_0(v) &=g_{0,0,0}(v) =M_1(v)\ , \\
   g_1(v) &=g_{1,0,0}(v) =v_1 M_1(v)\ , \quad
	 g_2(v) =g_{0,1,0}(v) =v_2 M_1(v)\ , \quad
	 g_3(v) =g_{0,0,1}(v) =v_3 M_1(v)\ , \\
   g_4(v) &=g_{2,0,0}(v) =\tfrac1{\sqrt{2}} (v_1^2 -1) M_1(v)\ , \quad
	 g_5(v) =g_{1,1,0}(v) =v_1 v_2 M_1(v)\ , \quad
	 g_6(v) =g_{1,0,1}(v) =v_1 v_3 M_1(v)\ , \\
   g_7(v) &=g_{0,2,0}(v) =\tfrac1{\sqrt{2}} (v_2^2 -1) M_1(v)\ , \quad
	 g_8(v) =g_{0,1,1}(v) =v_2 v_3 M_1(v)\ , \quad
	 g_9(v) =g_{0,0,2}(v) =\tfrac1{\sqrt{2}} (v_3^2 -1) M_1(v)\ , \\
   g_{10}(v) &=g_{3,0,0}(v) =\tfrac1{\sqrt{3!}} (v_1^3 -3v_1) M_1(v)\ , \quad
	 g_{11}(v) =g_{2,1,0}(v) =\tfrac1{\sqrt{2}} (v_1^2 -1) v_2 M_1(v)\ , \quad
	 g_{12}(v) =g_{2,0,1}(v) =\tfrac1{\sqrt{2}} (v_1^2 -1) v_3 M_1(v)\ , \\
   g_{13}(v) &=g_{1,2,0}(v) =\tfrac1{\sqrt{2}} v_1 (v_2^2 -1) M_1(v)\ , \quad
	 g_{14}(v) =g_{1,1,1}(v) = v_1 v_2 v_3 M_1(v)\ , \quad
	 g_{15}(v) =g_{1,0,2}(v) =\tfrac1{\sqrt{2}} v_1 (v_3^2 -1) M_1(v)\ , \\
   g_{16}(v) &=g_{0,3,0}(v) =\tfrac1{\sqrt{3!}} (v_2^3 -3v_2) M_1(v)\ , \quad
	 g_{17}(v) =g_{0,2,1}(v) = \tfrac1{\sqrt{2}} (v_2^2 -1) v_3 M_1(v)\ , \quad
	 g_{18}(v) =g_{0,1,2}(v) =\tfrac1{\sqrt{2}} v_2 (v_3^2 -1) M_1(v)\ , \\
   g_{19}(v) &=g_{0,0,3}(v) =\tfrac1{\sqrt{3!}} (v_3^3 -3v_3) M_1(v)\ , \quad
   g_{20}(v) =g_{4,0,0}(v) =\tfrac1{\sqrt{4!}} (v_1^4 -6v_1^2 +3) M_1(v) , \quad \ldots \ .
\end{align*}
}

\noindent
\underline{Basis 2 (``energy-basis''):}
In analogy to the 2D case from \S\ref{sec:linBGK:2D},
 it is natural to construct a basis of orthogonal polynomials $\tilde{H}_m(v)$ ($m\in\N_0^3$)
 that involves the kinetic energy polynomial $|v|^2/2$ (minus a multiple of the mass);
 in 3D the relevant term is $(|v|^2 -3)/2$.
Again, we only have to modify the Hermite polynomials of second order:
  \begin{equation*}
     \tilde{H}_m(v) =
       \begin{cases}
         \tfrac12 (H_{2,0,0}(v) +H_{0,2,0} +H_{0,0,2}(v)) = \tfrac{|v|^2}{2} -\tfrac32 & \text{if } m=(2,0,0)\ , \\
         H_{2,0,0}(v) -\tfrac12 (1+\sqrt{3}) H_{0,2,0} +\tfrac12 (\sqrt{3}-1) H_{0,0,2}(v) & \text{if } m=(0,2,0)\ , \\
         H_{2,0,0}(v) +\tfrac12 (\sqrt{3}-1) H_{0,2,0} -\tfrac12 (1+\sqrt{3}) H_{0,0,2}(v) & \text{if } m=(0,0,2)\ , \\
         H_m(v) & \text{else.}
       \end{cases}
  \end{equation*}
Similarly, we define normalized Hermite functions
  \begin{equation*}
     \tilde{g}_m(v) =
       \begin{cases}
         \frac1{\sqrt3} \big(g_{2,0,0}(v) +g_{0,2,0}(v) +g_{0,0,2}(v)\big) & \text{if } m=(2,0,0)\ , \\
         \frac1{\sqrt3} \big(g_{2,0,0}(v) -\tfrac12 (1+\sqrt{3}) g_{0,2,0}(v) +\tfrac12 (\sqrt{3}-1) g_{0,0,2}(v)\big) & \text{if } m=(0,2,0)\ , \\
         \frac1{\sqrt3} \big(g_{2,0,0}(v) +\tfrac12 (\sqrt{3}-1) g_{0,2,0}(v) -\tfrac12 (1+\sqrt{3}) g_{0,0,2}(v)\big) & \text{if } m=(0,0,2)\ , \\
         g_m(v) & \text{else.}
       \end{cases}
  \end{equation*}
We remark that it is most convenient to obtain $\tilde{g}_{0,2,0}$ and $\tilde{g}_{0,0,2}$
 from diagonalizing the matrix $\LL_2$ (see~\eqref{matrix:L2:3D} below).
The elements $\tilde{g}_m$ satisfy a recurrence relation similar to~\eqref{recur:dD1};
	 except for identities involving $\tilde{g}_{2,0,0}$, $\tilde{g}_{0,2,0}$ or $\tilde{g}_{0,0,2}$.
	For example,
	 \begin{align*}
	    v_1 \tilde{g}_{2,0,0}(v)
	       &= \frac1{\sqrt3} v_1 \big(g_{2,0,0}(v) +g_{0,2,0}(v) +g_{0,0,2}(v)\big) \\
	       &= \frac1{\sqrt3} \big(\sqrt{3} g_{3,0,0}(v) +\sqrt{2} g_{1,0,0}(v) +g_{1,2,0}(v) +g_{1,0,2}(v)\big) \ .
	 \end{align*}
Whereas the first basis $g_m(v)$ ($m\in\N_0^3$) inherits a simple recurrence formula with three elements;
 for the second basis $\tilde{g}_m(v)$ ($m\in\N_0^3$) the recurrence formula for $\tilde{g}_{2,0,0}(v)$ relates five elements.

\bigskip
To derive the vector representation of \eqref{linBGK:Fourier:3D}, 
 it is preferable to use the first basis with the linear ordering $g_m(v)$, $m\in\N_0$. 
With the identity $(g_4(v) +g_7(v) +g_9(v))/\sqrt{2} = (|v|^2 -3) M_1(v)/2$,
 we rewrite~\eqref{linBGK:Fourier:3D} as
 \begin{equation} \label{linBGK:Fourier:3D:2} 
 \partial_t h_k +  i\tfrac{2\pi}{L} k\cdot v\ h_k =
   g_0(v) \sigma_k(t) 
    + \begin{pmatrix} g_1(v) \\ g_2(v) \\ g_3(v) \end{pmatrix}\cdot \mu_k(t)
    + \big( g_4(v) +g_7(v) +g_9(v)\big) \frac{\tau_k(t) -3\sigma_k(t)}{3\sqrt{2}}
    - h_k(v,t)
 \end{equation}
 for $t\geq 0$.
As in 1D, the spatially homogeneous mode again satisfies $\ddt h_0(v,t)=-h_0(v,t)$,
 cf.\ \eqref{h0-decay}.
Next we expand $h_k(\cdot,t)\in L^2(\R^3;M_1^{-1})$ in the orthonormal basis $g_m$ ($m\in\N_0$):
 $$
   h_k(v,t)=\sum_{m=0}^\infty \hat h_{k,m}(t)\,g_m(v) \qquad\text{with}\quad
   \hat h_{k,m} = \langle h_k(v),g_m(v)\rangle_{L^2(M_1^{-1})}\ .
 $$
 For each spatial mode $k\in\Z^3$,
  the ``infinite vector'' $\hat {\bf h}_k(t) = ( \hat h_{k,0}(t),\,\hat h_{k,1}(t),\ ...)^\top\in \ell^2(\N_0)$ contains all Hermite coefficients of $h_k(\cdot,t)$. 
 In particular, we have
 \[ \hat h_{k,0} = \sigma_k \ , \qquad 
    \big(\hat h_{k,1}\,,\ \hat h_{k,2}\,,\ \hat h_{k,3} \big)^\top = \mu_k \in\R^3 \ , \qquad
    \tfrac1{\sqrt{2}} (\hat h_{k,4} +\hat h_{k,7} +\hat h_{k,9}) = \frac12\tau_k -\frac32 \sigma_k\ .
 \]
 Thus, we can rewrite~\eqref{linBGK:Fourier:3D:2} as
 \begin{equation} \label{linBGK:Fourier:3D:3} 
 \partial_t h_k +  i\tfrac{2\pi}{L} k\cdot v\ h_k
    = g_0\ \hat h_{k,0} 
      + \begin{pmatrix} g_1 \\ g_2 \\ g_3 \end{pmatrix}\cdot \begin{pmatrix} \hat h_{k,1} \\ \hat h_{k,2} \\ \hat h_{k,3} \end{pmatrix}
      +\tfrac1{3} \big( g_4 +g_7 +g_9\big) (\hat h_{k,4} +\hat h_{k,7} +\hat h_{k,9})
      - h_k(v,t)\ ,
\end{equation}
for $k\in\Z^3$, $t\geq 0$.
Since~\eqref{linBGK:Fourier:3D:3} is rotationally invariant (as in 2D),
 it suffices to consider \eqref{linBGK:Fourier:3D:3} for vectors $k=(\kappa,0)^\top$ with the discrete moduli 
$$
  \kappa\in K:=\big\{r\ge1 \, \big|\, \exists k\in\Z^3\setminus\{0\} \mbox{ with } r=|k| \big\} \ . 
$$
With the notation $h_\kappa:=h_{\kappa,0}$ and $\hat {\bf h}_\kappa:=\hat {\bf h}_{\kappa,0}$,
 \eqref{linBGK:Fourier:3D:3} reads
\[ \partial_t h_\kappa +  i\tfrac{2\pi}{L} \kappa v_1\ h_\kappa
   = g_0 \hat h_{\kappa,0} 
     + \begin{pmatrix} g_1 \\ g_2 \\ g_3 \end{pmatrix}\cdot \begin{pmatrix} \hat h_{\kappa,1} \\ \hat h_{\kappa,2} \\ \hat h_{\kappa,3} \end{pmatrix}
     +\tfrac1{3} \big( g_4 +g_7 +g_9\big) (\hat h_{\kappa,4} +\hat h_{\kappa,7} +\hat h_{\kappa,9})
     - h_\kappa \ , \quad
     \kappa\in K\ , \quad t\geq 0 \ .
\]
Then, the vector of its Hermite coefficients satisfies
\begin{equation}\label{linap4:3D}
\partial_t\hat {\bf h}_\kappa (t) + i\tfrac{2\pi}{L}\,\kappa \LL_1 \hat{\bf h}_\kappa (t) = -\LL_2 \hat{\bf h}_\kappa (t)\ ,\quad \kappa\in K\ ,\quad t\ge0\ ,
\end{equation}
where the operators $\LL_1,\,\LL_2$ are represented by ``infinite matrices'' on $\ell^2(\N_0)$:
\setcounter{MaxMatrixCols}{22}
\begin{equation*}
  \LL_1= \begin{pmatrix}
  0 & 1 & 0 & 0 & 0 & 0 & 0 & 0 & 0 & 0 & 0 & 0 & 0 & 0 & 0 & 0 & 0 & 0 & 0 & 0 & \cdots \\
  1 & 0 & 0 & 0 & \sqrt 2 & 0 & 0 & 0 & 0 & 0 & 0 & 0 & 0 & 0 & 0 & 0 & 0 & 0 & 0 & 0 & \cdots \\
  0 & 0 & 0 & 0 & 0 & 1 & 0 & 0 & 0 & 0 & 0 & 0 & 0 & 0 & 0 & 0 & 0 & 0 & 0 & 0 & \cdots \\
  0 & 0 & 0 & 0 & 0 & 0 & 1 & 0 & 0 & 0 & 0 & 0 & 0 & 0 & 0 & 0 & 0 & 0 & 0 & 0 & \cdots \\
  0 & \sqrt 2 & 0 & 0 & 0 & 0 & 0 & 0 & 0 & 0 & \sqrt 3 & 0 & 0 & 0 & 0 & 0 & 0 & 0 & 0 & 0 & \cdots \\
  0 & 0 & 1 & 0 & 0 & 0 & 0 & 0 & 0 & 0 & 0 & \sqrt 2 & 0 & 0 & 0 & 0 & 0 & 0 & 0 & 0 & \cdots \\
  0 & 0 & 0 & 1 & 0 & 0 & 0 & 0 & 0 & 0 & 0 & 0 & \sqrt 2 & 0 & 0 & 0 & 0 & 0 & 0 & 0 & \cdots \\
  0 & 0 & 0 & 0 & 0 & 0 & 0 & 0 & 0 & 0 & 0 & 0 & 0 & 1 & 0 & 0 & 0 & 0 & 0 & 0 & \cdots \\
  0 & 0 & 0 & 0 & 0 & 0 & 0 & 0 & 0 & 0 & 0 & 0 & 0 & 0 & 1 & 0 & 0 & 0 & 0 & 0 & \cdots \\
  0 & 0 & 0 & 0 & 0 & 0 & 0 & 0 & 0 & 0 & 0 & 0 & 0 & 0 & 0 & 1 & 0 & 0 & 0 & 0 & \cdots \\
  0 & 0 & 0 & 0 & \sqrt 3 & 0 & 0 & 0 & 0 & 0 & 0 & 0 & 0 & 0 & 0 & 0 & 0 & 0 & 0 & 0 & \cdots \\
  0 & 0 & 0 & 0 & 0 & \sqrt 2 & 0 & 0 & 0 & 0 & 0 & 0 & 0 & 0 & 0 & 0 & 0 & 0 & 0 & 0 & \cdots \\
  0 & 0 & 0 & 0 & 0 & 0 & \sqrt 2 & 0 & 0 & 0 & 0 & 0 & 0 & 0 & 0 & 0 & 0 & 0 & 0 & 0 & \cdots \\
  0 & 0 & 0 & 0 & 0 & 0 & 0 & 1 & 0 & 0 & 0 & 0 & 0 & 0 & 0 & 0 & 0 & 0 & 0 & 0 & \cdots \\
  0 & 0 & 0 & 0 & 0 & 0 & 0 & 0 & 1 & 0 & 0 & 0 & 0 & 0 & 0 & 0 & 0 & 0 & 0 & 0 & \cdots \\
  0 & 0 & 0 & 0 & 0 & 0 & 0 & 0 & 0 & 1 & 0 & 0 & 0 & 0 & 0 & 0 & 0 & 0 & 0 & 0 & \cdots \\
  0 & 0 & 0 & 0 & 0 & 0 & 0 & 0 & 0 & 0 & 0 & 0 & 0 & 0 & 0 & 0 & 0 & 0 & 0 & 0 & \cdots \\
  0 & 0 & 0 & 0 & 0 & 0 & 0 & 0 & 0 & 0 & 0 & 0 & 0 & 0 & 0 & 0 & 0 & 0 & 0 & 0 & \cdots \\
  0 & 0 & 0 & 0 & 0 & 0 & 0 & 0 & 0 & 0 & 0 & 0 & 0 & 0 & 0 & 0 & 0 & 0 & 0 & 0 & \cdots \\
  0 & 0 & 0 & 0 & 0 & 0 & 0 & 0 & 0 & 0 & 0 & 0 & 0 & 0 & 0 & 0 & 0 & 0 & 0 & 0 & \cdots \\
  \vdots & \vdots &\vdots &\vdots &\vdots &\vdots &\vdots &\vdots &\vdots & \vdots & \vdots &\vdots &\vdots &\vdots &\vdots &\vdots &\vdots &\vdots &\vdots &\vdots & \ddots 
  \end{pmatrix}\ ,
\end{equation*}
\begin{equation} \label{matrix:L2:3D}
 \LL_2=\diag(0,\, 0,\, 0,\, 0,\, \tfrac13
              \begin{pmatrix}
               2 & 0 & 0 & -1 & 0 & -1 \\
               0 & 3 & 0 & 0 & 0 & 0 \\
               0 & 0 & 3 & 0 & 0 & 0 \\
               -1 & 0 & 0 & 2 & 0 & -1 \\
               0 & 0 & 0 & 0 & 3 & 0 \\
               -1 & 0 & 0 & -1 & 0 & 2 \\
              \end{pmatrix},\, 1,\, 1,\cdots)\ .
\end{equation}

To determine the hypocoercivity index of the BGK model in 3D,
 it is preferable to use the second basis with the linear ordering $\tilde{g}_m(v)$, $m\in\N_0$.
Again, we shall give the matrix representation of the BGK equation~\eqref{linBGK:Fourier:3D} in the second velocity basis only for the spatial modes $k=(\kappa,0)^\top$, $\kappa\in K$.
To obtain the corresponding matrices $\tilde\LL_1$ and $\tilde\LL_2$, 
 we simply represent the linear basis transformation by the infinite matrix 
 \[
  \Smatrix = \diag (1,\ 1,\ 1,\, 1,\ 
                    \begin{pmatrix}
                     1/\sqrt3 & 0 & 0 & 1/\sqrt3 & 0 & 1/\sqrt3 \\
                     0 & 1 & 0 & 0 & 0 & 0 \\
                     0 & 0 & 1 & 0 & 0 & 0 \\
                     1/\sqrt3 & 0 & 0 & -(1+1/\sqrt3)/2 & 0 & (1-1/\sqrt3)/2 \\
                     0 & 0 & 0 & 0 & 1 & 0\\
                     1/\sqrt3 & 0 & 0 & (1-1/\sqrt3)/2 & 0 & -(1+1/\sqrt3)/2 \\
                    \end{pmatrix},\ 1,\ 1,\cdots)\ .
 \]
 which is self-inverse, i.e. $\Smatrix=\Smatrix^{-1}$.
Thus we compute $\tilde{\LL}_1 = \Smatrix^{-1} \LL_1 \Smatrix$ and $\tilde{\LL}_2 = \Smatrix^{-1} \LL_2 \Smatrix$, which yields
{\footnotesize
\begin{equation*}
  \tilde{\LL}_1= \begin{pmatrix}
  0 & 1 & 0 & 0 & 0 & 0 & 0 & 0 & 0 & 0 & 0 & 0 & 0 & 0 & 0 & 0 & 0 & 0 & 0 & 0 & 0 & \cdots \\
  1 & 0 & 0 & 0 & \tfrac{\sqrt{2}}{\sqrt{3}} & 0 & 0 & \tfrac{\sqrt{2}}{\sqrt{3}} & 0 & \tfrac{\sqrt{2}}{\sqrt{3}} & 0 & 0 & 0 & 0 & 0 & 0 & 0 & 0 & 0 & 0 & 0 & \cdots \\
  0 & 0 & 0 & 0 & 0 & 1 & 0 & 0 & 0 & 0 & 0 & 0 & 0 & 0 & 0 & 0 & 0 & 0 & 0 & 0 & 0 & \cdots \\
  0 & 0 & 0 & 0 & 0 & 0 & 1 & 0 & 0 & 0 & 0 & 0 & 0 & 0 & 0 & 0 & 0 & 0 & 0 & 0 & 0 & \cdots \\
  0 & \tfrac{\sqrt{2}}{\sqrt{3}} & 0 & 0 & 0 & 0 & 0 & 0 & 0 & 0 & 1 & 0 & 0 & \tfrac1{\sqrt{3}} & 0 & \tfrac1{\sqrt{3}} & 0 & 0 & 0 & 0 & 0 & \cdots \\
  0 & 0 & 1 & 0 & 0 & 0 & 0 & 0 & 0 & 0 & 0 & \sqrt{2} & 0 & 0 & 0 & 0 & 0 & 0 & 0 & 0 & 0 & \cdots \\
  0 & 0 & 0 & 1 & 0 & 0 & 0 & 0 & 0 & 0 & 0 & 0 & \sqrt{2} & 0 & 0 & 0 & 0 & 0 & 0 & 0 & 0 & \cdots \\
  0 & \tfrac{\sqrt{2}}{\sqrt{3}} & 0 & 0 & 0 & 0 & 0 & 0 & 0 & 0 & 1 & 0 & 0 & -\tfrac{3+\sqrt{3}}{6} & 0 & \tfrac{3-\sqrt{3}}{6} & 0 & 0 & 0 & 0 & 0 & \cdots \\
  0 & 0 & 0 & 0 & 0 & 0 & 0 & 0 & 0 & 0 & 0 & 0 & 0 & 0 & 1 & 0 & 0 & 0 & 0 & 0 & 0 & \cdots \\
  0 & \tfrac{\sqrt{2}}{\sqrt{3}} & 0 & 0 & 0 & 0 & 0 & 0 & 0 & 0 & 1 & 0 & 0 & \tfrac{3-\sqrt{3}}{6} & 0 & -\tfrac{3+\sqrt{3}}{6} & 0 & 0 & 0 & 0 & 0 & \cdots \\
  0 & 0 & 0 & 0 & 1 & 0 & 0 & 1 & 0 & 1 & 0 & 0 & 0 & 0 & 0 & 0 & 0 & 0 & 0 & 0 & 2 & \cdots \\
  0 & 0 & 0 & 0 & 0 & \sqrt{2} & 0 & 0 & 0 & 0 & 0 & 0 & 0 & 0 & 0 & 0 & 0 & 0 & 0 & 0 & 0 & \cdots \\
  0 & 0 & 0 & 0 & 0 & 0 & \sqrt{2} & 0 & 0 & 0 & 0 & 0 & 0 & 0 & 0 & 0 & 0 & 0 & 0 & 0 & 0 & \cdots \\
  0 & 0 & 0 & 0 & \tfrac1{\sqrt{3}} & 0 & 0 & -\tfrac{3+\sqrt{3}}{6} & 0 & \tfrac{3-\sqrt{3}}{6} & 0 & 0 & 0 & 0 & 0 & 0 & 0 & 0 & 0 & 0 & 0 & \cdots \\
  0 & 0 & 0 & 0 & 0 & 0 & 0 & 0 & 1 & 0 & 0 & 0 & 0 & 0 & 0 & 0 & 0 & 0 & 0 & 0 & 0 & \cdots \\
  0 & 0 & 0 & 0 & \tfrac1{\sqrt{3}} & 0 & 0 & \tfrac{3-\sqrt{3}}{6} & 0 & -\tfrac{3+\sqrt{3}}{6} & 0 & 0 & 0 & 0 & 0 & 0 & 0 & 0 & 0 & 0 & 0 & \cdots \\
  0 & 0 & 0 & 0 & 0 & 0 & 0 & 0 & 0 & 0 & 0 & 0 & 0 & 0 & 0 & 0 & 0 & 0 & 0 & 0 & 0 & \cdots \\
  0 & 0 & 0 & 0 & 0 & 0 & 0 & 0 & 0 & 0 & 0 & 0 & 0 & 0 & 0 & 0 & 0 & 0 & 0 & 0 & 0 & \cdots \\
  0 & 0 & 0 & 0 & 0 & 0 & 0 & 0 & 0 & 0 & 0 & 0 & 0 & 0 & 0 & 0 & 0 & 0 & 0 & 0 & 0 & \cdots \\
  0 & 0 & 0 & 0 & 0 & 0 & 0 & 0 & 0 & 0 & 0 & 0 & 0 & 0 & 0 & 0 & 0 & 0 & 0 & 0 & 0 & \cdots \\
  0 & 0 & 0 & 0 & 0 & 0 & 0 & 0 & 0 & 0 & 2 & 0 & 0 & 0 & 0 & 0 & 0 & 0 & 0 & 0 & 0 & \cdots \\
  \vdots & \vdots &\vdots &\vdots &\vdots &\vdots &\vdots &\vdots &\vdots & \vdots & \vdots &\vdots &\vdots &\vdots &\vdots &\vdots &\vdots &\vdots &\vdots &\vdots &\vdots &\ddots 
  \end{pmatrix}
\end{equation*}
}
\[ \tilde{\LL}_2=\diag(0,\, 0,\, 0,\, 0,\, 0,\, 1,\, 1,\, 1,\cdots)\ . \]
To determine the hypocoercivity index of the BGK model in 3D,
we use Definition \ref{hyp-index} in its equivalent formulation \ref{cond:trivkernel}: 
With $\ker\tilde\LL_2=\spn\{e_0,\,e_1,\,e_2,\,e_3,\,e_4\}$, and the relations 
\[
  \tilde\LL_1 e_0=e_1\,,\quad
  \tilde\LL_1 e_1=e_0+\tfrac{\sqrt{2}}{\sqrt{3}} (e_4+e_7+e_9)\,,\quad
  \tilde\LL_1 e_2=e_5\,,\quad
  \tilde\LL_1 e_3=e_6\,,\quad 
  \tilde\LL_1 e_4=\tfrac{\sqrt{2}}{\sqrt{3}} e_1+e_{10} +\tfrac1{\sqrt{3}} (e_{13} +e_{15})\,,  
\]
we find the index $\tau=2$ (like in 2D). 
Each of the three momentum-conservation modes (represented by $e_1$, $e_2$ and $e_3$) is directly coupled to a decreasing mode.\\ 

We define the matrices $\CC_\kappa := i\tfrac{2\pi}{L} \kappa\ \tilde\LL_1 + \tilde\LL_2$, $\kappa\in K\cup\{0\}$,
which determine the evolution of the spatial modes of the BGK equation in 3D, cf.\ \eqref{linap4:3D}.
Our next goal is to establish a spectral gap of $\CC_\kappa$, uniformly in $\kappa\in K$. 
This will prove Theorem \ref{linBGK-decay} in 3D.
To this end we make an ansatz for the transformation matrices: Following the detailed motivation from the 1D analog in \S\ref{sec:linBGK:1D}, let $\P_\kappa$, $\kappa\in K$ be the identity matrix 
whose upper-left $11\times 11$ block is replaced by
 \begin{equation}\label{P:3D}
  \begin{pmatrix}
  1 & -i \alpha/\kappa & 0 & 0 & 0 & 0 & 0 & 0 & 0 & 0 & 0 \\
  i \alpha/\kappa & 1 & 0 & 0 & 0 & 0 & 0 & -i \beta/\kappa & 0 & 0 & 0 \\
  0 & 0 & 1 & 0 & 0 & -i \gamma/\kappa & 0 & 0 & 0 & 0 & 0 \\
  0 & 0 & 0 & 1 & 0 & 0 & -i \omega/\kappa & 0 & 0 & 0 & 0 \\
  0 & 0 & 0 & 0 & 1 & 0 & 0 & 0 & 0 & 0 & -i \eta/\kappa \\
  0 & 0 & i \gamma/\kappa & 0 & 0 & 1 & 0 & 0 & 0 & 0 & 0 \\
  0 & 0 & 0 & i \omega/\kappa & 0 & 0 & 1 & 0 & 0 & 0 & 0 \\
  0 & i \beta/\kappa & 0 & 0 & 0 & 0 & 0 & 1 & 0 & 0 & 0 \\
  0 & 0 & 0 & 0 & 0 & 0 & 0 & 0 & 1 & 0 & 0 \\
  0 & 0 & 0 & 0 & 0 & 0 & 0 & 0 & 0 & 1 & 0 \\
  0 & 0 & 0 & 0 & i \eta/\kappa & 0 & 0 & 0 & 0 & 0 & 1 \\
  \end{pmatrix}
 \end{equation}
 with positive parameters $\alpha$, $\beta$, $\gamma$, $\omega$, and $\eta$ to be chosen below.
The distribution of the non-zero off-diagonal elements follows from the pattern in matrix $\tilde\LL_1$.

\begin{lemma} \label{lemma:decay:3D}
 If the matrices $\P_\kappa$ are chosen as~\eqref{P:3D} with $\beta =\sqrt3  \alpha$, $\gamma =\alpha$, $\omega=\alpha$, and $\eta=\alpha$ uniformly for all $\kappa\in K$,
 then there exists a positive $\alpha_+$ such that $\P_\kappa$ and $\CC_\kappa^*\P_\kappa + \P_\kappa\CC_\kappa$ are positive definite for all $\alpha\in(0,\alpha_+)$ and for all $\kappa\in K$.
Moreover,
\begin{equation}\label{goodbnd:3D}
\CC_\kappa^*\P_\kappa + \P_\kappa\CC_\kappa  \geq 2\mu \P_\kappa \qquad \text{ uniformly in } \kappa\in K \ , 
\end{equation}
with  
\begin{equation*} 
\mu := \left( \frac{20}{32} \right)^{20} \frac{\delta_{21}(1,\alpha,\sqrt3  \alpha,\alpha,\alpha,\alpha)}{2 (1+2\alpha)} > 0 \ , 
\end{equation*}
 where $\delta_{21}(1,\alpha,\sqrt3  \alpha,\alpha,\alpha,\alpha):= \det \DD_{1,\alpha,\sqrt3  \alpha,\alpha,\alpha,\alpha}$ with $\DD_{\kappa,\alpha,\sqrt3  \alpha,\alpha,\alpha,\alpha}$ defined in~\eqref{Dka5a4aaa}.
\end{lemma} 

The proof of this lemma is deferred to Appendix~\ref{A1}.

\begin{remark}\label{rem:mu:3D}
\begin{enumerate}[label=(\alph*)]
\item Consider 
 \[ \alpha_* = \argmax_{\alpha\in[0,\alpha_+]} \frac{\delta_{21}(1,\alpha,\sqrt3 \alpha,\alpha,\alpha,\alpha)}{2 (1+2\alpha)} \ . \]
 Choosing $\P_\kappa$ with $\alpha=\alpha_*$, $\beta =\sqrt3  \alpha$, $\gamma =\alpha$, $\omega=\alpha$, and $\eta=\alpha$ uniformly for all $\kappa\in K$, yields~\eqref{goodbnd:3D} with the maximal constant
 \begin{equation}\label{goodbnd2:3D:b}
  \mu = \left( \frac{20}{32} \right)^{20} \frac{\delta_{21}(1,\alpha_*,\sqrt3 \alpha_*,\alpha_*,\alpha_*,\alpha_*)}{2 (1+2\alpha_*)}
      = \max_{\alpha\in[0,\alpha_+]} \left( \frac{20}{32} \right)^{20} \frac{\delta_{21}(1,\alpha,\sqrt3 \alpha,\alpha,\alpha,\alpha)}{2 (1+2\alpha)} \ .
 \end{equation}
\item \label{rem:mu:3D:2Pi}
 For $L=2\pi$, we compute $\alpha_+ = 0,214287873283229$.
 Moreover, the constant $\mu$ is determined as $\mu = 0,0001774540949$ with $\alpha_* = 0,1644256115$.
\item In the limit $L\to +\infty$, the matrix $\CC_\kappa^*\P_\kappa + \P_\kappa\CC_\kappa$ has zero eigenvalues,
 which is apparent from its upper left submatrix $\DD_{\kappa,\alpha,\sqrt3 \alpha,\alpha,\alpha,\alpha}$ defined in~\eqref{Dka5a4aaa}.
 Accordingly, $\alpha_+ \to 0$ 
  in the limit $L\to \infty$.
 
 Moreover $\alpha_+ \to 0$ in the limit $L\to 0$. 
\end{enumerate}
\end{remark}

\begin{proof}[Proof of Theorem~\ref{linBGK-decay} in 3D]
We consider a solution $h$ of \eqref{linBGK:torus}, and the entropy functional $\cE(\tilde{f})$ defined by
\begin{equation}\label{entropy:3D}
  \cE(\tilde{f}) := \sum_{k\in \Z^3} \langle h_k(v), \P_{|k|} h_k(v)\rangle_{L^2(M_1^{-1})}\ ,
\end{equation}
with $\tilde{f}(t) := M_1 +h(t)$.
Here the matrices $\P_0 = \II$ and $\P_\kappa$ defined in \eqref{P:3D} for $\kappa=|k|\neq 0$
are regarded as bounded operators on $L^2(M_1^{-1})$.
Then 
\begin{equation}\label{e-inequality:3D}
  \ddt \cE(\tilde{f}) = -\sum_{k\in \Z^3} \langle h_k(v), (\CC_{|k|}^*\P_{|k|} + \P_{|k|}\CC_{|k|}) h_k(v)\rangle_{L^2(M_1^{-1})} 
   \leq - 2\min\{1,\mu\} \,\cE(\tilde{f}) \ , 
\end{equation}
where $1$ is the decay rate of $h_0$, cf.~\eqref{h0-decay}.
This implies~\eqref{ineq:linBGK-decay} with $\lambda^3(L):=2\min\{1,\mu\}$ and $\mu$ from~\eqref{goodbnd2:3D:b}.

The constants $c_3$ and $C_3$ in the estimate \eqref{entropy-equiv} follow from \eqref{Pbound3}: 
\begin{equation*} 
  c_3(L)=\left(1 +2 \alpha_* \right)^{-1}\,,\quad
  C_3(L)=\left(1 -2 \alpha_* \right)^{-1}\,.
\end{equation*}
This finishes the proof of Theorem \ref{linBGK-decay} in 3D.
\end{proof}


\section{Local exponential stability for the BGK equation in 3D}
\label{sec:6}

This analysis is an extension of \S4.5 in \cite{AAC16}. To make is self-contained we give the complete proof and not only the modification of the key steps.

For $\gamma \ge 0$, let $H^\gamma(\tilde\T^3)$ be the Sobolev space consisting of  the completion  of smooth functions $\varphi$ on $\tilde\T^3$ in the Hilbertian norm 
\[ \|\varphi\|_{H^\gamma}^2 :=   \sum_{k\in \Z^3} (1+|k|^2)^{\gamma}|\varphi_k|^2 \,, \]
where $\varphi_k$ ($k\in\Z^3$) is the $k$th Fourier coefficient of $\varphi$. Let $\mathcal{H}_\gamma$ denote the Hilbert space
$H^\gamma(\tilde\T^3)\otimes L^2(\R^3;M_1^{-1}(v)\d[v])$, where the inner product in $\mathcal{H}_\gamma$ is given by
\[ \langle f,g\rangle_{\mathcal{H}_\gamma} = \int_{\tilde\T^3} \int_{\R^3} \overline{f}(x,v) \left[\left(1 - \Delta_x \right)^{\gamma} g(x,v) \right]M_1^{-1}(v) \d[v] \d[\tilde x]  \,, \]
 where $\d[\tilde x]$ denotes the normalized Lebesgue measure on $\tilde\T^3$.
Then $\mathcal{H}_0$ is simply the weighted space $L^2(\tilde\T^3\times\R^3;M_1^{-1}(v)\d[\tilde x]\d[v])$.

\begin{proof}[Proof of Theorem~\ref{BGK-decay}]
\ref{BGK-decay:a}
For any solution $h(t)$ to \eqref{linBGK:torus} with $\cE^d(h^I+M_1)<\infty$, normalized according to \eqref{mloc2},
 we consider the function $\tilde f(t):=h(t) +f^\infty$ with $f^\infty=M_1$.
We define a family of entropy functionals $\cE_\gamma$ ($\gamma\geq 0$) by
 \begin{equation} \label{entropy:gamma:3D}
  \cE_\gamma(\tilde f) := \sum_{k\in \Z^3} (1+|k|^2)^{\gamma} \langle h_k(v), \P_{|k|} h_k(v)\rangle_{L^2(M_1^{-1})}\ ,
 \end{equation}
 as an extension of the entropy $\cE(\tilde f)$ in~\eqref{entropy:3D}. 
For all $\gamma\geq 0$, the estimates
\begin{equation} \label{equiv:norm:entropy:gamma} 
 \tfrac34 \cE_\gamma(\tilde f)
 \leq \tfrac1{1+2\alpha_*} \cE_\gamma(\tilde f)
 \leq \| h\|_{\mathcal{H}_\gamma}^2
 \leq \tfrac1{1-2\alpha_*} \cE_\gamma(\tilde f)
 \leq \tfrac32 \cE_\gamma(\tilde f) 
\end{equation}
follow from~\eqref{Pbound3} and Remark \ref{rem:mu:3D}\ref{rem:mu:3D:2Pi}. 
Moreover, the second statement in Theorem~\ref{BGK-decay}\ref{BGK-decay:a} follows just as in the proof of Theorem~\ref{linBGK-decay} in 3D
 where the numerical values are chosen according to Remark \ref{rem:mu:3D}\ref{rem:mu:3D:2Pi}.

\ref{BGK-decay:b}
Let $f$ be a solution of the BGK equation~\eqref{bgk} with constant temperature $T=1$ 
 and define $h(x,v,t) := f(x,v,t) - M_1(v)$ as in the introduction.
Moreover, let 
$\sigma$, $\mu$ and $\tau$ be defined in terms of~$f$ as in \eqref{f:perturb}.  
For all $\gamma\geq 0$, $\|\sigma\|_{H^\gamma}^2  = \langle \sigma M_1, f - f^\infty\rangle_{\mathcal{H}_\gamma}$ with $f^\infty =M_1$.
Then by the Cauchy-Schwarz inequality,
\begin{equation}\label{locbound1}
\|\sigma\|_{H^\gamma}^2 \leq  \|\sigma M_1\|_{\mathcal{H}_\gamma} \| f-f^\infty\|_{\mathcal{H}_\gamma}  =
\|\sigma\|_{H^\gamma} \| f-f^\infty\|_{\mathcal{H}_\gamma} \ .
\end{equation} 
Likewise, $\|\mu\|_{H^\gamma}^2  = \langle \mu\cdot v M_1, f - f^\infty\rangle_{\mathcal{H}_\gamma}$, and hence
\begin{equation}
\|\mu\|_{H^\gamma}^2 \leq  \|\mu\cdot v M_1\|_{\mathcal{H}_\gamma} \| f-f^\infty\|_{\mathcal{H}_\gamma}  \le
\sqrt{3} \|\mu\|_{H^\gamma} \| f-f^\infty\|_{\mathcal{H}_\gamma} \ ,
\end{equation}
as well as, $\|\tau\|_{H^\gamma}^2  = \langle \tau |v|^2 M_1, f - f^\infty\rangle_{\mathcal{H}_\gamma}$, and hence
\begin{equation}\label{locbound2}
\|\tau\|_{H^\gamma}^2 \leq  \|\tau |v|^2 M_1\|_{\mathcal{H}_\gamma} \| f-f^\infty\|_{\mathcal{H}_\gamma}  =
\sqrt{15} \|\tau\|_{H^\gamma} \| f-f^\infty\|_{\mathcal{H}_\gamma} \ .
\end{equation} 
Using a Sobolev embedding (with $\gamma>3/2$) we can estimate the perturbations of the first 3 moments in $L^\infty(\tilde\T^3)$ as
\begin{equation}\label{locbound3}
\|\sigma\|_\infty
\leq C_\gamma  \| f-f^\infty\|_{\mathcal{H}_\gamma}\,, \qquad
\|\mu\|_\infty\leq C_\gamma  \| f-f^\infty\|_{\mathcal{H}_\gamma} \quad
{\rm and}\qquad 
\|\tau\|_\infty 
\leq C_\gamma  \| f-f^\infty\|_{\mathcal{H}_\gamma} \ .
\end{equation}
Using these estimates it is a simple matter to control the approximation in (\ref{taylor}): For $s\in [0,1]$ and $(x,v) \in \tilde\T^3\times\R^3$, define (inspired by \eqref{taylor})
\begin{equation}\label{def-F}
   F(s,x,v) := \frac{(1+s\sigma)^{\frac{5}2}(x)}{\left(2\pi \big\{1+\frac13\big[s\tau(x)-\frac{s^2|\mu|^2(x)}{1+s\sigma(x)}\big]\big\}\right)^\frac{3}2} 
 \,\exp\Big\{-\frac{|v(1+s\sigma(x))-s\mu(x)|^2}{2\big(1+\frac13\big[s\tau(x)-\frac{s^2|\mu|^2(x)}{1+s\sigma(x)}\big]\big) (1+s\sigma(x))}\Big\} \ , 
\end{equation}
so that the gain term in the linearized BGK equation \eqref{linBGK:torus} is
${\displaystyle \partial_s F(0,x,v)}$.
In this notation,
\begin{align*}
 R_f(x,v)
  &:= M_f(x,v) - M_1(v) - \left[ \left( \frac52 - \frac{|v|^2}{2}\right)\sigma(x) +v\cdot\mu(x)
      + \left( -\frac{1}{2}+ \frac{|v|^2}{6}\right)\tau(x)\right] M_1(v) \\
  & = \int_0^1 \left[ \partial_s F(s,x,v) - \partial_s F(0,x,v)\right]{\rm d}s
      = \int_0^1 \int_0^s \left[ \partial_s^2 F(r,x,v)\right] {\rm d}r\, {\rm d}s\ .
\end{align*}
To display the complicated expression for $\partial_s^2 F(s,x,v)$, we define
\[
 \rho_s := 1+s\sigma\ , \quad u_s := \frac{s}{\rho_s} \mu\ , \quad
 \mu_s := s\mu\ , \quad 
 P_s := 1 +\frac13 \Big(s\tau - \frac{|\mu_s|^2}{\rho_s}\Big)\ . 
\]
Then $\partial_s^2 F(s,x,v)$ reads
\begin{align*} 
\frac{\partial_s^2 F(s,x,v)}{F(s,x,v)}
&= 
 \left( \frac{15}{4 \ \rho_s^2} -\frac5{2\ P_s \ \rho_s} \ | v -u_s|^2
 \right) \sigma^2
+\left( -\frac{15}{2\ P_s \ \rho_s}
  +\frac5{P_s^2} \, | v -u_s|^2
 \right) \sigma \ \partial_s P_s \\
&\quad  
+\left( \frac {5}{P_s \ \rho_s^2}
 \right) \sigma \ \left( ( v -u_s) \cdot \mu \right)
+\left( \frac{15}{4\ P_s^2}
  -\frac{5\ \rho_s}{2\ P_s^{3}}\ | v -u_s|^2
 \right) \left( \partial_s P_s \right)^2 \\
&\quad
+\left( -\frac5{P_s^2\ \rho_s}
 \right)\partial_s P_s \ \left( ( v -u_s) \cdot \mu \right)
+\left( 
 -\frac1{3 P_s^2\ \rho_s^2} | v -u_s|^2
 \right) |\mu|^2 \\
&\quad
+ \left( {\frac {1}{P_s \ \rho_s} ( v -u_s) \cdot \mu }
    -\tfrac12 {\frac {\sigma}{P_s} | v -u_s|^2}
    +\tfrac12 {\frac {\rho_s}{P_s^2} | v -u_s|^2 \ \partial_s P_s }
\right)^2\ ,
\end{align*}
where $\partial_s P_s := \tfrac13 \left( \tau-2\,{\frac {s|\mu|^2}{s\sigma+1}}+{\frac {{s}^2|\mu|^2\sigma}{ ( s\sigma+1 )^2}} \right)$.
One can now verify that $\partial_s^2 F(s,x,v)$ is of the order $O(\sigma^2+|\mu|^2+\tau^2)$, which will be related to $O((f-f^\infty)^2)$ due to the estimates \eqref{locbound1}--\eqref{locbound2}. 

Simple but cumbersome calculations now show that  
if $\gamma>3/2$ and $\norm{f-f^\infty}_{\mathcal{H}_\gamma}$ is sufficiently small, 
then there exists a finite constant $\tilde C_\gamma$ depending only on $\gamma$ such that for all $s\in [0,1]$,
\begin{equation*} 
\left\Vert \partial_s^2 F(s,x,v) \right \Vert_{\mathcal{H}_\gamma} 
 \leq \tilde C_\gamma \| f-f^\infty\|_{\mathcal{H}_\gamma}^2\ ,
\end{equation*}
and hence
\begin{equation}\label{R-estimate}
  \|R_f\|_{\mathcal{H}_\gamma}\le \tilde C_\gamma \| f-f^\infty\|_{\mathcal{H}_\gamma}^2\  .
\end{equation}
[The calculations are simplest for non-negative integer $\gamma$, in which case the Sobolev norms can be calculated by differentiation. 
For $\gamma>3/2$ and sufficiently small $\norm{f-f^\infty}_{\mathcal{H}_\gamma}$,
the estimates \eqref{locbound3} ensure for all $s\in [0,1]$
the boundedness of $0<\epsilon<\norm{1+s \sigma}_\infty\ ,\, \norm{1+\frac13(s \tau-\frac{s^2|\mu|^2}{1+s\sigma})}_\infty<\infty$ (i.e. the denominators in \eqref{def-F}) for some fixed $\epsilon>0$. They also ensure the $L^2(\R^3;M_1^{-1}(v)\d[v])$-integrability of $F(s,x,\cdot)$ by using
\[ \exp\big\{-\frac{|\rho_s v-\mu_s|^2}{2P_s\rho_s}\big\} \leq  e^{-|v|^2/3+1} \quad \text{for all } x\in\tilde\T^3\ . \]
%
In \eqref{bgk}, higher powers of $\norm{f-f^\infty}_{\mathcal{H}_\gamma}$ (arising due to derivatives of $\sigma$, $\mu$ and $\tau$) can be absorbed into the constant of the quadratic term.] 
 
Now define the linearized BGK operator
\begin{equation*}
\Q_2 h(x,v,t)  :=  \left[ \left( \frac52 - \frac{|v|^2}{2}\right)\sigma(x) +v\cdot\mu(x)+ \left( -\frac{1}{2}+ \frac{|v|^2}{6}\right)\tau(x)\right] M_1(v)  - h(x,v,t)
\end{equation*}
where $\sigma$, $\mu$ and $\tau$ are determined by $h$, and hence $f$. 
For all $\gamma\geq 0$,  $\Q_2$ is self-adjoint on $\mathcal{H}_\gamma$. 
Then the nonlinear BGK equation \eqref{bgk} becomes
\begin{equation} \label{bgk:torus3} 
 \partial_t h(x,v,t) +  v\cdot \nabla_x h(x,v,t) =  \Q_2 h(x,v,t)  +  R_f(x,v,t)\ , \qquad t\geq 0\ , 
\end{equation}
which differs from the linearized BGK equation \eqref{linBGK:torus} only by the additional term $R_f$.

It is now a simple matter to prove local exponential stability. 
We shall use here exactly the entropy functional $\cE_\gamma(f)$ defined in \eqref{entropy:gamma:3D} with $f=M_1+h$.
Now assume that $h$ solves (\ref{bgk:torus3}). 
To compute $\frac{{\rm d}}{{\rm d}t} \cE_\gamma(f)$ we use an inequality like \eqref{e-inequality:3D} for the drift term and for $\Q_2 h$ in \eqref{bgk:torus3},
as well as $\|\P_{|k|}\|\le (1+2\alpha_*)$ and \eqref{R-estimate} for the term $R_f$.
This yields
\begin{equation} \label{e-decay1}
\frac{{\rm d}}{{\rm d}t} \cE_\gamma(f) \leq  -2\mu \,\cE_\gamma(f)  + 2(1+2\alpha_*) \tilde C_\gamma\| h\|_{\mathcal{H}_\gamma}^3\ ,
\end{equation}
(if $\|h\|_{\mathcal{H}_\gamma}$ is small enough)
where we have used the fact that $h = f-f^\infty$.
Due to~\eqref{equiv:norm:entropy:gamma}, 
it is now simple to complete the proof of Theorem~\ref{linBGK-decay}\ref{BGK-decay:b} for $L=2\pi$: 
In this case, the best decay rate $\mu = 0,0001774540949$ is attained for $\alpha_* = 0,1644256115$ (cf. Remark \ref{rem:mu:3D}(b)).
Estimate \eqref{e-decay1} shows that there is a $\delta_\gamma>0$ so that if the initial data $f^I(x,v)$ satisfies $\|f^I- f^\infty\|_{\mathcal{H}_\gamma} < \delta_\gamma$,
then the solution $f(t)$ satisfies
$$\cE_\gamma(f(t))  \leq e^{-t/2820}\cE_\gamma(f^I)\ .$$
Here we used that the linear decay rate in \eqref{e-decay1} is slightly better than $\frac{1}{2820}$, to compensate the nonlinear term.
\end{proof}

\appendix
\section{Appendix: Deferred proofs}\label{A1}
%

\begin{proof}[Proof of Lemma~\ref{lem:mu:1D}]
We  compute that $\CC_k^*\P_k + \P_k\CC_k$ is twice the identity matrix whose upper left $5\times 5$ block is replaced by 
\[ \DD_{k,\alpha,\beta,\gamma} = 
\begin{pmatrix}
2\ell\alpha & 0 & \ell(\sqrt2 \alpha -\beta) & 0 & 0 \\
0 & 2\ell(\sqrt{2}\beta-\alpha) & 0 & \ell (\sqrt{3}\beta - \sqrt{2}\gamma) & 0\\
\ell( \sqrt{2}\alpha -\beta) & 0 & 2\ell(\sqrt{3}\gamma-\sqrt2\beta) & -i\gamma/k & 2\ell\gamma\\
0 & \ell( \sqrt{3}\beta -\sqrt{2}\gamma) & i\gamma/k & 2-2\ell\sqrt{3}\gamma & 0\\
0 & 0 & 2\ell \gamma & 0 & 2
\end{pmatrix} \ ,
\]
where $\ell:= \tfrac{2\pi}{L}$.
We seek to choose $\alpha$, $\beta$ and $\gamma$ to make the matrices $\P_k$ and $\DD_{k,\alpha,\beta,\gamma}$ positive definite for all $k\in\Z\setminus\{0\}$.
Under the assumption $|\alpha|^2 +|\beta|^2 +|\gamma|^2 <1$,
 the matrix $\P_k$ will be positive definite for all $k\ne0$.
To simplify the analysis we shall now set $\beta =\sqrt{2} \alpha$ and $\gamma =\sqrt{3} \alpha$. On the one hand this will make the first three diagonal entries of $\DD_{k,\alpha,\beta,\gamma}$ equal and annihilate four off-diagonal elements. But, on the other hand, this will then lead to a reduced decay rate. But optimal decay rates are anyhow not our goal here -- due to considering only a simple ansatz for the transformation matrices $\P_k$. 
For $\beta =\sqrt{2} \alpha$ and $\gamma =\sqrt{3} \alpha$ we have
\begin{equation} \label{Dkaaa}
\DD_{k,\alpha,\beta,\gamma} = 
\begin{pmatrix}
2\ell\alpha & 0 & 0 & 0 & 0 \\
0 & 2\ell \alpha & 0 & 0 & 0\\
0 & 0 & 2\ell\alpha & -i\sqrt{3}\alpha/k & 2\sqrt{3}\ell\alpha\\
0 & 0 & i\sqrt{3}\alpha/k & 2-6\ell\alpha & 0\\
0 & 0 & 2\sqrt{3}\ell\alpha & 0 & 2
\end{pmatrix} \ .
\end{equation}
The positive definiteness of $\DD_{k,\alpha,\sqrt{2} \alpha,\sqrt{3} \alpha}$ will follow from Sylvester's criterion,
 by choosing $\alpha$ such that all minors of $\DD_{k,\alpha,\sqrt{2} \alpha,\sqrt{3} \alpha}$ will be positive. 
Let $\delta_j(k,\alpha)$ denote the determinant of the lower right $j\times j$ submatrix of $\DD_{k,\alpha,\sqrt{2} \alpha,\sqrt{3} \alpha}$ for $j=1,2,3,4,5$. 
For our choice $\beta =\sqrt{2} \alpha$ and $\gamma =\sqrt{3} \alpha$,
 the first minor $\delta_1(k,\alpha) = 2$ is always positive
 and the second minor $\delta_2(k,\alpha) = 4(1-3\ell\alpha)$ is positive for $\alpha<1/(3\ell)$.
The third minor satisfies
 \[
  \delta_3(k,\alpha)
    = \alpha \big( 72\ell^3 \alpha^2 -(48\ell^2 +\tfrac{6}{k^2})\alpha +8\ell \big)
    \ge \delta_3(1,\alpha)\quad \text{for all } k\ne0 \ ,
 \]
and the lower bound $\delta_3(1,\alpha)$ is positive if
 \begin{equation} \label{root:3:1D}
 0 < \alpha < \alpha^{(3)} := \frac{1+8\ell^2 -\sqrt{1+16\ell^2}}{24\ell^3}\,,\qquad
 \ell:=\frac{2\pi}{L} \ .
 \end{equation}
Moreover $0<\alpha^{(3)}<1/(3\ell)$ for all $\ell>0$ and $\max_{\ell>0}\alpha^{(3)}(\ell)<0.257$.
The fourth and fifth minor are multiples of the third minor:  
 \[ \delta_4(k,\alpha) = 2 \ell \alpha \delta_3(k,\alpha), \quad \delta_5(k,\alpha) = (2\ell \alpha)^2 \delta_3(k,\alpha)\,. \]
Hence, all minors are positive under assumption~\eqref{root:3:1D}.
 
Matrix $\DD_{k,\alpha,\sqrt{2} \alpha,\sqrt{3} \alpha}$ has a block diagonal structure.
Thus it has a double eigenvalue $2\ell\alpha$ and the eigenvalues of its lower right $3\times 3$-submatrix
\begin{equation} \label{Dka3}
\DD_{k,\alpha,\sqrt{2} \alpha,\sqrt{3} \alpha}^{(3)} = 
\begin{pmatrix}
 2\ell\alpha & -i\sqrt{3}\alpha/k & 2\sqrt{3}\ell\alpha\\
 i\sqrt{3}\alpha/k & 2-6\ell\alpha & 0\\
 2\sqrt{3}\ell\alpha & 0 & 2
\end{pmatrix} \ .
\end{equation}
Let $\{\lambda_1,\lambda_2,\lambda_3\}$ be the eigenvalues of $\DD_{k,\alpha,\sqrt{2} \alpha,\sqrt{3} \alpha}^{(3)}$ arranged in increasing order.   
We seek a lower bound on~$\lambda_1$. 
As long as $\DD_{k,\alpha,\sqrt{2} \alpha,\sqrt{3} \alpha}^{(3)}$ is positive definite, 
 the arithmetic-geometric mean inequality implies 
 \begin{align*}
 \lambda_1(k,\alpha) = \frac{\delta_3(k,\alpha)}{\lambda_2 \lambda_3}
   &\geq \delta_3(k,\alpha) \left( \frac{ \lambda_2 + \lambda_3}{2}\right)^{-2}\\
   &\geq \delta_3(k,\alpha) \left( \frac{\tr[\DD_{k,\alpha,\sqrt{2} \alpha,\sqrt{3} \alpha}^{(3)}]}{2}\right)^{-2} 
    = \frac{\delta_3(k,\alpha)}{4(1-\ell\alpha)^2} \ ,
 \end{align*}
 since $\tr[\DD_{k,\alpha,\sqrt{2} \alpha,\sqrt{3} \alpha}^{(3)}] =4(1-\ell\alpha)$.
Thus, if $\P_k$ is chosen with some $\alpha \in (0,\alpha^{(3)})$, $\beta =\sqrt{2} \alpha$, and $\gamma =\sqrt{3} \alpha$ uniformly for all $|k|\in\N$, then
 \begin{equation}\label{Pbnd}
 \CC_k^*\P_k + \P_k\CC_k
  \geq \frac{\delta_3(1,\alpha)}{4(1-\ell\alpha)^2} \ {\bf I} \qquad \text{ uniformly in } |k|\in\N \ ,
  \end{equation}
 since $\min\big\{2\ell\alpha, \frac{\delta_3(1,\alpha)}{4(1-\ell\alpha)^2} \big\} =\frac{\delta_3(1,\alpha)}{4(1-\ell\alpha)^2}$ for all $\alpha \in (0,\alpha^{(3)})$.
A simple computation shows that the eigenvalues of $\P_k$ are $1$, $1 \pm\alpha \sqrt{3+\sqrt 6}/k$, and $1 \pm\alpha \sqrt{3-\sqrt 6}/k$. 
These eigenvalues are positive for all $0\leq\alpha\leq\max_{\ell>0}\alpha^{(3)}(\ell)$, $L>0$ and $k\in\N$.
Hence, uniformly in $|k|$,
 \begin{equation}\label{Pbound}
 \big(1 -\alpha \sqrt{3+\sqrt 6}\big) {\bf I}  \leq \P_k \leq \big(1 +\alpha \sqrt{3+\sqrt 6}\big) {\bf I}\ .
 \end{equation}
Combining \eqref{Pbound} with \eqref{Pbnd} yields the result \eqref{goodbnd}.
\end{proof}
\bigskip

The following lemma will be needed in the proofs of Lemma~\ref{lemma:decay:2D} and Lemma~\ref{lemma:decay:3D}. 
\begin{lemma}\label{lem:p}
 Let $p(\kappa,\alpha)$ be a rational function of the form
 \begin{equation*} 
  p(\kappa,\alpha) = \big(p_0(\alpha) +p_1(\alpha) \tfrac1{\kappa^2}\big) \tfrac1{\kappa^2} + p_2(\alpha) \,,
 \end{equation*}
 where $p_0$, $p_1$, and $p_2$ are polynomials in $\alpha$.
 If there exists $\widetilde\alpha>0$ such that
 \[ 
    0\le p_1(\alpha) \quad \text{and } \quad 
    p_0(\alpha) +2 p_1(\alpha) \le 0 \qquad
    \forall \alpha\in[0,\widetilde\alpha]\,,
 \]
 then $p(1,\alpha) \leq p(\kappa,\alpha)$ for all $\alpha\in[0,\widetilde\alpha]$ and $1\leq \kappa$.
\end{lemma}
\begin{proof}
 We want to prove $p(1,\alpha) \leq p(\kappa,\alpha)$ for all $\alpha\in[0,\widetilde\alpha]$ and $1\leq \kappa$, or equivalently,
 \[ p_0(\alpha) +p_1(\alpha) \leq \big(p_0(\alpha) +p_1(\alpha) \tfrac1{\kappa^2}\big) \tfrac1{\kappa^2} \qquad \forall \alpha\in[0,\widetilde\alpha] \quad \forall 1\leq \kappa \,. \] 
 We multiply the inequality with $\kappa^2$
 \[ \big(p_0(\alpha) +p_1(\alpha)\big)\kappa^2 \leq p_0(\alpha) +p_1(\alpha) \tfrac1{\kappa^2} \]
 and rearrange the summands
 \[ p_0(\alpha)\ (\kappa^2-1) \leq p_1(\alpha)\ \big(\tfrac1{\kappa^2} -\kappa^2) = -p_1(\alpha) \tfrac{(\kappa^2 -1)(\kappa^2 +1)}{\kappa^2} \,. \]
 For $\kappa=1$ the inequality holds trivially.
 Therefore, we continue with $\kappa>1$ and divide the inequality by $(\kappa^2-1)$ to obtain 
 \[ p_0(\alpha) \leq -p_1(\alpha) \tfrac{\kappa^2 +1}{\kappa^2}\ . \]
 Under our assumptions this inequality holds since 
 \[ p_0(\alpha) \leq -2 p_1(\alpha) \leq -p_1(\alpha) \tfrac{\kappa^2 +1}{\kappa^2} \leq 0 \,. \]
 This finishes the proof.
\end{proof}
\

\setcounter{MaxMatrixCols}{21}
\begin{proof}[Proof of Lemma~\ref{lemma:decay:2D}]
We  compute that $\CC_\kappa^*\P_\kappa + \P_\kappa\CC_\kappa$ is twice the identity matrix whose upper left $11\times 11$ block is replaced by $\DD_{\kappa,\alpha,\beta,\gamma,\omega}$ given as
{\footnotesize
\[ 
\begin{pmatrix}
2\ell\alpha & 0 & 0 & \ell\alpha & 0 & \ell(\alpha-\beta) & 0 & 0 & 0 & 0 & 0 \\
0 & -2\ell(\alpha-\beta) & 0 & 0 & 0 & -i \beta/\kappa & \ell\tfrac{\sqrt{3}\beta-\sqrt{2}\omega}{\sqrt{2}} & 0 & -\ell\beta/\sqrt{2} & 0 & 0 \\
0 & 0 & 2\ell\gamma & 0 & -i \gamma/\kappa & 0 & 0 & \sqrt{2}\ell\gamma & 0 & 0 & 0 \\
\ell\alpha & 0 & 0 & \ell\sqrt{6}\omega & 0 & \ell(\sqrt{\tfrac32}\omega-\beta) & -i \omega/\kappa & 0 & 0 & 0 & 2\ell\omega \\
0 & 0 & i \gamma/\kappa & 0 & 2(1-\ell\gamma) & 0 & 0 & 0 & 0 & 0 & 0 \\
\ell(\alpha-\beta) & i \beta/\kappa & 0 & \ell(\sqrt{\tfrac32}\omega-\beta) & 0 & 2(1-\ell\beta) & 0 & 0 & 0 & 0 & 0 \\
0 & \ell\tfrac{\sqrt{3}\beta-\sqrt{2}\omega}{\sqrt{2}} & 0 & i \omega/\kappa & 0 & 0 & 2-\sqrt{6}\ell\omega & 0 & -\ell\omega/\sqrt{2} & 0 & 0 \\
0 & 0 & \sqrt{2}\ell\gamma & 0 & 0 & 0 & 0 & 2 & 0 & 0 & 0 \\
0 & -\ell\beta/\sqrt{2} & 0 & 0 & 0 & 0 & -\ell\omega/\sqrt{2} & 0 & 2 & 0 & 0 \\
0 & 0 & 0 & 0 & 0 & 0 & 0 & 0 & 0 & 2 & 0 \\
0 & 0 & 0 & 2\ell\omega & 0 & 0 & 0 & 0 & 0 & 0 & 2
\end{pmatrix} \ ,
\]}

\noindent
with $\ell := 2\pi/L>0$.
We seek to choose $\alpha$, $\beta$, $\gamma$ and $\omega$ such that the matrices $\P_\kappa$ and $\DD_{\kappa,\alpha,\beta,\gamma,\omega}$ are positive definite for all $\kappa\in K$.
The positive definiteness of $\DD_{\kappa,\alpha,\beta,\gamma,\omega}$ will follow from Sylvester's criterion, if all minors of $\DD_{\kappa,\alpha,\beta,\gamma,\omega}$ are positive.
This will yield restrictions on the choice of parameters $\alpha$, $\beta$, $\gamma$ and $\omega$.
The analysis will simplify, if we choose $\beta$, $\gamma$ and $\omega$ as multiples of $\alpha$,  because then the first four columns will depend linearly on $\alpha$ and, moreover, several terms will drop out.
For $\beta =2\alpha$, $\gamma =\alpha$ and $\omega=\sqrt{6}\alpha$, we compute $\DD_{\kappa,\alpha,2\alpha,\alpha,\sqrt{6}\alpha}$ as
{\footnotesize
\begin{equation} \label{Dka2aa3a}
\begin{pmatrix}
2\ell\alpha & 0 & 0 & \ell\alpha & 0 & -\ell\alpha & 0 & 0 & 0 & 0 & 0 \\
0 & 2\ell\alpha & 0 & 0 & 0 & -i 2\alpha/\kappa & 0 & 0 & -\ell\sqrt{2}\alpha & 0 & 0 \\
0 & 0 & 2\ell\alpha & 0 & -i \alpha/\kappa & 0 & 0 & \sqrt{2}\ell\alpha & 0 & 0 & 0 \\
\ell\alpha & 0 & 0 & \ell 6\alpha & 0 & \ell\alpha & -i \sqrt{6}\alpha/\kappa & 0 & 0 & 0 & 2\ell\sqrt{6}\alpha \\
0 & 0 & i \alpha/\kappa & 0 & 2(1-\ell\alpha) & 0 & 0 & 0 & 0 & 0 & 0 \\
-\ell\alpha & i 2\alpha/\kappa & 0 & \ell\alpha & 0 & 2(1-\ell 2\alpha) & 0 & 0 & 0 & 0 & 0 \\
0 & 0 & 0 & i \sqrt{6}\alpha/\kappa & 0 & 0 & 2-\ell 6\alpha & 0 & -\ell\sqrt{3}\alpha & 0 & 0 \\
0 & 0 & \ell\sqrt{2}\alpha & 0 & 0 & 0 & 0 & 2 & 0 & 0 & 0 \\
0 & -\ell\sqrt{2}\alpha & 0 & 0 & 0 & 0 & -\ell\sqrt{3}\alpha & 0 & 2 & 0 & 0 \\
0 & 0 & 0 & 0 & 0 & 0 & 0 & 0 & 0 & 2 & 0 \\
0 & 0 & 0 & 2\ell\sqrt{6}\alpha & 0 & 0 & 0 & 0 & 0 & 0 & 2
\end{pmatrix} \ .
\end{equation} }
Let $\delta_j(\kappa,\alpha,\beta,\gamma,\omega)$ denote the determinant of the upper left $j\times j$ submatrix of $\DD_{\kappa,\alpha,\beta,\gamma,\omega}$ for integers $j=1,2,\ldots,11$. 
For our choice $\beta =2\alpha$, $\gamma =\alpha$ and $\omega=\sqrt{6}\alpha$, the minors $\delta_j(\kappa,\alpha):=\delta_j(\kappa,\alpha,2\alpha,\alpha,\sqrt{6}\alpha)$ are given in Table~\ref{table:minors:2D}.
The first four minors are positive, if $\alpha$ is positive.

\renewcommand{\arraystretch}{1.5}
\begin{table}
\begin{tabular}{|l|l|}
 \hline 
 $\delta_1(\kappa,\alpha)$ &= $2\ell\alpha$ \\
 \hline 
 $\delta_2(\kappa,\alpha)$ &= $4\ell^2\alpha^2$ \\
 \hline 
 $\delta_3(\kappa,\alpha)$ &= $8\ell^3\alpha^3$ \\
 \hline 
 $\delta_4(\kappa,\alpha)$ &= $44\ell^4 \alpha^4$ \\
 \hline 
 $\delta_5(\kappa,\alpha)$ &= $22\ell^3 \alpha^4 (4\ell -4\ell^2 \alpha -{\alpha}/{\kappa^2})$ \\
 \hline 
 $\delta_6(\kappa,\alpha)$ &= $\delta_5(\kappa,\alpha) p_6(\kappa,\alpha)/\ell$ \\
  & with $p_6(\kappa,\alpha) :=-\tfrac{54}{11} \ell^2 \alpha +2\ell -2\alpha /{\kappa^2}$. \\
 \hline 
 $\delta_7(\kappa,\alpha)$ &= $\frac{2}{11\ell^2} \delta_5(\kappa,\alpha) p_7(\kappa,\alpha)$ \\
  & with $p_{7}(\kappa,\alpha) = \big(p_{7,0}(\alpha) +p_{7,1}(\alpha) \frac1{\kappa^2}\big) \frac1{\kappa^2} + p_{7,2}(\alpha)$, \\
  & $p_{7,0}(\alpha) = 93\ell^2 \alpha^2 -34\ell\alpha$, \quad
    $p_{7,1}(\alpha) = 12\alpha^2$, \quad
    $p_{7,2}(\alpha) = 162\ell^4 \alpha^2 -120\ell^3 \alpha +22\ell^2$. \\
 \hline 
 $\delta_8(\kappa,\alpha)$ &= $44 \ell^3 \alpha^4 \frac{\delta_7(\kappa,\alpha)}{\delta_5(\kappa,\alpha)} p_8(\kappa,\alpha)$ \\
  & with $p_8(\kappa,\alpha) = 2\ell^3 \alpha^2 -6\ell^2\alpha +4\ell -{\alpha}/{\kappa^2}$. \\
 \hline 
 $\delta_9(\kappa,\alpha)$ &= $8 \ell \alpha^4 p_8(\kappa,\alpha) p_9(\kappa,\alpha)$ \\
  & with $p_{9}(\kappa,\alpha) = \big(p_{9,0}(\alpha) +p_{9,1}(\alpha) \frac1{\kappa^2}\big)  \frac1{\kappa^2} + p_{9,2}(\alpha)$, \\
  & $p_{9,0}(\alpha) = -12\ell^3 \alpha^3 +198\ell^2 \alpha^2 -68\ell \alpha$, \quad
    $p_{9,1}(\alpha) = 24\alpha^2$, \\  
  & $p_{9,2}(\alpha) = -81\ell^5 \alpha^3 +411\ell^4 \alpha^2 -262\ell^3 \alpha +44\ell^2$. \\
 \hline 
 $\delta_{10}(\kappa,\alpha)$ &= $2 \delta_9(\kappa,\alpha)$, \\
 \hline 
 $\delta_{11}(\kappa,\alpha)$ &= $64 \ell \alpha^4 p_8(\kappa,\alpha) p_{11}(\kappa,\alpha)$ \\
  & with $p_{11}(\kappa,\alpha) = \big(p_{11,0}(\alpha) +p_{11,1}(\alpha) \frac1{\kappa^2}\big) \frac1{\kappa^2} + p_{11,2}(\alpha)$, \\
  & $p_{11,0}(\alpha) = -72\ell^4 \alpha^4 -300\ell^3 \alpha^3 +294\ell^2 \alpha^2 -68\ell\alpha$, \qquad
    $p_{11,1}(\alpha) = 24\alpha^2$, \\
  & $p_{11,2}(\alpha) = 162\ell^6 \alpha^4 -909\ell^5 \alpha^3 +963\ell^4 \alpha^2 -358\ell^3 \alpha +44\ell^2$. \\
 \hline 
\end{tabular}
\caption{Let $\delta_j(\kappa,\alpha,\beta,\gamma,\omega)$ denote the determinant of the upper left $j\times j$ submatrix of $\DD_{\kappa,\alpha,\beta,\gamma,\omega}$ for integers $j=1,2,\ldots,11$. 
For our choice $\beta =2\alpha$, $\gamma =\alpha$ and $\omega=\sqrt{6}\alpha$, the minors $\delta_j(\kappa,\alpha)=\delta_j(\kappa,\alpha,2\alpha,\alpha,\sqrt{6}\alpha)$ are given in this table.}
\label{table:minors:2D}
\end{table}

The fifth minor $\delta_5(\kappa,\alpha)$ satisfies for positive $\alpha$ the inequality $\delta_5(\kappa,\alpha)\geq \delta_5(1,\alpha)$ for all $\kappa\in K$. 
Moreover, $\delta_5(1,\alpha)$ is positive for $\alpha\in(0,\alpha_{\delta_5})$ with $\alpha_{\delta_5} :={4\ell}/{(4\ell^2 +1)}$.
Thus the fifth minor $\delta_5(\kappa,\alpha)$ is positive for all $\kappa\in K$ if $\alpha\in(0,\alpha_{\delta_5})$.

The sixth minor $\delta_6(\kappa,\alpha)$ has a factorization as $\delta_6(\kappa,\alpha) = \delta_5(\kappa,\alpha) p_6(\kappa,\alpha)/\ell$.
The factor $p_6(\kappa,\alpha)$
satisfies for positive $\alpha$ the inequality $p_6(\kappa,\alpha)\geq p_6(1,\alpha)$ for all $\kappa\in K$.
Moreover, $p_6(1,\alpha)$ is positive for $\alpha\in(0,\alpha_{p_6})$ with $\alpha_{p_6} :={22\ell}/{(54\ell^2 +22)}$.
Thus the sixth minor $\delta_6(\kappa,\alpha)$ is positive if $0<\alpha<\alpha_{\delta_6}$ with $\alpha_{\delta_6} :=\min\{\alpha_{\delta_5} , \ \alpha_{p_6} \} =\alpha_{p_6}$.

The seventh minor $\delta_7(\kappa,\alpha)$ has a factorization as $\delta_7(\kappa,\alpha) = 2 \delta_5(\kappa,\alpha) p_7(\kappa,\alpha)/{11\ell^2}$.
Due to Lemma \ref{lem:p}, the inequality $p_{7}(\kappa,\alpha) \geq p_{7}(1,\alpha)$ holds  for some positive $\widetilde\alpha_{p_7}$,
 and consequently $\delta_{7}(\kappa,\alpha)\geq \delta_{7}(1,\alpha)$
 holds for all $0\leq \alpha \leq \min\{\widetilde\alpha_{p_7},\alpha_{\delta_5}\}$ and $\kappa\in K$. 
The quadratic polynomial $p_7(1,\alpha)$ has two positive roots~$0<\alpha_{p_7,-}<\alpha_{p_7,+}$ 
 and is positive for all $0<\alpha<\alpha_{p_7}$ with $\alpha_{p_7} :=\alpha_{p_7,-}$.
Consequently, for $0<\alpha <\alpha_{\delta_7}$ with $\alpha_{\delta_7} := \min\{\alpha_{\delta_5}\,,\ \widetilde\alpha_{p_7}\,, \ \alpha_{p_7}\}$
 the seventh minor $\delta_{7}(\kappa,\alpha)$ is positive for all $\kappa\in K$.

The eighth minor $\delta_8(\kappa,\alpha)$ has a factorization. 
For positive $\alpha$, factor~$p_8$ satisfies the inequality $p_8(\kappa,\alpha) >p_8(1,\alpha)$ for all $\kappa\in K$. 
The quadratic polynomial $p_8(1,\alpha)$ has two positive roots~$0<\alpha_{p_8,-}<\alpha_{p_8,+}$ 
 and is positive for all $0<\alpha<\alpha_{p_8}$ with $\alpha_{p_8} :=\alpha_{p_8,-}$.
Thus, the eighth minor $\delta_8(\kappa,\alpha)$ is positive for all $\kappa\in K$,
 if $0<\alpha< \alpha_{\delta_8}$ with $\alpha_{\delta_8} :=\min\{\alpha_{\delta_5}\,,\ \alpha_{\delta_7}\,,\ \alpha_{p_8} \}$.

The ninth minor $\delta_9(\kappa,\alpha)$ has a factorization as $\delta_9(\kappa,\alpha) = 8 \ell \alpha^4 p_8(\kappa,\alpha) p_9(\kappa,\alpha)$.
Due to Lemma~\ref{lem:p}, the inequality $p_{9}(\kappa,\alpha) \geq p_{9}(1,\alpha)$ holds for some positive $\widetilde\alpha_{p_9}$,
 and consequently $\delta_{9}(\kappa,\alpha)\geq \delta_{9}(1,\alpha)$
 holds for all $0\leq \alpha \leq \min\{\widetilde\alpha_{p_9},\alpha_{p_8}\}$ and $\kappa\in K$. 
The cubic polynomial $p_{9}(1,\alpha)$ is positive at $\alpha=0$ and $\lim_{\alpha\to\infty} p_9(1,\alpha)=-\infty$.
Hence, there exists a positive root $\alpha_{p_9}$ such that $p_9(1,\alpha)$ is positive for all $0<\alpha<\alpha_{p_9}$.
Consequently, for all $\alpha\in(0,\alpha_{\delta_9})$ with $\alpha_{\delta_9} :=\min\{\alpha_{p_9}\,,\ \widetilde\alpha_{p_9}\,,\ \alpha_{p_8}\}$,
the ninth minor $\delta_{9}(\kappa,\alpha)$ is positive for all $\kappa\in K$.

The tenth minor $\delta_{10}$ satisfies $\delta_{10}(\kappa,\alpha) = 2 \delta_9(\kappa,\alpha)$. 
Therefore the tenth minor $\delta_{10}(\kappa,\alpha)$ is positive for all $\kappa\in K$ if $\alpha\in(0,\alpha_{\delta_9})$. 

The eleventh minor $\delta_{11}(\kappa,\alpha)$ has a factorization as $\delta_{11}(\kappa,\alpha) = 64 \ell \alpha^4 p_8(\kappa,\alpha) p_{11}(\kappa,\alpha)$.
Due to Lemma \ref{lem:p}, the inequality $p_{11}(\kappa,\alpha) \geq p_{11}(1,\alpha)$ holds for some positive $\widetilde\alpha_{p_{11}}$,
 and consequently $\delta_{11}(\kappa,\alpha)\geq \delta_{11}(1,\alpha)$
 holds for all $0\leq \alpha \leq \min\{\widetilde\alpha_{p_{11}},\alpha_{p_8}\}$ and $\kappa\in K$. 
The quartic polynomial $p_{11}(1,\alpha)$ is positive at $\alpha=0$.
Hence, there exists a positive root $\alpha_{p_{11}}$ such that $p_{11}(1,\alpha)$ is positive for all $0<\alpha<\alpha_{p_{11}}$.
Consequently, for $\alpha\in(0,\alpha_{\delta_{11}})$ with $\alpha_{\delta_{11}} := :=\min\{\alpha_{p_{11}}\,,\ \widetilde\alpha_{p_{11}}\,,\ \alpha_{p_8}\}$,
the eleventh minor $\delta_{11}(\kappa,\alpha)$ is positive for all $\kappa\in K$.\\

Let $\{\lambda_1,\lambda_2,\ldots,\lambda_{11}\}$ be the eigenvalues of $\DD_{\kappa,\alpha,\beta,\gamma,\omega}$ arranged in increasing order.   
We seek a lower bound on $\lambda_1$. 
As long as $\DD_{\kappa,\alpha,\beta,\gamma,\omega}$ is positive definite, 
 the arithmetic-geometric mean inequality implies 
 \begin{align*}
 \lambda_1(\kappa,\alpha,\beta,\gamma,\omega) = \frac{\delta_{11}(\kappa,\alpha,\beta,\gamma,\omega)}{\prod_{j=2}^{11} \lambda_j}
   &\geq \delta_{11}(\kappa,\alpha,\beta,\gamma,\omega) \left( \frac{10}{\sum_{j=2}^{11} \lambda_j} \right)^{10}\\
   &\geq \delta_{11}(\kappa,\alpha,\beta,\gamma,\omega) \left( \frac{10}{\tr[\DD_{\kappa,\alpha,\beta,\gamma,\omega}]} \right)^{10} 
    \geq \left( \frac{10}{14} \right)^{10} \delta_{11}(\kappa,\alpha,\beta,\gamma,\omega) \ ,
 \end{align*}
 since $\tr[\DD_{\kappa,\alpha,\beta,\gamma,\omega}] =14$ independently of $\kappa$, $\alpha$, $\beta$, $\gamma$ and $\omega$.
 
A simple computation shows that the eigenvalues of $\P_\kappa$ are $1$, $1 \pm \alpha/\kappa$, $1 \pm \sqrt{5}\alpha/\kappa$, and $1\pm\sqrt{6}\alpha/\kappa$. 
Hence, uniformly in $\kappa\in K$,
 \begin{equation}\label{Pbound2}
 \big(1-\sqrt{6}|\alpha| \big) {\bf I}  \leq \P_\kappa \leq \big(1+\sqrt{6}|\alpha| \big) {\bf I}\ .
 \end{equation}
Thus, all matrices $\P_\kappa$ are positive definite, if $|\alpha|<\frac1{\sqrt{6}}$.
 
Finally, if $\P_\kappa$ is chosen with
 \begin{equation*} 
   \alpha \in (0,\alpha_+)\,, \qquad \text{where }
   \alpha_+ := \min\{ 1/\sqrt{6}\,,\ \alpha_{\delta_5}\,,\ \alpha_{\delta_6}\,,\ \alpha_{\delta_7}\,,\ \alpha_{\delta_8}\,,\ \alpha_{\delta_9}\,,\ 
   \alpha_{\delta_{11}} \}\,,
 \end{equation*}
 $\beta =2\alpha$, $\gamma =\alpha$, and $\omega=\sqrt{6}\alpha$ uniformly for all $\kappa\in K$, then
 \begin{equation}\label{Pbnd2}
 \CC_\kappa^*\P_\kappa + \P_\kappa\CC_\kappa \geq \left( \frac{10}{14} \right)^{10} \delta_{11}(1,\alpha,2\alpha,\alpha,\sqrt{6}\alpha) \ {\bf I} \qquad \text{ uniformly in } \kappa\in K \ . 
 \end{equation}

Combining \eqref{Pbound2} with \eqref{Pbnd2} yields the result.
\end{proof}
\

\setcounter{MaxMatrixCols}{21}
\begin{proof}[Proof of Lemma~\ref{lemma:decay:3D}]
We  compute that $\CC_\kappa^*\P_\kappa + \P_\kappa\CC_\kappa$ is twice the identity matrix whose upper left $21\times 21$ block is replaced by $\DD_{\kappa,\alpha,\beta,\gamma,\omega,\eta}$ given as
{\tiny
\[
 \begin{pmatrix} 
  2\ell\alpha & 0 & 0 & 0 & \tfrac{\sqrt2}{\sqrt3} \ell\alpha & 0 & 0 & A & 0 & \tfrac{\sqrt2}{\sqrt3} \ell\alpha & 0 & 0 & 0 & 0 & 0 & 0 & 0 & 0 & 0 & 0 & 0 \\
  0 & B & 0 & 0 & 0 & 0 & 0 & - i\beta/k & 0 & 0 & -C & 0 & 0 & \tfrac{3+\sqrt3}{6} \ell\beta & 0 & \tfrac{3-\sqrt3}{6} \ell\beta & 0 & 0 & 0 & 0 & 0 \\
  0 & 0 & 2\ell \gamma & 0 & 0 & -i \gamma/k & 0 & 0 & 0 & 0 & 0 & \sqrt2 \ell \gamma & 0 & 0 & 0 & 0 & 0 & 0 & 0 & 0 & 0 \\
  0 & 0 & 0 & 2\ell\omega & 0 & 0 & -i \omega/k & 0 & 0 & 0 & 0 & 0 & \sqrt2 \ell \omega & 0 & 0 & 0 & 0 & 0 & 0 & 0 & 0 \\
  \tfrac{\sqrt2}{\sqrt3} \ell\alpha & 0 & 0 & 0 & 2\ell\eta & 0 & 0 & D & 0 & \ell\eta & -i \eta/k & 0 & 0 & 0 & 0 & 0 & 0 & 0 & 0 & 0 & 2\ell\eta \\
  0 & 0 & i \gamma/k & 0 & 0 & 2-2\ell\gamma & 0 & 0 & 0 & 0 & 0 & 0 & 0 & 0 & 0 & 0 & 0 & 0 & 0 & 0 & 0 \\
  0 & 0 & 0 & i \omega/k & 0 & 0 & 2-2\ell\omega & 0 & 0 & 0 & 0 & 0 & 0 & 0 & 0 & 0 & 0 & 0 & 0 & 0 & 0 \\
  A & i \beta/k & 0 & 0 & D & 0 & 0 & 2-2\tfrac{\sqrt2}{\sqrt3} \ell\beta & 0 & -\tfrac{\sqrt2}{\sqrt3} \ell\beta & 0 & 0 & 0 & 0 & 0 & 0 & 0 & 0 & 0 & 0 & 0 \\
  0 & 0 & 0 & 0 & 0 & 0 & 0 & 0 & 2 & 0 & 0 & 0 & 0 & 0 & 0 & 0 & 0 & 0 & 0 & 0 & 0 \\
  \tfrac{\sqrt2}{\sqrt3} \ell\alpha & 0 & 0 & 0 & \ell\eta & 0 & 0 & -\tfrac{\sqrt2}{\sqrt3} \ell\beta & 0 & 2 & 0 & 0 & 0 & 0 & 0 & 0 & 0 & 0 & 0 & 0 & 0 \\
  0 & -C & 0 & 0 & i \eta/k & 0 & 0 & 0 & 0 & 0 & 2-2\ell\eta & 0 & 0 & -\tfrac{1}{\sqrt{3}} \ell\eta & 0 & -\tfrac{1}{\sqrt{3}} \ell\eta & 0 & 0 & 0 & 0 & 0 \\
  0 & 0 & \sqrt2 \ell\gamma & 0 & 0 & 0 & 0 & 0 & 0 & 0 & 0 & 2 & 0 & 0 & 0 & 0 & 0 & 0 & 0 & 0 & 0 \\
  0 & 0 & 0 & \sqrt2 \ell\omega & 0 & 0 & 0 & 0 & 0 & 0 & 0 & 0 & 2 & 0 & 0 & 0 & 0 & 0 & 0 & 0 & 0 \\
  0 & \tfrac{3+\sqrt3}{6} \ell\beta & 0 & 0 & 0 & 0 & 0 & 0 & 0 & 0 & -\tfrac{1}{\sqrt{3}} \ell\eta & 0 & 0 & 2 & 0 & 0 & 0 & 0 & 0 & 0 & 0 \\
  0 & 0 & 0 & 0 & 0 & 0 & 0 & 0 & 0 & 0 & 0 & 0 & 0 & 0 & 2 & 0 & 0 & 0 & 0 & 0 & 0 \\
  0 & \tfrac{3-\sqrt3}{6} \ell\beta & 0 & 0 & 0 & 0 & 0 & 0 & 0 & 0 & -\tfrac{1}{\sqrt{3}} \ell\eta & 0 & 0 & 0 & 0 & 2 & 0 & 0 & 0 & 0 & 0 \\
  0 & 0 & 0 & 0 & 0 & 0 & 0 & 0 & 0 & 0 & 0 & 0 & 0 & 0 & 0 & 0 & 2 & 0 & 0 & 0 & 0 \\
  0 & 0 & 0 & 0 & 0 & 0 & 0 & 0 & 0 & 0 & 0 & 0 & 0 & 0 & 0 & 0 & 0 & 2 & 0 & 0 & 0 \\
  0 & 0 & 0 & 0 & 0 & 0 & 0 & 0 & 0 & 0 & 0 & 0 & 0 & 0 & 0 & 0 & 0 & 0 & 2 & 0 & 0 \\
  0 & 0 & 0 & 0 & 0 & 0 & 0 & 0 & 0 & 0 & 0 & 0 & 0 & 0 & 0 & 0 & 0 & 0 & 0 & 2 & 0 \\
  0 & 0 & 0 & 0 & 2\ell\eta & 0 & 0 & 0 & 0 & 0 & 0 & 0 & 0 & 0 & 0 & 0 & 0 & 0 & 0 & 0 & 2 \\
 \end{pmatrix} \ . 
\]
}
with $\ell:=2\pi/L>0$ and 
\[ A := \tfrac{\ell}{\sqrt3} (\sqrt2 \alpha -\sqrt3 \beta)\,,\quad B := \tfrac2{\sqrt3} \ell (\sqrt2 \beta -\sqrt3 \alpha)\,,\quad C := 
\tfrac{\ell}{\sqrt3} (\sqrt2 \eta -\sqrt3 \beta)\,,\quad D := \tfrac{\ell}{\sqrt3} (\sqrt3 \eta -\sqrt2 \beta)\,.
\]
We seek to choose $\alpha$, $\beta$, $\gamma$, $\omega$ and $\eta$ such that the matrices $\P_\kappa$ and $\DD_{\kappa,\alpha,\beta,\gamma,\omega,\eta}$ are positive definite for all $\kappa\in K$.
The positive definiteness of $\DD_{\kappa,\alpha,\beta,\gamma,\omega,\eta}$ will follow from Sylvester's criterion, if all minors of $\DD_{\kappa,\alpha,\beta,\gamma,\omega,\eta}$ are positive.
This will yield restrictions on the choice of parameters $\alpha$, $\beta$, $\gamma$, $\omega$ and $\eta$.
The analysis will simplify, if we choose $\beta$, $\gamma$, $\omega$ and $\eta$ as multiples of $\alpha$,
 because then the first six columns will depend linearly on $\alpha$.
For $\beta =\sqrt3 \alpha$, $\gamma =\alpha$, $\omega=\alpha$ and $\eta=\alpha$, we compute $\DD_{\kappa,\alpha,\sqrt3 \alpha,\alpha,\alpha,\alpha}$ as

\begin{align}
{\tiny \begin{pmatrix} 
  2\ell\alpha & 0 & 0 & 0 & \tfrac{\sqrt2}{\sqrt3} \ell\alpha & 0 & 0 & \tfrac{\sqrt2 -3}{\sqrt3} \ell \alpha & 0 & \tfrac{\sqrt2}{\sqrt3} \ell\alpha & 0 & 0 & 0 & 0 & 0 & 0 & 0 & 0 & 0 & 0 & 0 \\
  0 & 2 (\sqrt2 -1) \ell\alpha & 0 & 0 & 0 & 0 & 0 & - i\sqrt3 \alpha/\kappa & 0 & 0 & \tfrac{3-\sqrt2}{\sqrt3} \ell\alpha & 0 & 0 & \tfrac{\sqrt3+1}{2} \ell\alpha & 0 & \tfrac{\sqrt3-1}{2} \ell\alpha & 0 & 0 & 0 & 0 & 0 \\
  0 & 0 & 2\ell\alpha & 0 & 0 & -i \alpha/\kappa & 0 & 0 & 0 & 0 & 0 & \sqrt2 \ell \alpha & 0 & 0 & 0 & 0 & 0 & 0 & 0 & 0 & 0 \\
  0 & 0 & 0 & 2\ell\alpha & 0 & 0 & -i \alpha/\kappa & 0 & 0 & 0 & 0 & 0 & \sqrt2 \ell \alpha & 0 & 0 & 0 & 0 & 0 & 0 & 0 & 0 \\
  \tfrac{\sqrt2}{\sqrt3} \ell\alpha & 0 & 0 & 0 & 2\ell\alpha & 0 & 0 & (1-\sqrt2) \ell\alpha & 0 & \ell\alpha & -i \alpha/\kappa & 0 & 0 & 0 & 0 & 0 & 0 & 0 & 0 & 0 & 2\ell\alpha \\
  0 & 0 & i \alpha/\kappa & 0 & 0 & 2-2\ell\alpha & 0 & 0 & 0 & 0 & 0 & 0 & 0 & 0 & 0 & 0 & 0 & 0 & 0 & 0 & 0 \\
  0 & 0 & 0 & i \alpha/\kappa & 0 & 0 & 2-2\ell\alpha & 0 & 0 & 0 & 0 & 0 & 0 & 0 & 0 & 0 & 0 & 0 & 0 & 0 & 0 \\
  \tfrac{\sqrt2 -3}{\sqrt3} \ell\alpha & i \sqrt3 \alpha/\kappa & 0 & 0 & (1 -\sqrt2) \ell\alpha & 0 & 0 & 2-2\sqrt2 \ell \alpha & 0 & -\sqrt2 \ell \alpha & 0 & 0 & 0 & 0 & 0 & 0 & 0 & 0 & 0 & 0 & 0 \\
  0 & 0 & 0 & 0 & 0 & 0 & 0 & 0 & 2 & 0 & 0 & 0 & 0 & 0 & 0 & 0 & 0 & 0 & 0 & 0 & 0 \\
  \tfrac{\sqrt2}{\sqrt3} \ell\alpha & 0 & 0 & 0 & \ell\alpha & 0 & 0 & -\sqrt2 \ell \alpha & 0 & 2 & 0 & 0 & 0 & 0 & 0 & 0 & 0 & 0 & 0 & 0 & 0 \\
  0 & \tfrac{3-\sqrt2}{\sqrt3} \ell\alpha & 0 & 0 & i \alpha/\kappa & 0 & 0 & 0 & 0 & 0 & 2-2\ell\alpha & 0 & 0 & -\tfrac{1}{\sqrt{3}} \ell\alpha & 0 & -\tfrac{1}{\sqrt{3}} \ell\alpha & 0 & 0 & 0 & 0 & 0 \\
  0 & 0 & \sqrt2 \ell\alpha & 0 & 0 & 0 & 0 & 0 & 0 & 0 & 0 & 2 & 0 & 0 & 0 & 0 & 0 & 0 & 0 & 0 & 0 \\
  0 & 0 & 0 & \sqrt2 \ell\alpha & 0 & 0 & 0 & 0 & 0 & 0 & 0 & 0 & 2 & 0 & 0 & 0 & 0 & 0 & 0 & 0 & 0 \\
  0 & \tfrac{1+\sqrt3}{2} \ell \alpha & 0 & 0 & 0 & 0 & 0 & 0 & 0 & 0 & -\tfrac{1}{\sqrt{3}} \ell\alpha & 0 & 0 & 2 & 0 & 0 & 0 & 0 & 0 & 0 & 0 \\
  0 & 0 & 0 & 0 & 0 & 0 & 0 & 0 & 0 & 0 & 0 & 0 & 0 & 0 & 2 & 0 & 0 & 0 & 0 & 0 & 0 \\
  0 & \tfrac{\sqrt3-1}{2} \ell \alpha & 0 & 0 & 0 & 0 & 0 & 0 & 0 & 0 & -\tfrac{1}{\sqrt{3}} \ell\alpha & 0 & 0 & 0 & 0 & 2 & 0 & 0 & 0 & 0 & 0 \\
  0 & 0 & 0 & 0 & 0 & 0 & 0 & 0 & 0 & 0 & 0 & 0 & 0 & 0 & 0 & 0 & 2 & 0 & 0 & 0 & 0 \\
  0 & 0 & 0 & 0 & 0 & 0 & 0 & 0 & 0 & 0 & 0 & 0 & 0 & 0 & 0 & 0 & 0 & 2 & 0 & 0 & 0 \\
  0 & 0 & 0 & 0 & 0 & 0 & 0 & 0 & 0 & 0 & 0 & 0 & 0 & 0 & 0 & 0 & 0 & 0 & 2 & 0 & 0 \\
  0 & 0 & 0 & 0 & 0 & 0 & 0 & 0 & 0 & 0 & 0 & 0 & 0 & 0 & 0 & 0 & 0 & 0 & 0 & 2 & 0 \\
  0 & 0 & 0 & 0 & 2\ell\alpha & 0 & 0 & 0 & 0 & 0 & 0 & 0 & 0 & 0 & 0 & 0 & 0 & 0 & 0 & 0 & 2 \\
 \end{pmatrix} }\ . \nonumber \\[3mm]
  \label{Dka5a4aaa}
\end{align}
Let $\delta_j(\kappa,\alpha,\beta,\gamma,\omega,\eta)$ denote the determinant of the upper left $j\times j$ submatrix of $\DD_{\kappa,\alpha,\beta,\gamma,\omega,\eta}$ for integers $j=1,2,\ldots,21$. 
For our choice $\beta =\sqrt3 \alpha$, $\gamma =\alpha$, $\omega=\alpha$ and $\eta=\alpha$,
the minors $\delta_j(\kappa,\alpha) :=\delta_j(\kappa,\alpha,\sqrt3 \alpha,\alpha,\alpha,\alpha)$ are given in Tables~\ref{table:minors:3D:1-14}--\ref{table:minors:3D:15-21}.

\renewcommand{\arraystretch}{1.4}
\begin{table}
\begin{tabular}{|l|l|}
 \hline 
  $\delta_1(\kappa,\alpha)$ &= $2\ell\alpha$ \\
 \hline 
  $\delta_2(\kappa,\alpha)$ &= $4(\sqrt2 -1)\ell^2\alpha^2$ \\
 \hline 
  $\delta_3(\kappa,\alpha)$ &= $8(\sqrt2 -1)\ell^3\alpha^3$ \\
 \hline 
  $\delta_4(\kappa,\alpha)$ &= $16(\sqrt2 -1)\ell^4 \alpha^4$ \\
 \hline 
  $\delta_5(\kappa,\alpha)$ &= $\tfrac{80}{3} (\sqrt2 -1)\ell^5 \alpha^5$ \\  
 \hline 
  $\delta_6(\kappa,\alpha)$ &= $\tfrac{40}{3} (\sqrt2 -1)\ell^4 \alpha^5 p_6(\kappa,\alpha)$ \\
   & with $p_6(\kappa,\alpha):= -4\ell^2 \alpha -\frac{\alpha}{\kappa^2} +4\ell$. \\
 \hline 
  $\delta_7(\kappa,\alpha)$ &= $\tfrac{20}{3} (\sqrt2 -1)\ell^3 \alpha^5 p_6(\kappa,\alpha)^2$ \\
 \hline
  $\delta_8(\kappa,\alpha)$ &= $12 \ell^2\ \alpha^5\ p_6(\kappa,\alpha)^2\ p_8(\kappa,\alpha)$ \\
   & with $p_8(\kappa,\alpha) = \tfrac{2-3\sqrt2}{3} \ell^2 \alpha -\tfrac56 \frac{\alpha}{\kappa^2} +\tfrac{10}{9}(\sqrt2 -1)\ell$. \\
 \hline
  $\delta_9(\kappa,\alpha)$ &= $2\ \delta_8(\kappa,\alpha)$ \\
 \hline
  $\delta_{10}(\kappa,\alpha)$ &= $\tfrac43 \ell^2\ \alpha^5\ p_6(\kappa,\alpha)^2\ p_{10}(\kappa,\alpha)$ \\
   & with $p_{10}(\kappa,\alpha) = 9((\sqrt2 -1)\ell^2 +\frac1{\kappa^2})\ell \alpha^2 -6((8\sqrt2 -6)\ell^2 +\frac5{\kappa^2})\alpha +40(\sqrt2 -1)\ell$. \\
 \hline
  $\delta_{11}(\kappa,\alpha)$ &= $\tfrac29 \ell\ \alpha^5\ p_6(\kappa,\alpha)^2\ p_{11}(\kappa,\alpha)$ \\
   & with $p_{11}(\kappa,\alpha) = \big(p_{11,0}(\alpha) +p_{11,1}(\alpha) \frac1{\kappa^2}\big) \frac1{\kappa^2} + p_{11,2}(\alpha) \ell^2$, \\
   & $p_{11,0}(\alpha) = (54\sqrt2 -144)\ell^3 \alpha^3 +(672 -72\sqrt2)\ell^2 \alpha^2 -(216 +144\sqrt2)\ell\alpha$, \\
   & $p_{11,1}(\alpha) =18(6 -\ell\alpha) \alpha^2$, \\
   & $p_{11,2}(\alpha) =(9 -54\sqrt2)\ell^3 \alpha^3 +(456\sqrt2 -24) \ell^2 \alpha^2 +(472 -816\sqrt2) \ell \alpha + 480(\sqrt2 -1)$. \\
 \hline
  $\delta_{12}(\kappa,\alpha)$ &= $\delta_{11}(\kappa,\alpha) \frac{p_{12}(\kappa,\alpha)}{p_6(\kappa,\alpha)}$
   = $\tfrac29 \ell\ \alpha^5\ p_6(\kappa,\alpha)\ p_{11}(\kappa,\alpha)\ p_{12}(\kappa,\alpha)$ \\
   & with $p_{12}(\kappa,\alpha) = 4\ell^3 \alpha^2 -12\ell^2 \alpha +8\ell -\frac{2 \alpha}{\kappa^2}$. \\
 \hline
  $\delta_{13}(\kappa,\alpha)$ &= $\delta_{12}(\kappa,\alpha) \frac{p_{12}(\kappa,\alpha)}{p_6(\kappa,\alpha)} 
   = \delta_{11}(\kappa,\alpha) \Big(\frac{p_{12}(\kappa,\alpha)}{p_6(\kappa,\alpha)}\Big)^2
   = \tfrac29 \ell\ \alpha^5\ p_{11}(\kappa,\alpha) p_{12}(\kappa,\alpha)^2$ \\
 \hline
  $\delta_{14}(\kappa,\alpha)$ &= $\tfrac1{9\ (1+\sqrt3)^2} \ell\ \alpha^5\ p_{12}(\kappa,\alpha)^2\ p_{14}(\kappa,\alpha)$ \\
   & with $p_{14}(\kappa,\alpha) = \big(p_{14,0}(\alpha) +p_{14,1}(\alpha) \frac1{\kappa^2}\big) \frac1{\kappa^2} +\ell^2 p_{14,2}(\alpha)$, \\
   & $p_{14,0}(\alpha) =
   (-108\sqrt6 -72\sqrt3 -180\sqrt2 -144) \ell^4 \alpha^4
   +(360\sqrt6 -1824\sqrt3 +720\sqrt2 -3396) \ell^3 \alpha^3$ \\
   &\quad $+(-576\sqrt6 +5952\sqrt3 -1152\sqrt2 +11760) \ell^2 \alpha^2
   +(-1152\sqrt6 -1728\sqrt3 -2304\sqrt2 -3456) \ell \alpha$, \\
   & $p_{14,1}(\alpha) = 144\ (\sqrt3 +2)\ (6 -\ell \alpha)\ \alpha^2$, \\
   & $p_{14,2}(\alpha) = (1440 -180\sqrt6 +828\sqrt3 -324\sqrt2) \ell^4 \alpha^4
   -(9348 -336\sqrt6 -5400\sqrt3 -624\sqrt2) \ell^3 \alpha^3$ \\
   &\quad $+(11056 +3424\sqrt6 +6368\sqrt3 +6864\sqrt2) \ell^2 \alpha^2
   +(4192 -6528\sqrt6 +1856\sqrt3 -13056\sqrt2) \ell \alpha$ \\
   &\quad $+(3840\sqrt6 -3840\sqrt3 +7680\sqrt2 -7680)$. \\
  \hline
\end{tabular}
\caption{Let $\delta_j(\kappa,\alpha,\beta,\gamma,\omega,\eta)$ denote the determinant of the upper left $j\times j$ submatrix of $\DD_{\kappa,\alpha,\beta,\gamma,\omega,\eta}$ for integers $j=1,2,\ldots,21$. 
For our choice $\beta =\sqrt3 \alpha$, $\gamma =\alpha$, $\omega=\alpha$ and $\eta=\alpha$, the minors $\delta_j(\kappa,\alpha)=\delta_j(\kappa,\alpha,\sqrt3 \alpha,\alpha,\alpha,\alpha)$ for integers $j=1,2,\ldots,14$, are given in this table.}
\label{table:minors:3D:1-14}
\end{table}
The first five minors are positive if $\alpha$ is positive.

The sixth minor $\delta_6(\kappa,\alpha)$ 
 satisfies for positive $\alpha$ the inequality $\delta_6(\kappa,\alpha)\geq \delta_6(1,\alpha)$ for all $\kappa\in K$.
Moreover, the factor $p_6(1,\alpha)$ is positive for $\alpha\in(0,\alpha_{p_6})$ with $\alpha_{p_6} :={4\ell}/{(4\ell^2 +1)}$.
Thus the sixth minor $\delta_6(\kappa,\alpha)$ is positive for all $\kappa\in K$ if $0<\alpha<\alpha_{\delta_6}$ with $\alpha_{\delta_6} :=\alpha_{p_6}$.

Following from the analysis of factor $p_6(\kappa,\alpha)$,
 the seventh minor $\delta_7(\kappa,\alpha)$ is positive for all $\kappa\in K$ if $0<\alpha<\alpha_{\delta_7} :=\alpha_{\delta_6}$.

The eighth minor $\delta_8(\kappa,\alpha)$ has a factorization.
For positive $\alpha$ and $\kappa\in K$, the inequalities $p_8(\kappa,\alpha) \geq p_8(1,\alpha)$ and consequently $\delta_8(\kappa,\alpha)\geq \delta_8(1,\alpha)$ hold. 
Moreover, the linear polynomial $p_8(1,\alpha)$ is positive for 
 $0<\alpha<\alpha_{p_8}$ with $\alpha_{p_8} := {20}{(\sqrt2 -1)\ell}/{3 ((6\sqrt2 -4)\ell^2 +5)}$.
Thus, for all $\kappa\in K$, the eighth minor $\delta_8(\kappa,\alpha)$ is positive if $0<\alpha<\alpha_{\delta_8}$ with $\alpha_{\delta_8} :=\min \big\{\alpha_{p_6}\,,\ \alpha_{p_8}\} =\alpha_{p_8}$.

The ninth minor satisfies $\delta_9(\kappa,\alpha) = 2\ \delta_8(\kappa,\alpha)$,
hence, it is positive for all $0<\alpha<\alpha_{\delta_9}:=\alpha_{\delta_8}$ and $\kappa\in K$.

The tenth minor~$\delta_{10}(\kappa,\alpha)$ has a factorization.
The factor $p_{10}(\kappa,\alpha)$ has the $\kappa$-dependent summand $(9\ell\alpha -30)\alpha /\kappa^2$, which is negative for $0<\alpha<{10}/{3\ell}$.
Under this assumption, the inequalities $p_{10}(\kappa,\alpha) \geq p_{10}(1,\alpha)$ and $\delta_{10}(\kappa,\alpha)\geq \delta_{10}(1,\alpha)$ hold for all $\kappa\in K$. 
The quadratic polynomial $p_{10}(1,\alpha)$ has two positive roots $0<\alpha^{(10)}_-<\alpha^{(10)}_+$ and is positive if $\alpha<\alpha_{p_{10}}$ with $\alpha_{p_{10}} :=\alpha^{(10)}_-$.
Thus, the tenth minor $\delta_{10}(\kappa,\alpha)$ is positive for all $\kappa\in K$, if $0<\alpha<\alpha_{\delta_{10}}$ with $\alpha_{\delta_{10}}:=\min\big\{\alpha_{p_6}\,,\ {10}/{3\ell}\,,\ \alpha_{p_{10}}\} =\min\big\{\alpha_{p_6}\,,\ \alpha_{p_{10}}\}$.

The eleventh minor $\delta_{11}(\kappa,\alpha)$ has a factorization.
Due to Lemma~\ref{lem:p},
the inequality $p_{11}(\kappa,\alpha) \geq p_{11}(1,\alpha)$ holds for some positive $\widetilde\alpha_{p_{11}}$, 
 and consequently $\delta_{11}(\kappa,\alpha)\geq \delta_{11}(1,\alpha)$ holds for all $0<\alpha<\widetilde\alpha_{p_{11}}$ and $\kappa\in K$. 
The polynomial $p_{11}(1,\alpha)$ is positive at $\alpha=0$,
 hence there exists a positive number $\alpha_{p_{11}}$ such that $p_{11}(1,\alpha)$ is positive for $0<\alpha <\alpha_{p_{11}}$ and all $\kappa\in K$.
Consequently, for $0<\alpha <\alpha_{\delta_{11}}$ with $\alpha_{\delta_{11}} :=\min\{\alpha_{p_{11}}\,,\ \widetilde\alpha_{p_{11}}\,,\ \alpha_{p_6}\}$ the eleventh minor $\delta_{11}(\kappa,\alpha)$ is positive for all $\kappa\in K$.

The twelfth minor $\delta_{12}(\kappa,\alpha)$ has a factorization.
For positive $\alpha$, the inequalities $p_{12}(\kappa,\alpha) \geq p_{12}(1,\alpha)$ and $\delta_{12}(\kappa,\alpha)\geq \delta_{12}(1,\alpha)$ hold
 for all $0<\alpha<\min\{\alpha_{\delta_{11}}\,,\ \alpha_{p_6}\} =\alpha_{\delta_{11}}$ and $\kappa\in K$. 
The quadratic polynomial $p_{12}(1,\alpha)$ has two positive roots $0<\alpha^{(12)}_- \leq\alpha^{(12)}_+$ and is positive for $0<\alpha<\alpha_{p_{12}}$ with $\alpha_{p_{12}}:= \alpha^{(12)}_-$.
Thus, the twelfth minor $\delta_{12}(\kappa,\alpha)$ is positive for all $\kappa\in K$, if $0<\alpha<\alpha_{\delta_{12}}$ with $\alpha_{\delta_{12}} :=\min\{\alpha_{p_6}\,,\ \alpha_{p_{12}}\,,\ \alpha_{\delta_{11}}\} =\min\{\alpha_{p_{12}}\,,\ \alpha_{\delta_{11}}\}$.

The thirteenth minor satisfies $\delta_{13}(\kappa,\alpha) =\tfrac29 \ell\ \alpha^5\ p_{11}(\kappa,\alpha) p_{12}(\kappa,\alpha)^2$. 
Therefore the thirteenth minor $\delta_{13}(\kappa,\alpha)$ is positive for all $\kappa\in K$ if $0<\alpha<\alpha_{\delta_{13}}$ with $\alpha_{\delta_{13}}:=\min\{\alpha_{p_{11}}\,,\ \alpha_{p_{12}} \}$.

The fourteenth minor $\delta_{14}(\kappa,\alpha)$ has a factorization.
The polynomial $p_{14,1}(\alpha)$ is positive if $0<\alpha<6/\ell$.
Moreover, the quartic polynomial $p_{14,0}(\alpha) +2 p_{14,1}(\alpha)$ is zero at $\alpha=0$,
 having a negative derivative at $\alpha=0$.
Thus there exists a positive number $\alpha^{(14,0)}$ such that 
 $p_{14,0}(\alpha) +2 p_{14,1}(\alpha)$ is negative for $0<\alpha<\alpha^{(14,0)}$.
Due to Lemma~\ref{lem:p},
 the inequality $p_{14}(\kappa,\alpha) \geq p_{14}(1,\alpha)$ holds for $0\leq \alpha\leq \widetilde\alpha_{p_{14}}:=\min\{6/\ell\,,\ \alpha^{(14,0)} \}$, 
 and consequently $\delta_{14}(\kappa,\alpha)\geq \delta_{14}(1,\alpha)$ holds for all $0\leq \alpha\leq \widetilde\alpha_{p_{14}}$ and $\kappa\in K$. 
The polynomial $p_{14}(1,\alpha)$ is positive at $\alpha=0$,
 hence there exists a positive number $\alpha_{p_{14}}$ such that $p_{14}(\kappa,\alpha)$ is positive for $0<\alpha <\alpha_{p_{14}}$ and all $\kappa\in K$.
Consequently, for $0<\alpha<\alpha_{\delta_{14}}$ with $\alpha_{\delta_{14}} :=\min\{\alpha_{p_{12}}\,,\ \widetilde\alpha_{p_{14}}\,,\ \alpha_{p_{14}}\}$ the fourteenth minor $\delta_{14}(\kappa,\alpha)$ is positive for all $\kappa\in K$.

\renewcommand{\arraystretch}{1.4}
\begin{table}
\begin{tabular}{|l|l|}
 \hline 
 $\delta_{15}(\kappa,\alpha)$ &= $2\ \delta_{14}(\kappa,\alpha)$ \\
 \hline 
 $\delta_{16}(\kappa,\alpha)$ &= $\tfrac89 \tfrac{2+\sqrt3}{(1+\sqrt3)^2} \ell\ \alpha^5\ p_{12}(\kappa,\alpha)^2\ p_{16}(\kappa,\alpha)$ \\
  & with $p_{16}(\kappa,\alpha) = \big(p_{16,0}(\alpha) +p_{16,1}(\alpha) \frac1{\kappa^2}\big) \frac1{\kappa^2} +\ell^2 p_{16,2}(\alpha)$, \\
  & $p_{16,0}(\alpha) =
   -36 (\sqrt2 +2) \ell^4 \alpha^4
   +(144 \sqrt2 -744) \ell^3 \alpha^3$ \\
  &\quad $+(-288 \sqrt2 +2976) \ell^2 \alpha^2
   +(-576 \sqrt2 -864) \ell \alpha$, \\
  & $p_{16,1}(\alpha) = 72 (6 -\alpha \ell) \alpha^2$, \\
  & $p_{16,2}(\alpha) =
   27 \ell^5 \alpha^5
   +(-144 \sqrt2 +216) \ell^4 \alpha^4
   +(-24 \sqrt2 -2412) \ell^3 \alpha^3$ \\
  &\quad $+(1632 \sqrt2 +3104) \ell^2 \alpha^2
   +(-3264 \sqrt2 +928) \ell \alpha
   +1920 (\sqrt2 -1)$. \\
 \hline 
 $\delta_{17}(\kappa,\alpha)$ &= $2\ \delta_{16}(\kappa,\alpha)$ \\
 \hline 
 $\delta_{18}(\kappa,\alpha)$ &= $2^2\ \delta_{16}(\kappa,\alpha)$ \\
 \hline 
 $\delta_{19}(\kappa,\alpha)$ &= $2^3\ \delta_{16}(\kappa,\alpha)$ \\
 \hline 
 $\delta_{20}(\kappa,\alpha)$ &= $2^4\ \delta_{16}(\kappa,\alpha)$ \\
 \hline 
 $\delta_{21}(\kappa,\alpha)$ &= $\tfrac{256 (\sqrt3 +2) (24\sqrt2 +61)}{23121 (\sqrt3 +1)^2} \ell\ \alpha^5\ p_{12}(\kappa,\alpha)^2\ p_{21}(\kappa,\alpha)$ \\
  & with $p_{21}(\kappa,\alpha) = \Big(p_{21,0}(\alpha) +p_{21,1}(\alpha) \frac1{\kappa^2}\Big) \frac1{\kappa^2} +\ell^2 p_{21,2}(\alpha)$, \\
  & $p_{21,0}(\alpha) =
   (-1152 \sqrt2 +2928) \ell^5 \alpha^5
   +(-468 \sqrt2 -2664) \ell^4 \alpha^4
   +(75024 \sqrt2 -175272) \ell^3 \alpha^3$ \\
  &\quad $+(-130464 \sqrt2 +300768) \ell^2 \alpha^2
   +(-14400 \sqrt2 -25056) \ell \alpha$, \\
  & $p_{21,1}(\alpha) = (-1728 \sqrt2 +4392) (6 -\ell \alpha) \alpha^2$, \\
  & $p_{21,2}(\alpha) =
   7707 \ell^5 \alpha^5
   +(-25248 \sqrt2 +95000) \ell^4 \alpha^4
   +(89448 \sqrt2 -353228) \ell^3 \alpha^3$ \\
   &\quad $+(158880 \sqrt2 +38048) \ell^2 \alpha^2
   +(-417216 \sqrt2 +464416) \ell \alpha 
   +1920 (85\sqrt2 -109)$. \\ 
 \hline 
\end{tabular}
\caption{Let $\delta_j(\kappa,\alpha,\beta,\gamma,\omega,\eta)$ denote the determinant of the upper left $j\times j$ submatrix of $\DD_{\kappa,\alpha,\beta,\gamma,\omega,\eta}$ for integers $j=15,\ldots,21$. 
For our choice $\beta =\sqrt3 \alpha$, $\gamma =\alpha$, $\omega=\alpha$ and $\eta=\alpha$, the minors $\delta_j(\kappa,\alpha)=\delta_j(\kappa,\alpha,\sqrt3 \alpha,\alpha,\alpha,\alpha)$ are given in this table.}
\label{table:minors:3D:15-21}
\end{table}
The fifteenth minor $\delta_{15}(\kappa,\alpha) = 2\ \delta_{14}(\kappa,\alpha)$
 is positive for all $\kappa\in K$ if $0<\alpha< \alpha_{\delta_{15}} :=\alpha_{\delta_{14}}$.

The sixteenth minor $\delta_{16}(\kappa,\alpha)$ has a factorization.
The polynomial $p_{16,1}(\alpha)$ is positive if $0<\alpha<6/\ell$.
Under this assumption, the quartic polynomial~$p_{16,0}(\alpha) +2 p_{16,1}(\alpha)$ is zero at $\alpha=0$,
 having a negative derivative at $\alpha=0$.
Thus there exists a positive number $\alpha^{(16,0)}$ such that 
 $p_{16,0}(\alpha) +2 p_{16,1}(\alpha)$ is negative for $0<\alpha<\alpha^{(16,0)}$.
Due to Lemma~\ref{lem:p},
 the inequality $p_{16}(\kappa,\alpha) \geq p_{16}(1,\alpha)$ holds for all $0\leq \alpha\leq\widetilde\alpha_{p_{16}}:=\min\{6/\ell\,,\ \alpha^{(16,0)} \}$,
 and $\delta_{16}(\kappa,\alpha)\geq \delta_{16}(1,\alpha)$ holds for all $0\leq \alpha\leq\widetilde\alpha_{p_{16}}$ and $\kappa\in K$. 
The polynomial $p_{16}(1,\alpha)$ is positive at $\alpha=0$,
 hence there exists a positive number $\alpha_{p_{16}}$ such that $p_{16}(\kappa,\alpha)$ is positive for $0<\alpha <\alpha_{p_{16}}$ and all $\kappa\in K$.
Consequently, for $0<\alpha< \alpha_{\delta_{16}}$ with $\alpha_{\delta_{16}}:=\min\{\alpha_{p_{16}}\,,\ \widetilde\alpha_{p_{16}}\,,\ \alpha_{p_{12}} \}$ the sixteenth minor $\delta_{16}(\kappa,\alpha)$ is positive for all $\kappa\in K$.

The seventeenth to twentieth minors are multiples of the sixteenth minor. 
Therefore, these minors are positive for all $\kappa\in K$ under the same condition $0<\alpha<\alpha_{\delta_{16}}$.

The twenty-first minor~$\delta_{21}(\kappa,\alpha)$ has a factorization.
The polynomial $p_{21,1}(\alpha)$ is positive if $0<\alpha<6/\ell$.
The quintic polynomial~$p_{21,0}(\alpha) +2 p_{21,1}(\alpha)$ is zero at $\alpha=0$,
 having a negative derivative at $\alpha=0$.
Thus there exists a positive number $\alpha^{(21,0)}$ such that 
 $p_{21,0}(\alpha) +2 p_{21,1}(\alpha)$ is negative for $0<\alpha<\alpha^{(21,0)}$.
Due to Lemma~\ref{lem:p},  
 the inequality $p_{21}(\kappa,\alpha) \geq p_{21}(1,\alpha)$ holds for $0\leq\alpha\leq \widetilde\alpha_{p_{21}}:=\min\{6/\ell\,,\ \alpha^{(21,0)} \}$,
 and $\delta_{21}(\kappa,\alpha)\geq \delta_{21}(1,\alpha)$ holds
 for all $0\leq\alpha\leq\widetilde\alpha_{p_{21}}$ and $\kappa\in K$. 
The polynomial $p_{21}(1,\alpha)$ is positive at $\alpha=0$, 
 hence there exists a positive number $\alpha_{p_{21}}$ such that $p_{21}(1,\alpha)$ is positive for $0<\alpha <\alpha_{p_{21}}$.
Consequently, for $0<\alpha< \alpha_{\delta_{21}}$ with $\alpha_{\delta_{21}}=\min\{\alpha_{p_{21}}\,,\ \widetilde\alpha_{p_{21}}\,,\ \alpha_{p_{12}} \}$ the twenty-first minor $\delta_{21}(\kappa,\alpha)$ is positive for all $\kappa\in K$.

\medskip
Let $\{\lambda_1,\lambda_2,\ldots,\lambda_{21}\}$ be the eigenvalues of $\DD_{\kappa,\alpha,\beta,\gamma,\omega,\eta}$ arranged in increasing order.   
We seek a lower bound on $\lambda_1$. 
As long as $\DD_{\kappa,\alpha,\beta,\gamma,\omega,\eta}$ is positive definite, 
 the arithmetic-geometric mean inequality implies 
 \begin{align*}
 \lambda_1(\kappa,\alpha,\beta,\gamma,\omega,\eta) &= \frac{\delta_{21}(\kappa,\alpha,\beta,\gamma,\omega,\eta)}{\prod_{j=2}^{21} \lambda_j}
   \geq \delta_{21}(\kappa,\alpha,\beta,\gamma,\omega,\eta) \left( \frac{20}{\sum_{j=2}^{21} \lambda_j} \right)^{20}\\
   &\geq \delta_{21}(\kappa,\alpha,\beta,\gamma,\omega,\eta) \left( \frac{20}{\tr[\DD_{\kappa,\alpha,\beta,\gamma,\omega,\eta}]} \right)^{20} 
    = \left( \frac{20}{32} \right)^{20} \delta_{21}(\kappa,\alpha,\beta,\gamma,\omega,\eta) \ ,
 \end{align*}
 since $\tr[\DD_{\kappa,\alpha,\beta,\gamma,\omega,\eta}] =32$ independently of $\kappa$, $\alpha$, $\beta$, $\gamma$, $\omega$ and $\eta$.
 
A simple computation shows that the eigenvalues of $\P_\kappa$ are $1$, $1 \pm \alpha/\kappa$ (3-fold), and $1 \pm2\alpha/\kappa$. 
Hence for positive $\alpha$ 
 \begin{equation}\label{Pbound3}
 (1-2\alpha ) {\bf I}  \leq \P_\kappa \leq (1+2\alpha) {\bf I}
 \end{equation}
uniformly in $\kappa$.
Thus, all matrices $\P_\kappa$ are positive definite, if $0<\alpha<1/2$.
Finally, if $\P_\kappa$ is chosen with  
 \begin{equation*} 
   \alpha \in (0,\alpha_+)\,, \qquad \text{where }
   \alpha_+ := \min\{ 1/2\,,\ \alpha_{\delta_6}\,,\ \alpha_{\delta_7}\,,\ \ldots\,,\ \alpha_{\delta_{21}} \}\,,
 \end{equation*}
$\beta =\sqrt3  \alpha$, $\gamma =\alpha$, $\omega=\alpha$, and $\eta=\alpha$ uniformly for all $\kappa\in K$, then
 \begin{equation}\label{Pbnd3}
 \CC_\kappa^*\P_\kappa + \P_\kappa\CC_\kappa \geq \left( \frac{20}{32} \right)^{20} \delta_{21}(1,\alpha) \ {\bf I} \qquad \text{ uniformly in } \kappa\in K \ . 
 \end{equation}
Combining \eqref{Pbound3} with \eqref{Pbnd3} yields the result.
\end{proof}

\bigskip
\bigskip

\textbf{Acknowledgement:}
The first author (FA) was supported by the FWF-funded SFB \# F65.
The second author (AA) was partially supported by the FWF-doctoral school ``Dissipation and dispersion in non-linear partial differential equations'' and the FWF-funded SFB \# F65.
The third author (EC) was partially supported by U.S. N.S.F. grant DMS 1501007.


\end{document}